\documentclass[18pt]{article}

\usepackage[margin=1in]{geometry}

\title{Long-term accuracy of numerical approximations of SPDEs with the stochastic Navier-Stokes equations as a paradigm}
\author{Nathan E. Glatt-Holtz, Cecilia F. Mondaini \\
  \scriptsize{emails: negh@tulane.edu, cf823@drexel.edu}
  }

\usepackage{amsfonts, amssymb, amsmath, amsthm, multicol}
\usepackage[colorlinks=true, pdfstartview=FitV, linkcolor=blue,
            citecolor=blue, urlcolor=blue]{hyperref}
\usepackage[usenames]{color}
\definecolor{Red}{rgb}{0.7,0,0.1}
\definecolor{Green}{rgb}{0,0.7,0}
\usepackage{accents}
\usepackage{comment}
\usepackage{float}
\usepackage{graphicx}
\usepackage{multirow}
\usepackage{parcolumns}
\usepackage{subcaption}
\usepackage{textcomp}
\usepackage{mathrsfs}
\usepackage{mathtools}
\usepackage{dsfont}
\usepackage{nccmath}
\usepackage{enumitem}
\usepackage[capitalize,nameinlink,noabbrev]{cleveref}

\usepackage{cjhebrew}

\numberwithin{equation}{section}

\newtheorem{Thm}{Theorem}[section]
\newtheorem{Lem}[Thm]{Lemma}
\newtheorem{Prop}[Thm]{Proposition}
\newtheorem{Cor}[Thm]{Corollary}

\newtheorem{Rmk}[Thm]{Remark}

\newtheorem*{Thm*}{Theorem}

\newcommand{\bu}{\mathbf{u}}

\newcommand{\bv}{\mathbf{v}}
\newcommand{\bw}{\mathbf{w}}
\newcommand{\bx}{\mathbf{x}}
\newcommand{\rd}{{\text{\rm d}}}
\newcommand{\mA}{\mathcal{A}}
\newcommand{\mB}{\mathcal{B}}
\newcommand{\mC}{\mathcal{C}}
\newcommand{\mF}{\mathcal{F}}
\newcommand{\mI}{\mathcal{I}}

\newcommand{\mK}{\mathcal{K}}
\newcommand{\mL}{\mathcal{L}}
\newcommand{\mM}{\mathcal{M}}
\newcommand{\mN}{\mathcal{N}}
\newcommand{\cO}{\mathcal{O}}
\newcommand{\mP}{\mathcal{P}}

\newcommand{\mS}{\mathcal{S}}

\newcommand{\mW}{\mathcal{W}}
\newcommand{\mX}{\mathcal{X}}
\newcommand{\mY}{\mathcal{Y}}
\newcommand{\tC}{\widetilde{C}}

\newcommand{\bP}{\mathbb{P}}
\newcommand{\bE}{\mathbb{E}}
\newcommand{\bT}{\mathbb{T}}
\newcommand{\dL}{\dot{L}}
\newcommand{\dH}{\dot{H}}
\newcommand{\bdH}{\mathbf{\dot{H}}}
\newcommand{\bdL}{\mathbf{\dot{L}}}
\newcommand{\dt}{\delta}

\newcommand{\xn}{\xi_N}
\newcommand{\xfd}{\xi_{N,\delta}} 

\newcommand{\txfd}{\widetilde\xi_{N,\delta}}
\newcommand{\txi}{\widetilde{\xi}}
\newcommand{\tu}{\widetilde{\mathbf{u}}}

\newcommand{\ufd}{\bu_{N,\delta}}

\newcommand{\tufd}{\widetilde \bu_{N,\delta}}

\newcommand{\tr}{\text{Tr}}

\newcommand{\rhoes}{\rho_{\varepsilon,s}}
\newcommand{\rhoesa}{\rho_{\varepsilon,s,\alpha}}
\newcommand{\Wes}{\mW_{\varepsilon,s}}
\newcommand{\Wesa}{\mW_{\varepsilon,s,\alpha}}
\newcommand{\Wesap}{\mW_{\varepsilon,s,\alpha'}}

\newcommand{\hW}{\widehat{W}}
\newcommand{\ind}{\mathds{1}}

\newcommand{\sm}{\alpha} 
\newcommand{\smx}{p}

\newcommand{\ra}{b}

\newcommand{\HN}{\Pi_N \dL^2}
\newcommand{\HK}{\Pi_K \dL^2}
\newcommand{\zn}{\zeta_N}
\newcommand{\zfd}{\zeta_{N,\delta}}
\newcommand{\QN}{Q_N}
\newcommand{\RR}{\mathbb{R}}  
\newcommand{\NN}{\mathbb{N}}
\newcommand{\ZZ}{\mathbb{Z}}
\newcommand{\z}{\zeta}
\newcommand{\tM}{\widetilde{M}}
\newcommand{\Pdisc}{P_n^{N, \dt}}
\newcommand{\Pdisct}{\mathcal{P}_t^{N, \dt}}
\newcommand{\tPdisc}{\widetilde{P}_n^{N,\dt, \xi_0}}
\newcommand{\Pcont}{P_t}

\newcommand{\Wass}{\mathcal{W}}
\newcommand{\Co}{\mathcal{C}} 

\newcommand{\KL}[2]{D_{\text{KL}}( #1 \| #2 )}
\newcommand{\tv}[1]{\left\| #1 \right\|_{\text{TV}}}

\newcommand{\range}[1]{range(#1)}
\newcommand{\cinv}{\mu_{\ast}}
\newcommand{\dinv}{\mu_{\ast}^{N, \dt}}

\newcommand{\tp}{{p'}} 

\newcommand{\gsp}{X}
\newcommand{\ggsp}{X}
\newcommand{\gP}{P}
\newcommand{\gdist}{\rho}
\newcommand{\gdista}{\rho_\sm}

\newcommand{\tT}{\widetilde{T}}
\newcommand{\LC}{C_0}
\newcommand{\pmr}{\theta} 
\newcommand{\gf}{f}
\newcommand{\gn}[1]{\| #1 \|} 
\newcommand{\gfd}{\Lambda} 
\newcommand{\gspmr}{\Theta}
\newcommand{\tsm}{\tilde{\sm}}
\newcommand{\pfunc}{g}
\newcommand{\tgf}{R} 
\newcommand{\tmu}{\tilde{\mu}}

\newcommand{\cD}{\mathcal{D}}
\newcommand{\f}{\mathbf{f}}
\newcommand{\spH}{H}
\newcommand{\spV}{V}
\newcommand{\ipV}[2]{ (\!( #1, #2)\!) }
\newcommand{\PL}{P_L}
\newcommand{\bH}{\mathbf{H}}
\newcommand{\btu}{\tilde{\mathbf{u}}}
\newcommand{\tP}{\widetilde{P}}

\begin{document}
\markboth{}{}

\maketitle

\begin{abstract}

This work introduces a general framework for establishing the long time accuracy for approximations 
of Markovian dynamical systems on separable Banach spaces. Our results illuminate the role that a certain uniformity in Wasserstein contraction rates for the approximating dynamics bears on long time accuracy estimates.  In particular, our approach yields weak consistency bounds on $\RR^+$ while providing a means to sidestepping a commonly occurring situation where certain higher order moment bounds are unavailable for the approximating dynamics. Additionally, to facilitate the analytical core of our approach, we develop a refinement of certain `weak Harris theorems'. This extension expands the scope of applicability of such Wasserstein contraction estimates to a variety of interesting SPDE examples involving weaker dissipation or stronger nonlinearity than would be covered by the existing literature.

As a guiding and paradigmatic example, we apply our formalism to the stochastic 2D Navier-Stokes equations and to a semi-implicit in time and spectral Galerkin in space numerical approximation of this system. In the case of a numerical approximation,
we establish quantitative estimates on the approximation of invariant measures as well as prove weak consistency on $\RR^+$. To develop these numerical analysis results, we provide a refinement of $L^2_x$ accuracy bounds in comparison to the existing literature which are results of independent interest.  
\end{abstract}

{\noindent \small
   }

\begin{footnotesize}

\setcounter{tocdepth}{2}
\tableofcontents
\end{footnotesize}

\newpage

\section{Introduction}\label{sec:intro}

Questions concerning long time accuracy under approximations for dynamical systems
exhibiting chaotic behavior are notoriously difficult. This is nevertheless a topic of wide interest particularly 
given that statistical theories of turbulence in fluid dynamics can be framed in terms of observables against invariant measures. According to this widely used paradigm such measures are connected to the fundamental governing equations through (putative) ergodic averages and thus may be regarded as containers for statistically robust properties of turbulent flows. Thus, from this point of view, it is natural to ask if the essential features of these invariant measures are maintained under suitable numerical approximations or in a variety of physically interesting singular parameter limits.

Unfortunately, the robustness of statistical properties, i.e. the verification of an ergodic hypothesis, for solutions of deterministic models such as the Navier-Stokes equations and its many variations are typically far from the reach of rigorous analysis. On the other hand, certain stochastic versions of these equations are more tractable to analyze in this regard. Moreover, such stochastic models often retain physical relevance 
while providing an important motivation and a set of unique challenges that have been driving a flurry of developments in the ergodic and mixing theory of infinite dimensional Markov processes in recent decades. In this stochastic setting, the question of the stability of long time statistical properties as a function of
model parameters is therefore of broad interest for a diverse variety of nonlinear, infinite dimensional, randomly stirred systems.

This work develops a novel framework for addressing such long time
stability and accuracy questions for parameter dependent Markov
processes on a Polish space. Our approach leverages a certain uniform
Wasserstein contraction condition (a strong form of exponential
mixing) which, as we will illustrate on several paradigmatic examples,
has a rather broad scope of applicability for finite and infinite
dimensional stochastic systems. Our results demonstrate that
appropriately leveraging uniform contraction provides an important
twist on an existing vein of research concerning infinite time
stability under parameter perturbation for certain stochastic systems
\cite{shardlow2000perturbation,kuksin2003some,HairerMattingly2008,MattinglyStuartTretyakov,hairer2010simple,HairerMattinglyScheutzow2011,johndrow2017error,foldes2017asymptotic,foldes2019large,cerrai2020convergence}. Here
we also note that the framework in
\cite{wang2010approximation,gottlieb2012long} for deterministic
dynamical systems anticipate some of the developments here, including
the invocation of a uniform dissipativity condition. However, the
scope of \cite{wang2010approximation, gottlieb2012long} is
fundamentally limited by its inability to rule out non-uniqueness (let
alone address ergodic and mixing properties) for the long term
statistics of the infinite dimensional deterministic models considered
therein.

As an important technical foundation to carry out our broad program, we develop a refinement of the so called `weak Harris approach' to exponential mixing. This portion of our contribution builds on the seminal works \cite{HairerMattingly2008,HairerMattinglyScheutzow2011}, which lay out a powerful
framework for addressing Wasserstein contraction. These earlier works make use of delicate norm constructions which sidestep the need to Byzantine explicit coupling constructions. On the other hand, the representative and natural selection of examples presented in \cite{GHRichMa2017,ButkovskyKulikScheutzow2019}, demonstrates the need to refine the approach for the typical situation where models lack certain higher order moment estimates or possess a weaker form of smoothing at small scales or both.

A primary domain of application for our framework regards the error analysis for numerical approximations of certain stochastic partial differential equations (SPDEs). This is an area of applied analysis that has undergone some rapid development in the past decade; see for example \cite{mattingly2002ergodicity,jentzen2009numerical,CarelliProhl2012,brzezniak2013finite,BessaihMillet2019,BessaihMillet2021} and containing references. Thus, to illustrate the scope of our approach on a paradigmatic example, we carry out a case study of the space-time numerical approximation of the stochastic 2D Navier-Stokes equations given by a spectral Galerkin discretization in space and a semi-implicit Euler time discretization. In the course of our analysis, we provide some novel approximation bounds and some significant refinements of existing finite time error bounds in comparison to the existing literature 
\cite{CarelliProhl2012,BessaihMillet2019,BessaihMillet2021} which are of independent interest. Note that our general framework has also been useful for several concurrent projects. In a recent work by the first author, \cite{GlattHoltzNguyen2021}, we make use of uniform contractivity to address certain singular limit problems 
concerning SPDEs with diffusive memory terms. Elsewhere in \cite{GHHKMBias2022}, we address 
application in bias estimation for statistical sampling algorithms.

\subsection{The Uniform Contraction Framework for Long Time Stability}\label{subsec:unif:contr:overview}

Let us now give an overview of the abstract foundation of our approach. We provide 
an idealized version here so that the reader can observe the underlying simplicity of our
framework. Of course, to carry out our program in practice we will need to impose a number of 
technical assumptions; we refer the reader to \cref{thm:gen:wk:conv},
\cref{thm:gen:wk:conv:2}, and \cref{cor:gen:wk:conv:obs} below for these more involved formulations.

Our departure point is to observe that, for certain stochastic Markovian systems, 
long time accuracy estimates can be developed in the presence of a strong type of mixing taking the form of a contraction estimate in a suitable Wasserstein distance. Note that such contraction estimates have been previously exploited in a variety of specific contexts for SDEs and SPDEs
and other Markovian processes \cite{HairerMattingly2008,HairerMattinglyScheutzow2011, hairer2010simple, johndrow2017error,foldes2017asymptotic,foldes2019large,cerrai2020convergence} to provide rigorous bounds on parameter dependent invariant measures. 
The crucial new element here centers on suitably exploiting parameter independent uniformity in the contraction rates.

Suppose that $\{\gP_t^\pmr\}_{t \geq 0}$ is a collection of Markov transition operators defined
on a Polish space $(\gsp,\rho)$ parameterized by $\pmr \in \Theta$. These operators act on probability measures $\nu$ and observables $\varphi$ as
\begin{align*}
	\nu \gP_t^\pmr (du) := \int\gP_t^\pmr (v, du) \nu(dv), \qquad
	\gP_t^\pmr \varphi (u) :=  \int \varphi(v)\gP_t^\pmr(u, dv),
\end{align*}
respectively. Let us suppose that for some $\theta_0$, corresponding to the `true' or `limiting' dynamics of interest, we have Wasserstein contraction. Namely, for any $t \geq 0$
\begin{align}
	\Wass(\mu\gP_t^{\pmr_0}, \tilde{\mu} \gP_t^{\pmr_0}) \leq C_0 e^{ - \kappa t} \Wass(\mu, \tilde{\mu}),
	\label{eq:wass:cont:intro}
\end{align}
for any Borel probability measure $\mu, \tilde{\mu}$, where $C_0, \kappa >0$ are constants independent of 
$\mu, \tilde{\mu}$ and $t \geq 0$. Here, as in e.g. \cite{villani2008}, $\Wass$ is the Wasserstein
distance corresponding to $\rho$, i.e.
\begin{align}
 	\Wass(\nu_1, \nu_2) = \inf_{\Gamma \in \Co(\nu_1, \nu_2)} \int \rho(u, \tilde{u}) \Gamma(du, d\tilde{u}),
	\label{eq:wass:dist:intro}
\end{align}
with  $\Co(\nu_1, \nu_2)$ denoting all of the couplings of $\nu_1$ and $\nu_2$. Note that such Wasserstein 
contraction estimates can be obtained using the so called `weak Harris approach' developed 
in \cite{HairerMattingly2008,HairerMattinglyScheutzow2011}, which we refine for our purposes 
here in \cref{thm:gen:sp:gap} below.

Suppose now that for every $\pmr \in \Theta$ we have a corresponding
measure $\mu_\pmr$ which is invariant under $\{\gP_t^\pmr \}_{t \geq 0}$,
namely $\mu_\pmr \gP_t^\pmr = \mu_\pmr$ for any $t \geq 0$. We then
make the following simple observation. Exploiting invariance, the
triangle inequality, and the contraction estimate
\eqref{eq:wass:cont:intro}, we have
\begin{align}
	\Wass(\mu_{\pmr_0}, \mu_\pmr) 
	=& \Wass(\mu_{\pmr_0} \gP_t^{\pmr_0}, \mu_\pmr \gP_t^\pmr) 
	\leq \Wass(\mu_{\pmr_0} \gP_t^{\pmr_0}, \mu_\pmr \gP_t^{\pmr_0})  + \Wass(\mu_{\pmr} \gP_t^{\pmr_0}, \mu_\pmr \gP_t^\pmr) 
	\notag\\
	\leq& C_0 e^{ - \kappa t}  \Wass(\mu_{\pmr_0}, \mu_\pmr)  + \Wass(\mu_{\pmr} \gP_t^{\pmr_0}, \mu_\pmr \gP_t^\pmr)
	\label{eq:tri:arg}
\end{align}
which holds for any $\pmr \in \Theta$ and any $t \geq 0$.  Thus, by selecting $t_*$ such that, say, $C_0 e^{ - \kappa t_*} \leq 1/2$, we can rearrange the above expression and obtain
\begin{align}
	\Wass(\mu_{\pmr_0}, \mu_\pmr) \leq  2 \Wass(\mu_{\pmr} \gP_{t*}^{\pmr_0}, \mu_\pmr  \gP_{t*}^\pmr).
	\label{eq:reduction:intro}
\end{align}
Thus we obtain a bound which reduces the question of long time
accuracy in the sense of invariant statistics to a certain finite time
error estimate and alongside suitable $\theta$-uniform moment bound on
$\mu_\theta$.

To make this significance of \eqref{eq:reduction:intro} a bit more concrete, we recall that $\Wass$ possesses a desirable Lipschitz structure. For example, if for each $u \in \gsp$, we can find a coupling, $u_{\pmr_0}(t, u), u_{\pmr}(t, u)$ of $\gP_t^{\pmr_0} (u, \cdot),  \gP_t^\pmr (u, \cdot)$ such that
\begin{align}
	\bE \rho( u_{\theta_0}(t, u),  u_{\theta}(t, u)) \leq e^{\tilde{C}_0t} f(u) g(\pmr, \pmr_0),
	\label{eq:finite:err:intro}
\end{align}
where $g$ is a (bounded) distance on $\Theta$, then basic properties of $\Wass$ lead to
\begin{align}
	\Wass(\mu_{\pmr} \gP_{t*}^{\pmr_0}, \mu_\pmr  \gP_{t*}^\pmr)
	\leq e^{\tilde{C}_0t_*}  g(\theta, \theta_0) \int f(u) \mu_{\theta}(d u) .
	\label{eq:finite:err:intro:1}
\end{align}
Hence we obtain from \eqref{eq:reduction:intro} that
\begin{align}
	\Wass(\mu_{\theta_0}, \mu_\theta) \leq  2 e^{\tilde{C}_0 t_*} g(\theta, \theta_0) \int f(u) \mu_{\theta}(d u) .
	\label{eq:finite:err:mom:intro}
\end{align}	
Of course obtaining \eqref{eq:wass:cont:intro} and then providing suitable qualitative estimates 
for $\Wass(\mu_{\pmr} \gP_{t*}^{\pmr_0}, \mu_\pmr  \gP_{t*}^\pmr)$ to leverage via \eqref{eq:reduction:intro} as in \eqref{eq:finite:err:intro}-\eqref{eq:finite:err:mom:intro} represents a bespoke and nontrivial mathematical challenge for each of specific works mentioned previously, \cite{HairerMattingly2008,HairerMattinglyScheutzow2011,johndrow2017error,foldes2017asymptotic,foldes2019large,cerrai2020convergence}. Furthermore, we emphasize for what follows that in order to exploit \eqref{eq:finite:err:mom:intro} we must obtain a uniform bound on $\int f(u) \mu_{\theta}(d u)$ 
as a function of $\theta$. 

This work develops a different and seemingly novel variation on the reduction in \eqref{eq:wass:cont:intro}, \eqref{eq:reduction:intro}. Suppose that, instead of \eqref{eq:wass:cont:intro}, we impose the stronger uniform contraction assumption  
\begin{align}
	\Wass(\mu \gP_t^\pmr, \tilde{\mu} \gP_t^\pmr ) \leq C e^{ - \kappa t} \Wass(\mu, \tilde{\mu}),
	\label{eq:wass:cont:unf:intro:1}
\end{align}
where, to emphasize, the constants $C, \kappa > 0$ are now supposed to be independent of the parameter $\theta \in \Theta$. In comparison to \eqref{eq:tri:arg}, we now proceed as
\begin{align}
	\Wass(\mu_{\theta_0}, \mu_\theta) 
	\leq& \Wass(\mu_{\theta_0} \gP_t^{\pmr_0}, \mu_{\theta_0} \gP_t^\pmr )  + \Wass(\mu_{\theta_0} \gP_t^\pmr, \mu_\theta \gP_t^\pmr) 
	\leq \Wass(\mu_{\theta_0} \gP_t^{\pmr_0}, \mu_{\theta_0} \gP_t^\pmr) + C_0 e^{ - \kappa t}  \Wass(\mu_{\theta_0}, \mu_\theta),
	\label{eq:tri:arg:unf}
\end{align}
so that, by again choosing $t_*$ such that 
\begin{align}
C_0 e^{ -  \kappa t_*} \leq 1/2, 
\label{eq:tstr:cond:intro}
\end{align}
we now find
\begin{align}
	\Wass(\mu_{\theta_0}, \mu_\theta) \leq  2 \Wass(\mu_{\theta_0} \gP_{t*}^{\pmr_0}, \mu_{\theta_0} \gP_{t*}^\pmr).
	\label{eq:reduction:unf:intro}
\end{align}
This seemingly innocent difference in comparison to \eqref{eq:reduction:intro} trades uniformity in the contraction rate for a single moment bound on the limit system. Indeed, under \eqref{eq:finite:err:intro} we obtain
\begin{align}
	\Wass(\mu_{\theta_0}, \mu_\theta) \leq  2 e^{\tilde{C}_0t_*} g(\theta, \theta_0) \int f(u) \mu_{\theta_0}(d u),
\label{eq:reduction:unf:intro:1}
\end{align}
so that we trade the requirement \eqref{eq:wass:cont:unf:intro:1} for
the uniform bound
$\sup_{\theta \in \Theta} \int f(u) \mu_{\theta}(d u)$.

This difference between \eqref{eq:reduction:unf:intro:1} and
\eqref{eq:finite:err:mom:intro} turns out to sometimes be an
indispensable trade off. We exploit it for the questions of numerical
accuracy we consider here as well as other situations of interest as
in the concurrent work \cite{GlattHoltzNguyen2021}. Specifically, as
we will describe in further detail immediately below, for our
applications here $\theta \not = \theta_0$ represents a numerical
approximation parameter for the stochastic Navier-Stokes equations. These
numerical approximations destroy (or complicate) certain crucial
Lyapunov structures, namely we lack the availability of moments for
$\mu_{\theta}$ when $\theta \not = \theta_0$ as would be needed in
\eqref{eq:finite:err:mom:intro}. In any case we refer to
\cref{thm:gen:wk:conv} which is framed in a context applicable to a
slightly weaker form of the uniform contraction estimates
\eqref{eq:wass:cont:unf:intro:1} required for our applications.

Leaving this consideration aside, the uniform contraction assumption
\eqref{eq:wass:cont:unf:intro:1} combined with finite time error
estimate bounds as in \eqref{eq:finite:err:intro} leads to other
desirable long time approximation estimates. Indeed with
\eqref{eq:wass:cont:unf:intro:1} and invoking invariance we obtain the
bound
\begin{align*}
	 \Wass(\mu \gP_t^\pmr, \mu \gP_t^{\pmr_0}) 
	 \leq&  \Wass(\mu \gP_t^\pmr , \mu_{\theta} \gP_t^\pmr ) + \Wass( \mu_\theta,  \mu_{\theta_0})+ \Wass( \mu_{\theta_0} \gP_t^{\pmr_0} , \mu \gP_t^{\pmr_0} ) \\
	 \leq& C_0 e^{-\kappa t} (\Wass(\mu, \mu_{\theta}) + \Wass(\mu, \mu_{\theta_0})) +  \Wass( \mu_\theta,  \mu_{\theta_0})\\
	 \leq&  C_0 e^{-\kappa t} \Wass(\mu, \mu_{\theta_0}) + (1+ C_0) \Wass( \mu_\theta,  \mu_{\theta_0}),
\end{align*}
for any `initial' distribution $\mu$. Hence, with this bound and \eqref{eq:reduction:unf:intro}, valid under \eqref{eq:reduction:intro} and \eqref{eq:finite:err:intro}, we obtain
\begin{align*}
\Wass(&\mu \gP_t^\pmr , \mu \gP_t^{\pmr_0}) \\
  &\leq \min\left\{ C_0 e^{-\kappa t} \Wass(\mu, \mu_{\theta_0}) + 2(1+ C_0)e^{\tilde{C}_0t_*}\! \! \int f(u) \mu_{\theta_0}(d u) \, g (\theta, \theta_0),
	 e^{\tilde{C}_0t} \! \! \int f(u) \mu(d u) \, g (\theta, \theta_0) \right\},
\end{align*}
valid for any $t \geq 0$, where we recall that $t^* > 0$ is given as in \eqref{eq:tstr:cond:intro}. Hence, optimizing appropriately over $t \geq 0$ in this bound we conclude
\begin{align}
	\sup_{t \geq 0} \Wass(\mu \gP_t^\pmr , \mu \gP_t^{\pmr_0}) 
	\leq  C\left( \Wass(\mu, \mu_{\theta_0}) + \int f(u) \mu_{\theta_0}(d u)  + \int f(u) \mu(d u) \right)  g (\theta, \theta_0)^\gamma.
	\label{eq:time:inf:intro}
\end{align}
where $C, \gamma > 0$ are independent of $\mu$ and $\theta$.  

Note that, in the numerical analysis context of interest here, this
bound, \eqref{eq:time:inf:intro}, immediately yields a weak order
approximation estimate valid on the entire time interval
$[0,\infty)$. The operational version of \eqref{eq:time:inf:intro}
formulated in order to address our nonlinear SPDE applications is
formulated in \cref{thm:gen:wk:conv:2} and in
\cref{cor:gen:wk:conv:obs}, which make explicit the connection with
weak order convergence in stochastic numerical analysis.

\subsection{Contributions to the Weak Harris Approach for Wasserstein Contraction}

Of course the elegant simplicity of the above discussion obscures a
number of bedeviling technical challenges which one must address in
order to carry out this program in practice. One challenge is to
establish (uniform) Wasserstein contraction bounds as in
\eqref{eq:wass:cont:intro} and in \eqref{eq:wass:cont:unf:intro:1}.
For this purpose that we develop \cref{thm:gen:sp:gap} below, which
provides general criteria for such contraction estimates. This is a
result that has independent interest for a variety of
infinite dimensional contexts as highlighted by the recent
contributions
\cite{GHRichMa2017,ButkovskyKulikScheutzow2019,glatt2021long}.

As previously mentioned, \cref{thm:gen:sp:gap} provides a new variation on the so called `weak Harris approach' developed in \cite{HairerMattingly2008,HairerMattinglyScheutzow2011}.
This weak Harris approach builds on a wide and well developed literature on mixing rates for Markov chains; 
see e.g. \cite{DaPratoZabczyk1996,meyn2012markov,Kulik2018,douc2018markov} for a systematic 
presentation. The classical Harris theorems date back to the 1950's by building on Doeblin's 
coupling approach to address mixing in unbounded phase spaces. The key is to 
appropriately incorporate the role of Lyapunov structure to facilitate coupling at `large scales'. Typically, in this literature mixing occurs in a total variation (TV) topology or other related `strong topologies' on probability measures; see \cite{hairer2011yet} for a recent treatment close to our setting.  
This use of a total variation topology highlights a limitation of the classical Harris approach:
it turns out to be ill-adapted to infinite dimensional contexts where measures tend  to be mutually singular as exemplified by the Feldman-Hajek theorem (see e.g. \cite[Theorem 2.25]{DaPratoZabczyk2014}).

More recent variations on this Harris approach, largely developed in
an extended body of literature in the SPDE context starting from
\cite{Mattingly2002,BKL2002,Hai02,KukShi2002}, address mixing in
Wasserstein (or the closely related dual-Lipschitz)
distance. Wasserstein distance reflects a weak topology which
sidesteps the issue of mutually singular laws arising in
infinite dimensional stochastic systems. A distinguished contribution
of the works \cite{HairerMattingly2008,HairerMattinglyScheutzow2011}
in this literature is to provide a contraction (or a so called
`spectral-gap') estimate as in \eqref{eq:wass:cont:intro} as suits our
needs here. As in the earlier literature, these results are based on
natural conditions leading to couplings which synchronize two point
dynamics at large, intermediate and small scales, through Lyapunov
structure, irreducibility and smoothing properties,
respectively. However, an elegant feature of
\cite{HairerMattingly2008,HairerMattinglyScheutzow2011} in this wider
mixing literature is the identification of a particular class of
metrics (or pseudo-metrics) on the phase space which are carefully
tailored to account for the three different mechanisms acting at
different scales which drive the coupling. This `norm approach' thus
avoids Byzantine explicit coupling constructions yielding a flexible
approach for applications while producing elegant, transparent proofs.

The approach \cite{HairerMattingly2008, HairerMattinglyScheutzow2011} is well 
adapted to the 2D randomly forced Navier-Stokes on compact domains in the absence of boundaries and as
well as several other reaction-diffusion type models of interest. However, a crucial requirement in \cite{HairerMattingly2008,HairerMattinglyScheutzow2011} 
appears to be stronger than can be expected for a rich variety 
of interesting SPDE examples involving weaker dissipation and/or stronger nonlinearity 
as illustrated in \cite{GHRichMa2017,ButkovskyKulikScheutzow2019,glatt2021long}. Indeed,
\cite{HairerMattingly2008} develops their theory around a certain geodesic metric which is adapted to a quadratic exponential Lyapunov structure, namely $V(u) = \exp(\alpha |u|^2)$ for some $\alpha > 0$.  It
turns out that this geodesic metric approach involves the use of a certain gradient bound on the Markovian dynamics closely related to the so called `asymptotic strong Feller' (ASF) condition introduced earlier in \cite{HaiMa2006}. While the pseudo-metric structures considered later in \cite{HairerMattinglyScheutzow2011} are in various ways more flexible, including in terms of its requirement on the Lyapunov structure,
\cite{HairerMattinglyScheutzow2011} still ultimately relies on these same ASF type gradient bound on
the Markov semigroup. To summarize, the existing works \cite{HairerMattingly2008,HairerMattinglyScheutzow2011} require a significant degree of uniformity across the phase space in contraction rates when two point dynamics are in close proximity. This is rather more than can be hoped for in a variety of interesting situations.

\cref{thm:gen:sp:gap} provides our new take on the weak Harris
approach. Its main advantage over these previous formulations consists
in sidestepping the need for a gradient bound.  Our approach builds on
machinery introduced recently in \cite{ButkovskyKulikScheutzow2019},
which provides a powerful and user friendly toolbox for addressing
exponential mixing by confronting the representative gallery of SPDE
examples introduced in \cite{GHRichMa2017}. Our result here may be
seen to be a sort of intermediate formulation of the topology for
contraction, laying between \cite{HairerMattingly2008} and
\cite{HairerMattinglyScheutzow2011}, and focusing specifically on
$V(u) = \exp(\alpha |u|^2)$.  This intermediate formulation then has
the advantage of allowing us to treat the gallery of examples from
\cite{GHRichMa2017,ButkovskyKulikScheutzow2019}.  Note that our
pseudo-metric does not maintain a generalized triangle inequality due
to the underlying Lyapunov structure around
$V (u) = \exp(\alpha |u|^2)$ as would be strictly required for bounds
like \eqref{eq:tri:arg}, \eqref{eq:tri:arg:unf}.  Instead, taking
advantage of a stronger `super-Lyapunov' structure for $V$, we provide
a `contraction-like' condition (see \eqref{Wass:contr:gen},
\eqref{def:gdista}), which is strong enough to then follow the general
stream of argumentation leading to our reduction bounds
\eqref{eq:reduction:unf:intro:1}, \eqref{eq:time:inf:intro}.

The list of problems in \cite{GHRichMa2017,ButkovskyKulikScheutzow2019} include, notably, 
the 2D stochastic Navier-Stokes equations on a bounded domain subject to the usual nonslip boundary
condition. We provide complete details for this case below in \cref{sec:cont:bnd:dom}, which we believe illustrates the full significance for \cref{thm:gen:sp:gap} in applications. We refer to \cref{rmk:s:bdd} below which provides technical level comparison of our \cref{thm:gen:sp:gap} to the results in \cite{HairerMattingly2008} and in \cite{HairerMattinglyScheutzow2011}.

\subsection{Results for the Numerical Approximation of the Stochastic Navier-Stokes Equations}\label{subsec:res:approx:SNSE}

As already alluded to above, our immediate goal is to demonstrate the
efficacy of the abstract formalism developed in \cref{sec:abs:res} for
the numerical analysis of certain classes of nonlinear SPDEs. As a
paradigmatic model problem, we carry out a detailed study of a fully
discrete numerical scheme to approximate the 2D stochastic
Navier-Stokes equations (SNSE) in the presence of spatially smooth but
sufficiently rich (or `mildly degenerate') stochastic forcing
structure.

Our main result which we preview immediately below as
\cref{thm:num:apx:SNSE:intro} adds to an extensive body of research on
the numerical analysis of stochastic dynamical systems.  However, the
literature on long time numerical approximation for SPDEs is scant,
and, to the best of our knowledge, there is no previous literature on
the stochastic Navier-Stokes or other such `strongly nonlinear'
equations in this regard.  To summarize, our primary contribution
in comparison to the existing numerical analysis literature is to
provide rigorous approximation bounds on invariant measures and to
establish weak convergence estimates, \`a la \eqref{eq:weak:const:intro},
on an infinite time horizon for the stochastic Navier-Stokes
equations.  This usage of our abstract framework lays out an approach
which would apply to the numerical analysis of a number of other
strongly nonlinear infinite dimensional systems seemingly out of reach
of the previously existing approaches, one which follow a very
different set of methodologies in comparison to the extant literature.

Let us be more concrete.  For our numerical application, we consider
the 2D Navier-Stokes equations on the torus $\bT^2$ so that we can
work with the convenient vorticity formulation
\begin{align}\label{2DSNSEv:intro}
\rd \xi + (- \nu \Delta \xi + \bu \cdot \nabla \xi ) \rd t = \sum_{k =1}^d \sigma_k dW_k = \sigma dW, \quad \bu = \mK \ast \xi.
\end{align}
Here $\mK$ is the Biot-Savart operator, which uniquely recovers the
divergence free vector field $\bu$ from $\xi$ (so that
$\xi = \nabla^\perp \cdot \bu$).  The physical parameter $\nu > 0$
represents the kinematic viscosity of the fluid.  The system is driven
by a white in time and spatially smooth Gaussian process
$\sigma dW = \sum_{k=1}^d \sigma_k \rd W^k$, where
$W = (W_1, W_2, \ldots, W_d)$ is a collection of i.i.d. Brownian
motions on a probability space $(\Omega, \mF, \bP)$, and
$\sigma_1, \ldots, \sigma_d$ are (spatially) mean free elements of
$L^2(\bT^2)$. We work under the simplifying assumption, the so called
`essentially-elliptic case', where we suppose a certain non-degeneracy
condition that noise excitation acts directly on some number of low
Fourier modes, namely
\begin{align}
	\mbox{span} \{\sigma_1, \ldots, \sigma_d\} \supset \Pi_K L^2(\bT^2)
	\label{eq:non-degen:intro}
\end{align}
for $K = K( \nu, \sum_{k=1}^d |\sigma_k|_{L^2}^2)$. This is a standard assumption in the SNSE literature, cf. 
\cite{Mattingly2003, KuksinShirikyan12, GHRichMa2017}.  
For simplicity we consider spatially mean free flows, a condition maintained by \eqref{2DSNSEv:intro} so long as the noise itself is mean free.

As numerical approximation of \eqref{2DSNSEv:intro}, we consider a spectral Galerkin discretization in space and a semi-implicit Euler time discretization, given by
\begin{align}\label{disc2DSNSEv:intro}
\xi_{N,\delta}^n = \xi_{N,\delta}^{n-1} + \dt [\nu \Delta \xi_{N,\delta}^n - \Pi_N(\bu_{N,\delta}^{n-1} \cdot \nabla \xi_{N,\delta}^n)] 
	+ \sqrt{\dt}\sum_{k =1}^d \Pi_N \sigma_k \eta_n^k,
\quad \text{ for } n \geq 1,
\end{align}
where  the numerical discretization parameters are the size of the time step $\delta > 0$, and $N \geq 1$ 
the degree of spectral (spatial) approximation so that the operator $\Pi_N$ projects onto the first $2N$ 
Fourier modes. As previously, $\bu_{N,\dt}^{n-1} = \mK \ast \xi_{N,\dt}^{n-1}$.  Here $\sqrt{\dt}  \eta_n^k$ have 
the laws of increments of the Brownian motions $W^k$, so that $\eta_n^k$ is generated by a sequence of
i.i.d. standard Gaussian random variables.  

Our main numerical result is given here in a heuristic formulation as follows.

\begin{Thm}\label{thm:num:apx:SNSE:intro}
Consider  \eqref{2DSNSEv:intro} and \eqref{disc2DSNSEv:intro} under the suitable non-degeneracy condition \eqref{eq:non-degen:intro} that the stochastic perturbation 
acts directly on sufficiently many low frequencies (depending only on $\nu > 0$ and $|\sigma|^2_{L^2} = \sum_k | \sigma_k|^2_{L^2}$).
Then \eqref{2DSNSEv:intro} has a unique statistically invariant state $\mu_*$, and \eqref{disc2DSNSEv:intro}
has a unique statistically invariant state $\mu_*^{N,\delta}$ for any $N \geq 1$, $\delta > 0$.  Moreover, for any sufficiently regular observable 
$\varphi: L^2(\bT^2) \to \RR$, we have the bound
\begin{align}\label{eq:inv:mrs:conv:intro}
	\left| \int \varphi( \xi' ) \mu_*^{N,\delta}( d \xi') - \int \varphi( \xi' ) \mu_*( d \xi') \right| \leq C_\varphi (\delta^{r} + N^{-r'})
\end{align}
for some $r = r( \nu, |\sigma|^2) > 0$, $r' = r'( \nu, |\sigma|^2) > 0$ which do not depend on $\varphi$ and 
where $C_\varphi$, $r$, $r'$ are all $\delta, N$-independent. 

Finally, \eqref{disc2DSNSEv:intro} is a weakly
consistent approximation of \eqref{2DSNSEv:intro}. Namely, for any such observable $\varphi$ and sufficiently regular $\xi_0$ it holds that
\begin{align}\label{eq:weak:const:intro}
	\sup_{n \geq 0} \left| \bE \varphi(\xi_{N,\delta}^n(\xi_0)) - \bE \varphi(\xi(n\delta; \xi_0)) \right| \leq C_\varphi (\delta^{r} + N^{-r'}),
\end{align}
where $\xi(t;\xi_0)$, $t \geq 0$, and $\xi_{N,\dt}^n(\xi_0)$, $n \in \NN$, denote the solutions of \eqref{2DSNSEv:intro} and \eqref{disc2DSNSEv:intro}, respectively, with initial datum $\xi_0$.
\end{Thm}

The precise and complete formulation of \cref{thm:num:apx:SNSE:intro} is divided between \cref{thm:bias:ineq:NSE} and \cref{thm:wk:conv:SNSE} below. In particular, its proof is founded on a discretization-uniform contraction bound from \cref{thm:sp:gap:disc:SNSE} and the finite-time error estimates from
\cref{prop:spatial:error} and  \cref{prop:time:disc:error}, which provide concrete instantiations of \eqref{eq:wass:cont:unf:intro:1} and \eqref{eq:finite:err:intro}, in addition to being contributions of independent interest for \eqref{2DSNSEv:intro} and \eqref{disc2DSNSEv:intro}.

Further, we notice that  \eqref{eq:inv:mrs:conv:intro} together with the contraction inequality from \cref{thm:sp:gap:disc:SNSE} can be used to derive useful error estimates for the estimator $\frac{1}{n} \sum_{k =1}^n\varphi( \xi_{N,\delta}^k)$ as an approximation of the stationary average $\int \varphi( \xi' ) \mu_*( d \xi')$. Indeed, in \cref{rmk:variance:estimator} below we sketch the main steps involved in the derivation of the following bias estimate
\begin{align}\label{ineq:bias:intro}
\left| \bE \left( \frac{1}{n} \sum_{k =1}^n\varphi( \xi_{N,\dt}^k(\xi_0)) - \int \varphi(\xi' ) \mu_*( d\xi') \right) \right|
\leq C_{\varphi} \left( \frac{1}{n \delta} + \delta^{r} + N^{-r'} \right),
\end{align}
and the mean-squared error estimate
\begin{align}\label{eq:inv:mrs:appro:intro}
\bE \left| \frac{1}{n} \sum_{k =1}^n\varphi( \xi_{N,\delta}^k) - \int \varphi( \xi' ) \mu_*( d \xi') \right|^2 \leq C_\varphi \left( \frac{1}{n \delta} + \delta^{2r} + N^{-2r'}\right),
\end{align}
for $C_\varphi$, $r$, $r'$ as in
\eqref{eq:inv:mrs:conv:intro}. Clearly, estimates such as these have a
direct significance in practical applications where one naturally
computes the time-discrete average
$\frac{1}{n} \sum_{k =1}^n\varphi( \xi_{N,\delta}^k)$ for a certain
number $n$ of states as a way of approximating the average of a given
observable $\varphi$ with respect to the (typically unknown)
underlying stationary distribution $\mu_*$.

\subsubsection*{Previous Literature, Elements of Our Analysis}

There is an extensive literature on numerical analysis of stochastic
systems.  Some general background on this subject in the context of
SDEs can be found in e.g. \cite{kloeden1992stochastic,
  milstein2004stochastic}, and for SPDEs we refer to
\cite{jentzen2009numerical,lord2014introduction}.  In this community,
approximation results are typically characterized in terms of `strong'
and `weak' convergence.  The former notion of strong convergence
concerns, in our notations, bounds on the quantity
$|\xi_{N,\dt}^j(\xi_0) - \xi (j \dt;\xi_0)|_{L^2}$ either in mean or
in probability.  The latter notion of `weak' convergence
involves estimates for
$\bE (\phi(\xi_{N,\dt}^j(\xi_0)) - \phi(\xi (j \dt;\xi_0)))$ over
different classes of test functions, which can be expressed in terms of bounds on the Wasserstein metric, see \eqref{ineq:obs:Wass} below. For such classes of observables embodied in these metrics, strong convergence implies weak convergence but of course not visa-versa
and typically weak rates are better than strong rates (cf. \cite{DavieGaines2001,DebusschePrintems2009}).

The available approaches for weak convergence are mainly centered around the following observation. Expanding $\bE (\phi(\xi_{N,\dt}^j(\xi_0)) - \phi(\xi (j \dt;\xi_0)))$ in a telescoping sum allows one to estimate the error using the Kolmogorov equation associated to the limiting dynamic. These approaches require some degree of regularity for solutions of the Kolmogorov equation.  Additionally, note that one needs to show that these estimates are uniform in $j$ in order to address long time accuracy. This typically entails obtaining a time decay for the corresponding solutions. A further difficulty for SPDEs is that our Kolmogorov equation is a parabolic PDE whose ``spatial'' variable is infinite dimensional. In the setting we are concerned with here, due to the necessity and interest for noise acting in a limited subset of the phase space, the Kolmogorov equation has a degenerate second-order term. Furthermore, its drift term involves an unbounded operator and a strongly nonlinear term. 

Details of the Kolmogorov approach to weak convergence vary e.g. according to the model of interest,
the type of discretization considered, and the topology in which
numerical convergence is established, but frequently the estimates
appear with time-interval length dependent bounds.  For
finite dimensional SDEs we mention the pioneering works
\cite{Milshtein1979, Milshtein1995, Talay1984, Talay1986}, which were
further refined in a significant body of work, see
e.g. \cite{TalayTubaro1990, kloeden1992stochastic, BallyTalay1996,
  SzepessyTemponeZouraris2001, KohatsuHiga2001, CKHL2006} and
references therein. Analogous weak convergence results for SPDEs were
more recently obtained in e.g. \cite{AnderssonLarsson2016,
  AnderssonKruseLarsson2016, BuckwarShardlow2005,
  ConusJentzenKurniawan2019, DavieGaines2001, DeBouardDebussche2006,
  DebusschePrintems2009, Debussche2011, Hausenblas2010,
  JentzenKurniawan2021, KovacsLarssonLindgren2012,
  KovacsLarssonLindgren2013, Wang2016, WangGan2013} for equations that
are linear or with globally Lipschitz nonlinearities, and
\cite{BrehierDebussche2018, BrehierGoudenege2020, CaiGanWang2021,
  CuiHong2019,CuiHongSun2021,Dorsek2012} for more general nonglobally
Lipschitz scenarios. Another set of works focused on obtaining
long time approximation error bounds as in
\eqref{eq:inv:mrs:conv:intro}, \eqref{ineq:bias:intro} or
\eqref{eq:inv:mrs:appro:intro}, for either SDEs \cite{abdulle2014high,
  debussche2012weak, MattinglyStuartTretyakov,
  shardlow2000perturbation, talay1990second, talay2002stochastic} or
SPDEs \cite{Brehier2014,brehier2017approximation, Brehier2022,
  CuiHongSun2021, HongWang2019}. 

As previously mentioned, our approach for proving
\cref{thm:num:apx:SNSE:intro} is instead based on the uniform
Wasserstein contraction framework described above. Namely, by
establishing a uniform Wasserstein contraction estimate as in
\eqref{eq:wass:cont:unf:intro:1} together with a finite-time error
estimate as in \eqref{eq:finite:err:intro}. Regarding the latter, our
bound is in fact given in terms of the strong discretization error in
$L^p(\Omega;L^\infty_{\text{loc},t} L^2_x)$ for a sufficiently small
$p$, which is estimated from \cref{prop:spatial:error} and
\cref{prop:time:disc:error} below. Clearly, this approach is most
likely not guaranteeing an optimal weak convergence rate in
\eqref{eq:weak:const:intro}. Indeed, as we previously mentioned, it is generally expected that the
weak order of convergence is larger than the strong order, and many of
the references on weak convergence results mentioned above focused
precisely on establishing this order improvement.

However, we emphasize that the main advantage of our approach lies in
yielding a \emph{uniform in time} weak error estimate,
\eqref{eq:weak:const:intro}, in addition to providing long time error
estimates for approximations of the limiting stationary distribution,
i.e. \eqref{eq:inv:mrs:conv:intro}, \eqref{ineq:bias:intro},
\eqref{eq:inv:mrs:appro:intro}. Notably, these are the first results
of such type to be established for the stochastic Navier-Stokes
equations. Previous works on numerical approximations for the SNSE
focused on strong convergence in either probability
\cite{bessaih2014splitting,BreitDodgson2021,BreitProhl2022,CarelliProhl2012,RandrianasoloHausenblas2019}
or in mean 
\cite{BessaihMillet2019,BessaihMillet2021,BessaihMillet2022,brzezniak2013finite,Dorsek2012,MilsteinTretyakov2021}
for various space or time discretizations and noise types but always
for bounds on finite time windows $[0,T]$ with exponentially growing
constants as a function of $T > 0$.

Regarding the strong error bound implied by \cref{prop:spatial:error} and \cref{prop:time:disc:error}, we notice that it is given more explicitly, for sufficiently regular starting point $\xi_0$, as
\begin{align}\label{ineq:str:err}
	\bE \sup_{j \leq J} |\xi_{N,\dt}^j(\xi_0) - \xi (j \dt;\xi_0)|_{L^2}^p \leq C \left[ \dt^{p' p} + N^{-\frac{p}{2}} \right]
\end{align}
for any $J \in \NN$, $p' \in (0,1/2)$, and for $p > 0$ sufficiently small depending only on the viscosity parameter $\nu$ and $|\sigma|^2_{L^2} = \sum_k | \sigma_k|^2_{L^2}$, where $C = C(p, p', J)$ is a positive constant. This yields strong $L^p(\Omega)$ convergence\footnote{See e.g. \cite[Definition 2.6]{Printems2001}.} for the scheme \eqref{disc2DSNSEv:intro} with respect to the topology in $L^\infty_{\text{loc},t} L^2_x$, and with rates (almost) $1/2$ in time and $1$ in space ($N \sim h^{-2}$ for a spatial grid with cell edge length $h$). Additionally, as $\nu$ grows large, $p$ can be almost $2$. This temporal convergence rate is optimal due to $1/2$-H\"{o}lder time regularity of the solution, which is in turn implied by the regularity of the underlying stochastic forcing term. Optimal rates for other types of numerical approximations or noise terms for the 2D SNSE were also achieved in previous works under the velocity formulation, particularly \cite{bessaih2014splitting,BreitDodgson2021,BreitProhl2022} regarding convergence in probability, and \cite{BessaihMillet2022} concerning strong $L^2(\Omega)$ convergence under a suitable smallness assumption on the noise. In relation to these recent results, we expect our method of proof for deriving \eqref{ineq:str:err} to be of independent interest. See more details in \cref{subsubsec:fin:tim:err:L2} below.

\subsection{Outlook and Future Work}

A number of avenues for future development suggest themselves as an outgrowth of the work here.    Firstly, the model 
problems suggested in \cite{GHRichMa2017,ButkovskyKulikScheutzow2019,glatt2021long} provide a set of interesting
challenges for numerical analysis. Here note that, while the results in \cref{sec:cont:bnd:dom} provide a first step toward addressing 
the case of 2D stochastic Navier-Stokes on a domain with boundaries, subtle details remain to complete the analogous
program to the one which we fulfilled in the periodic setting in \cref{sec:app:SNSE}. Note furthermore that  \cref{sec:app:SNSE}
addresses just one of a variety of possible numerical approximations of governing equations, and indeed each of the model equations 
in \cite{GHRichMa2017,ButkovskyKulikScheutzow2019,glatt2021long} would be expected to have their own bespoke 
natural  approximation schemes. Another challenge for numerical  accuracy would be to address the fully hypo-elliptic case.  Here, to obtain a uniform rate of contraction
one would  presumably need to develop a discrete time analogue of the infinite dimensional Malliavin calculus based approaches
developed in \cite{HaiMa2006,HM11,FGHRT15,kuksin2020exponential}. Of course, other interesting parameter limit problems
outside of numerical approximation may be addressed from our formalism as in our concurrent work \cite{GlattHoltzNguyen2021}.

Finally, it is notable that abstract frameworks developed in \cref{sec:abs:res} have a scope of applicability reaching
far beyond the SPDE models that we have focused on here. As already identified in \cite{johndrow2017error}, one may leverage the type of contractivity obtained from weak Harris results as a means of bias estimation in 
a variety of applications in computational statistics. Our upcoming contribution \cite{GHHKMBias2022} expands on this insight particularly leveraging the use of uniformity identified here.

\subsection*{Organization}

The rest of this manuscript is organized as follows.  In \cref{sec:abs:res} we present our main abstract results, namely our Wasserstein contraction criteria in \cref{sec:gen:sp:gap} followed by our parameter convergence/stability at time $\infty$ given in \cref{sec:gen:wk:conv}.  \cref{sec:app:SNSE} presents our first application concerning the 
numerical analysis of a fully discrete scheme for the stochastic Navier-Stokes equations.  
Finally, \cref{sec:cont:bnd:dom} presents contraction estimates for the stochastic Navier-Stokes equations on a bounded domain.

\section{Abstract results}
\label{sec:abs:res}

Before presenting our general results in \cref{sec:gen:sp:gap} and \cref{sec:gen:wk:conv}, we briefly recall in \cref{subsec:prelim} some standard definitions regarding Markov processes and the Wasserstein distance on spaces of probability measures. For more details, we refer to e.g. \cite{DaPratoZabczyk1996,villani2008}.

\subsection{Preliminaries}\label{subsec:prelim}

Let $\ggsp$ be a Polish space. Throughout this manuscript, we denote by $\mB(\ggsp)$ the $\sigma$-algebra of Borel subsets of $\ggsp$, and by $\Pr(\ggsp)$ the corresponding space of Borel probability measures. We also denote by $\mM_b(\ggsp)$ the family of all real-valued, bounded and Borel-measurable functions on $\ggsp$.

We recall that $P: \ggsp \times \mB(\ggsp) \to [0,1]$ is a \emph{Markov kernel} if $P(\cdot,\cO)$ is measurable for each fixed $\cO \in \mB(\ggsp)$, and $P(u, \cdot)$ is a probability measure for each fixed $u \in \ggsp$. For any measure $\mu \in \Pr(\ggsp)$, we recall that its dual action on a Markov kernel $P$ is given by
\begin{align*}
\mu P (\cO) \coloneqq \int_{\ggsp} P(u, \cO) \mu(du), \quad \cO \in \mB(\ggsp).
\end{align*}
A measure $\mu \in \Pr(\ggsp)$ is said to be \emph{invariant} with respect to a family of Markov kernels $P_t$, $t \geq 0$, if and only if $\mu P_t = \mu$ for every $t \geq 0$.

Moreover, a \emph{Markovian transition function} is a family of Markov kernels $P_t$, $t \geq 0$, such that, for each $u \in \ggsp$ and $\cO \in \mB(\ggsp)$, $P_0(u, \cO) = \ind_\cO(u)$, where $\ind_\cO$ denotes the indicator function of $\cO$, and it satisfies the Chapman-Kolmogorov relation 
\begin{align*}
	P_{t + s}(u,\cO) = P_t P_s(u,\cO) \coloneqq \int_{\ggsp} P_s(v,\cO) P_t(u,dv).
\end{align*}

Given such Markovian transition function, its associated \emph{Markov semigroup} is defined as the family of operators $P_t$, $t \geq 0$, acting on functions $\varphi \in \mM_b(\ggsp)$ as
\begin{align}\label{def:Pt:phi}
	P_t \varphi(u) \coloneqq \int_\ggsp \varphi(v) P_t(u, dv), \quad u \in \ggsp.
\end{align} 

Finally, we recall that a mapping $\rho: \ggsp \times \ggsp \to \RR^+$ is called a \emph{distance-like} function if it is symmetric, lower semi-continuous, and satisfies that $\rho(u, \tilde{u}) = 0$ if and only if $u =  \tilde{u}$, see \cite[Definition 4.3]{HairerMattinglyScheutzow2011}. For any such distance-like function $\rho$, its Wasserstein-like extension to $\Pr(\ggsp)$ is the mapping $\Wass_\rho: \Pr(\ggsp) \times \Pr(\ggsp) \to \RR^+ \cup \{\infty\}$ defined as
\begin{align}\label{def:Wass:rhoe}
	\Wass_\rho (\mu, \tmu) = \inf_{\Gamma \in \Co(\mu, \tmu)} \int_{\ggsp \times \ggsp} \rho(u,  \tilde{u}) \Gamma (d u, d \tilde{u}),
\end{align}
where $\Co(\mu, \tmu)$ denotes the family of all \textit{couplings} of $\mu$ and $\tmu$, i.e. all probability measures $\Gamma$ on the product space $\ggsp \times \ggsp$ with marginals $\mu$ and $\tmu$. We notice that when $\rho$ is a metric on $\ggsp$, then its corresponding extension $\Wass_\rho$ coincides with the usual Wasserstein-1 distance, \cite{villani2008}.

\subsection{Wasserstein contraction}
\label{sec:gen:sp:gap}

Our first general result, \cref{thm:gen:sp:gap} below, provides a general set of assumptions on a given Markov semigroup that are sufficient for guaranteeing its contraction with respect to a suitable Wasserstein distance. Our formulation is  inspired by the weak Harris theorem from \cite[Theorem 4.8]{HairerMattinglyScheutzow2011}, which yields an analogous Wasserstein contraction under three main assumptions on the Markov semigroup. Namely, the existence of a Lyapunov function; a smallness condition for trajectories departing from certain level sets of the Lyapunov function; and a contractivity assumption between trajectories departing from points that are sufficiently ``close'' to each other. 

In our set of hypotheses, we focus on stochastic systems possessing an exponential Lyapunov structure, while allowing for more flexibility regarding the contractivity requirement, see \eqref{loc:contr} below. In particular, our ``contraction'' coefficient is given as the product of a constant that is smaller than $1$ with an exponential term depending on one of the starting points. This is tailored to reflect a typical situation in applications to SPDEs, particularly involving a dissipative structure. Indeed, this is demonstrated in the applications to the stochastic Navier-Stokes equations in \cref{subsec:Wass:contr:NSE} and \cref{sec:cont:bnd:dom} below. 

\begin{Thm}\label{thm:gen:sp:gap}
	Let $\gsp$ be a separable Banach space with norm $\gn{\cdot}$. Consider an index set $\mI \subset \RR^+$,
	and take $\{\gP_t\}_{t \in \mI}$ to be a Markov semigroup  on $\gsp$ satisfying
	\begin{enumerate}[label={(A\arabic*)}]
		\item\label{A:1} (Exponential Lyapunov structure) There exists a continuous function $\psi: \RR^+ \to \RR^+$ with $\lim_{t \to \infty} \psi(t) = 0$, and also $\sm_0 > 0$ such that for all $\sm \in (0, \sm_0]$, $t \in \mI$ and $u_0 \in \gsp$, the following inequality holds:
		\begin{align}\label{exp:Lyap:ineq}	
			\gP_t \exp(\sm \gn{u_0}^2) \leq \exp(\sm (\psi(t) \gn{u_0}^2 + \LC))
		\end{align}
		for some constant $\LC > 0$ which is independent of $t, u_0 $ and $\sm$.
	\end{enumerate}
	Furthermore, we fix a collection  $\gfd$ of distance-like functions $\gdist: \gsp \times \gsp \to [0,1]$ and consider the following set of assumptions on $\gfd$ and $\{\gP_t\}_{t \in \mI}$: 
	\begin{enumerate}[label={(A\arabic*)}]
		\setcounter{enumi}{1}
		\item\label{A:2}(Eventual $\rho$-smallness of bounded sets) For every $M > 0$ and $\gdist \in \Lambda$ there exists $T_1 = T_1 (M, \gdist)> 0$ and $\kappa_1 = \kappa_1(M) \in (0,1)$, which is independent of $\gdist$, such that
		\begin{align}
			\sup_{t \in \mI, \, t \geq T_1} \Wass_{\gdist}(\gP_t(u_0,\cdot), \gP_t(v_0, \cdot)) \leq 1 - \kappa_1
		\end{align}
		for every $u_0, v_0 \in \gsp$ with $\gn{u_0} \leq M$ and $\gn{v_0} \leq M$.
		\item\label{A:3} 
		For every $\kappa_2 \in (0,1)$ and for every $r > 0$ there exists $\gdist \in \Lambda$ for which the following holds:
		\begin{enumerate}[label={(A3.\roman*)}]
			\item\label{A3:i}(Eventual local $\rho$-contractivity)  There exists $T_2 = T_2(\kappa_2, r) > 0$ such that 
			\begin{align}\label{loc:contr}
				\sup_{t \in \mI, \, t \geq T_2} \Wass_{\gdist}(\gP_t(u_0,\cdot), \gP_t(v_0, \cdot)) \leq \kappa_2 \exp(r \gn{u_0}^2 ) \gdist (u_0, v_0)
			\end{align}
			for every $u_0, v_0 \in \gsp$ with $\gdist(u_0,v_0) < 1$.
			\item\label{A3:ii} For all $\tau \geq 0$, there exists $C = C(\tau, \gdist) > 0$ such that
			\begin{align}\label{fin:time:Wass}
			\sup_{t\in \mI, \,\, t \in [0,\tau]} \Wass_{\gdist}(\gP_t(u_0,\cdot), \gP_t(v_0, \cdot)) \leq C \exp(r \gn{u_0}^2 ) \gdist (u_0, v_0)
			\end{align}
			for every $u_0, v_0 \in \gsp$ with $\gdist(u_0,v_0) < 1$.
		\end{enumerate}
	\end{enumerate}
	Then, under the assumptions \ref{A:1}, \ref{A:2} and \ref{A3:i} it follows that for every $m \geq 1$, there exists $\sm_m  > 0$ such that for each $\sm \in (0, \sm_m]$ there exists $\gdist \in \Lambda$, $\kappa \in (0,1)$ and $T > 0$ for which the following inequality holds
	\begin{align}\label{Wass:contr:gen}
	\Wass_{\gdista}(\mu \gP_t, \tmu \gP_t) \leq \kappa \Wass_{\gdist_{\sm/m}} (\mu, \tmu)
	\end{align}
	for every $t \in \mI$ with $t \geq T$, and for all $\mu, \tmu \in \Pr(\gsp)$. Here, for each $a > 0$, $\gdist_{a}: \gsp \times \gsp \to \RR^+$ is the distance-like function defined as 
	\begin{align}\label{def:gdista}
	\gdist_{a}(u, v) \coloneqq \gdist(u, v)^{1/2} \exp(a \gn{u}^2 + a \gn{v}^2), \quad u, v \in \gsp.
	\end{align}
	
	Moreover, under additionally assumption \ref{A3:ii} it follows that for every $m \geq 1$, there exists $\sm_m  > 0$ such that for each $\sm \in (0, \sm_m]$ there exists $\gdist \in \Lambda$, $T > 0$ and constants $C_1, C_2 > 0$ for which it holds that
	\begin{align}\label{Wass:contr:gen:2}
	\Wass_{\gdista}(\mu\gP_t, \tmu \gP_t) \leq C_1 e^{- C_2 t} \Wass_{\gdist_{\sm/m}} (\mu, \tmu)
	\end{align}
	for every $\mu, \tmu \in \Pr(\gsp)$ and all $t \in \mI$ with $t \geq T$. Here, the constants $C_1$ and $C_2$ depend \emph{only} on the parameters $m$, $\sm$, $\gdist$, $T$, and the constants $\sm_0$, $C_0$, $\sup_{t \geq 0} \psi$ from assumption \ref{A:1}.
\end{Thm}

\begin{proof}
	We start by noticing that, since each $\gdista$ is lower-semicontinuous and non-negative, it follows from \cite[Theorem 4.8]{villani2008} that for every $\mu, \tmu \in \mP(\gsp)$
	\begin{align}\label{Wass:convex}
		\Wass_{\gdista}(\mu \gP_t , \tmu \gP_t) \leq \int_{\gsp \times \gsp} \Wass_{\gdista} (\gP_t(u_0,\cdot), \gP_t(v_0,\cdot)) \Gamma (d u_0, d v_0),
	\end{align}
	for every coupling $\Gamma \in \mC(\mu, \tmu)$. Therefore, to show \eqref{Wass:contr:gen} it suffices to obtain that for every $m \geq 1$ there exists $\sm_m  > 0$ such that for each $\sm \in (0, \sm_m]$ there exists $\gdist \in \gfd$, $\kappa \in (0,1)$ and $T > 0$ for which the following holds
	\begin{align}\label{pt:sp:gap:kappa}
		\Wass_{\gdista}(\gP_t(u_0,\cdot), \gP_t(v_0,\cdot)) 
		\leq
		\kappa \gdist_{\sm/m}(u_0, v_0)
	\end{align}
	for every $t \in \mI$ with $t \geq T$, and every $u_0, v_0 \in \gsp$.

	Fix $m \geq 1$ and $u_0, v_0 \in \gsp$. For some fixed $\gdist \in \Lambda$ to be suitably chosen later in terms of $m, \sm$ and $\sm_0, C_0$ from \ref{A:1}, and following similar ideas from \cite[Theorem 4.8]{HairerMattinglyScheutzow2011}, we split the proof into three cases:
	\begin{enumerate}[label={(\roman*)}]
		\item Let us first suppose that $\gdist(u_0, v_0) = 1$ and $\gn{u_0}^2 + \gn{v_0}^2 \leq 6 m \LC$, with $\LC > 0$ as in assumption \ref{A:1}. From the definition of $\Wass_{\gdista}$ in \eqref{def:Wass:rhoe} and H\"older's inequality, it follows that for all $t \in \mI$
		\begin{align}\label{bound:W:P:a}
			&\Wass_{\gdista }(\gP_t(u_0,\cdot), \gP_t(v_0,\cdot)) 
			=  \inf_{\Gamma  \in \Co(\gP_t(u_0, \cdot), \gP_t(v_0, \cdot))} \int_{\gsp \times \gsp} \gdist(u, v)^{1/2} \exp\left(\sm \gn{u}^2 + \sm \gn{v}^2 \right) \Gamma (d u, d v) 
			\notag \\
			&\qquad \leq \inf_{\Gamma  \in \Co(\gP_t(u_0, \cdot), \gP_t(v_0, \cdot))} \left(\int_{\gsp \times \gsp} \gdist(du, d v)\, \Gamma (d u, d v) \right)^{1/2} \left( \int_{\gsp \times \gsp} \exp(4 \sm \gn{u}^2) \, \Gamma (d u, d v) \right)^{1/4} 
			\notag\\
			&\qquad \qquad \qquad \qquad \qquad \qquad\qquad \qquad \qquad\qquad \qquad \qquad \cdot \left( \int_{\gsp \times \gsp} \exp(4 \sm \gn{v}^2) \, \Gamma (d u, d v) \right)^{1/4} 
			\notag \\
			&\qquad  = \Wass_\gdist(\gP_t(u_0,\cdot), \gP_t(v_0,\cdot))^{1/2}  
			\left( \gP_t \exp(4 \sm\gn{u_0}^2) \right)^{1/4}  	\left( \gP_t \exp(4 \sm\gn{v_0}^2) \right)^{1/4}.
		\end{align}
		We now invoke assumption \ref{A:2} with $M \coloneqq (6 m \LC)^{1/2}$ to estimate the first factor in \eqref{bound:W:P:a}, and assumption \ref{A:1} to estimate the last two factors, assuming $\sm \in (0, \sm_0/4]$. It thus follows that for every $t \in \mI$ with $t \geq T_1$
		\begin{align}\label{bound:W:P:b}
			\Wass_{\gdista} (\gP_t(u_0,\cdot), \gP_t(v_0,\cdot)) 
			&\leq
			(1 - \kappa_1)^{1/2} \exp(2 \sm \LC) \exp(\sm \psi(t) (\gn{u_0}^2 + \gn{v_0}^2 ))
		\end{align}
		Since $\kappa_1 \in (0,1)$, we can take $\sm_m \in (0, \sm_0/4]$ sufficiently small such that for every $\sm \in (0,\sm_m]$ it holds
		\begin{align}\label{choice:sm}
			\tilde{\kappa}_1 \coloneqq (1 - \kappa_1)^{1/2} \exp(2 \sm \LC)  \leq  (1 - \kappa_1)^{1/2} \exp(2 \sm_m \LC) < 1.
		\end{align}
		Moreover, since $\lim_{t \to \infty} \psi(t) = 0$, we can take $\tT_1 \geq T_1$ sufficiently large such that $\psi(\tT_1) \leq 1/m$, so that it follows from \eqref{bound:W:P:b} that for any fixed $\sm \in (0,\sm_m]$ and for every $t \in \mI$ with $t \geq \tT_1$
		\begin{align}\label{W:middle:sc}
			\Wass_{\gdista} (\gP_t(u_0,\cdot), \gP_t(v_0,\cdot)) 
			\leq
			\tilde{\kappa}_1 \exp \left(\frac{\sm}{m} (\gn{u_0}^2 + \gn{v_0}^2 )\right)
			= \tilde{\kappa}_1 \gdist_{\sm/m}(u_0, v_0),
		\end{align}	
		where the equality follows from the assumption that $\gdist(u_0, v_0) = 1$.
		
		\item  Next, we assume that $\gdist(u_0, v_0) = 1$ and $\gn{u_0}^2 + \gn{v_0}^2 > 6 m \LC$. Since $\gdist(u, v) \leq 1$ for all $u, v \in \gsp$, it follows together with H\"older's inequality and assumption \ref{A:1} that for any fixed $\sm \in (0, \sm_m]$, with $\sm_m$ as in \eqref{choice:sm}, and for every $t \in \mI$
		\begin{align}\label{bound:W:P:c}
			\Wass_{\gdista}(P_t(u_0,\cdot), P_t(v_0, \cdot)) 
			&=  \inf_{\Gamma  \in \Co(\gP_t(u_0, \cdot), \gP_t(v_0, \cdot))} \int_{\gsp \times \gsp} \gdist(u, v)^{1/2} \exp\left(\sm \gn{u}^2 + \sm \gn{v}^2 \right) \Gamma (d u, d v) 
			\notag\\
			&\leq  \left( P_t \exp(2 \sm \gn{u_0}^2)\right)^{1/2} \left( P_t  \exp(2 \sm \gn{v_0}^2) \right)^{1/2}
			\notag\\
			&\leq \exp(2 \sm \LC) \exp(\sm \psi(t) (\gn{u_0}^2 + \gn{v_0}^2 )).
		\end{align}
		
		Notice that
		\begin{align*}
			\exp(2 \sm \LC) = \exp(-\sm \LC) \exp\left( \frac{\sm}{2m} 6 m \LC \right)
			< \exp(-\sm \LC) \exp \left( \frac{\sm}{2m} (\gn{u_0}^2 + \gn{v_0}^2 ) \right).
		\end{align*}
		Thus, from \eqref{bound:W:P:c}, we obtain
		\begin{align*}
			\Wass_{\gdista}(P_t(u_0,\cdot), P_t(v_0, \cdot)) 
			\leq 
			\exp(-\sm \LC)  \exp \left( \left(\frac{\sm}{2m} + \sm \psi(t) \right) (\gn{u_0}^2 + \gn{v_0}^2 ) \right).
		\end{align*}
		Therefore, taking $\tT > 0$ sufficiently large such that $\psi(\tT) < 1/(2m)$, we deduce that for all $t \in \mI$ with $t \geq \tT$
		\begin{align}\label{W:large:sc}
			\Wass_{\gdista}(P_t(u_0,\cdot), P_t(v_0, \cdot)) 
			\leq 
			\tilde{\kappa} \exp \left( \frac{\sm}{m}  (\gn{u_0}^2 + \gn{v_0}^2 ) \right) = \tilde{\kappa} \gdist_{\sm/m}(u_0,v_0),
		\end{align}
		where $\tilde{\kappa} \coloneqq \exp(-\sm \LC) < 1$.
		
		\item Finally, let us suppose that $\gdist(u_0, v_0) < 1$. Take $r \coloneqq \sm/m$, for a fixed $\sm \in (0, \sm_m]$, with $\sm_m \in (0, \sm_0/4]$ as chosen in \eqref{choice:sm}. Moreover, take  $\kappa_2 \in (0,1)$ satisfying $\kappa_2 < \exp(-\sm_0 C_0)$, with $\sm_0, C_0$ from assumption \ref{A:1}. For these choices of $r$ and $\kappa_2$, we fix $\gdist \in \gfd$ as being the corresponding distance-like function for which assumption \ref{A:3} holds. Here we notice carefully that since $\gdist$ depends on $r$, which depends on $\sm$, which, in its turn, as seen from \eqref{choice:sm}, depends on $\kappa_1$ from assumption \ref{A:2}, it thus follows that $\gdist$ depends on $\kappa_1$. Therefore, the fact that $\kappa_1$ in assumption \ref{A:2} is \emph{independent} of $\gdist$ is crucial for preventing a circular argument in the choice of $\gdist \in \gfd$.
		
		Proceeding with the same estimate as in \eqref{bound:W:P:a}, we now invoke assumption \ref{A3:i} to estimate the first factor, and assumption \ref{A:1} to estimate the remaining two factors. It thus follows that for every $t \in \mI$ with $t \geq T_2$
		\begin{align}\label{bound:W:P:d}
			&\Wass_{\gdista}(P_t(u_0,\cdot), P_t(v_0, \cdot)) \notag \\
			&\leq 
			\left[ \kappa_2 \exp\left(\frac{\sm}{m}\gn{u_0}^2 \right) \gdist(u_0,v_0) \right]^{1/2}  \exp(2 \sm \LC) \exp(\sm \psi(t) (\gn{u_0}^2 + \gn{v_0}^2 ))
			\notag \\
			&\leq \kappa_2^{1/2}  \exp(2 \sm \LC) \gdist(u_0, v_0)^{1/2} \exp\left(  \left(\frac{\sm}{2m} + \sm \psi(t) \right) (\gn{u_0}^2 + \gn{v_0}^2 ) \right).
		\end{align}
		Since $\sm \in (0, \sm_0/4]$ and, by the choice of $\kappa_2$, we have $\kappa_2 < \exp(-\sm_0 \LC)$, then
		\begin{align}\label{ineq:tkappa2}
			\tilde{\kappa}_2 \coloneqq (\kappa_2 \exp(4 \sm \LC))^{1/2} 
			\leq 
			(\kappa_2 \exp( \sm_0 \LC))^{1/2} < 1.
		\end{align}
		Moreover, taking as before $\tT_2 \geq T_2$ sufficiently large such that $\psi( \tT_2) < 1/(2m)$, we obtain from \eqref{bound:W:P:d} and \eqref{ineq:tkappa2} that for all $t \in \mI$ with $t \geq \tT_2$
		\begin{align}\label{W:small:sc}
			\Wass_{\gdista}(P_t(u_0,\cdot), P_t(v_0, \cdot)) 
			\leq
			\tilde{\kappa}_2 \gdist_{\sm/m}(u_0,v_0).
		\end{align}	
	\end{enumerate}
	
	From \eqref{W:middle:sc}, \eqref{W:large:sc} and \eqref{W:small:sc}, it follows that for each fixed $m \geq 1$ there exists $\sm_m > 0$ such that for each $\sm \in (0, \sm_m]$ there exists $\gdist \in \gfd$ and $T > 0$ for which \eqref{pt:sp:gap:kappa} holds with $\kappa \coloneqq \min \{\tilde{\kappa}_1, \tilde{\kappa}_2, \tilde{\kappa} \}$ and for all $t \in \mI$ with $t \geq T$. This finishes the first part of the proof.
	
	We proceed to show inequality \eqref{Wass:contr:gen:2} under the additional assumption \ref{A3:ii}, with the same choices of $\kappa_2$ and $r$ from step $(iii)$ above. Again as a consequence of \eqref{Wass:convex}, it suffices to show that for every $m \geq 1$ there exists $\sm_m > 0$ such that for each $\sm \in (0, \sm_m]$ there exists $\gdist \in \gfd$, $T > 0$ and constants $C_1, C_2 > 0$ such that
	\begin{align}\label{pt:sp:gap}
	\Wass_{\gdista}(\gP_t(u_0,\cdot), \gP_t(v_0,\cdot)) 
	\leq C_1 e^{- C_2 t} \gdist_{\sm/m}(u_0,v_0)
	\end{align}
	for all $t \in \mI$ with $t \geq T$, and for all $u_0, v_0 \in \gsp$.

	Fix $m \geq 1$. Take $K \geq 1$ such that the function $\psi$ from assumption \ref{A:1} satisfies $\psi(t) \leq K$ for all $t \geq 0$. Take also $\sm_{2mK} > 0$ as in \eqref{choice:sm} corresponding to the parameter $2 m K$. Let us fix $\sm \in (0, \sm_{2m K}]$ and the corresponding $\kappa \in (0,1)$, $\gdist \in \gfd$ and $T > 0$ for which \eqref{Wass:contr:gen} holds. Clearly, we may assume $T \in \mI$. Then, for any $t \in \mI$ with $t \geq T$, we may write $t = j T + s$, with $j \in \NN$, $j \geq 1$, and $s \in [0,T) \cap \mI$. Notice that $j T \in \mI$, for all $j \in \NN$.
	Thus, invoking \eqref{pt:sp:gap:kappa} $j$ times, it follows that for all $u_0, v_0 \in \gsp$ 
	\begin{align}\label{pt:sp:gap:all:t:0}
		\Wass_{\gdista} (P_t(u_0, \cdot), P_t(v_0, \cdot)) 
		&=
		\Wass_{\gdista} ( P_{(j-1)T + s}(u_0, \cdot) P_T, P_{(j-1)T + s} (v_0, \cdot) P_T) 
		\notag\\
		&\leq
		\kappa \Wass_{\gdist_{\sm/(2mK)}} (P_{(j-1)T + s}(u_0, \cdot), P_{(j-1)T + s}(v_0, \cdot))
		\notag\\
		&\leq
		\kappa \Wass_{\gdist_{\sm}} (P_{(j-1)T + s}(u_0, \cdot), P_{(j-1)T + s}(v_0, \cdot))
		\notag\\
		&\leq \ldots
		\leq \kappa^j \Wass_{\gdist_{\sm/(2mK)}} (P_s(u_0, \cdot), P_s(v_0, \cdot)). 
	\end{align}
	From a similar calculation as in \eqref{bound:W:P:a}, we have that for all $s \in [0, T) \cap \mI$
	\begin{align*}
		&\Wass_{\gdist_{\sm/(2mK)}}(P_s(u_0,\cdot), P_s(v_0, \cdot))
		\\
		&\qquad\qquad \leq
		\Wass_\gdist(\gP_s(u_0,\cdot), \gP_s(v_0,\cdot))^{1/2}  
		\left( \gP_s \exp\left(2 \frac{\sm}{mK} \gn{u_0}^2\right) \right)^{1/4}  	\left( \gP_s \exp\left(2 \frac{\sm}{mK} \gn{v_0}^2\right) \right)^{1/4}.
	\end{align*}
	Hence, recalling the analogous choice of $r$ in step $(iii)$ above, namely $r \coloneqq \sm/(2mK)$, with $\sm \in (0,\sm_{2mK}]$ and $\sm_{2mK}$ as in \eqref{choice:sm}, it follows from assumption \ref{A3:ii} along with assumption \ref{A:1} that
	\begin{align}\label{pt:sp:gap:all:t:1}
		&\Wass_{\gdist_{\sm/(2mK)}}(P_s(u_0,\cdot), P_s(v_0, \cdot))
		\notag\\
		&\qquad\qquad \leq
		C \exp\left( \frac{\sm}{4mK} \gn{u_0}^2 \right) \gdist(u_0, v_0)^{1/2} \exp\left(  \frac{\sm}{mK} \LC \right) \exp\left(\frac{\sm}{2mK} \psi(s) (\gn{u_0}^2 + \gn{v_0}^2 )\right)
		\notag\\
		&\qquad\qquad \leq
		C \exp\left( \frac{\sm}{4m} \gn{u_0}^2 \right) \gdist(u_0, v_0)^{1/2} \exp\left(  \frac{\sm}{mK} \LC \right) \exp\left(\frac{\sm}{2m}  (\gn{u_0}^2 + \gn{v_0}^2 )\right)
		\notag\\
		&\qquad\qquad\leq
		C  \gdist_{\sm/m} (u_0, v_0)
	\end{align}
	Plugging \eqref{pt:sp:gap:all:t:1} into \eqref{pt:sp:gap:all:t:0}, yields
	\begin{align*}
		\Wass_{\gdist_{\sm}}(P_t(u_0,\cdot), P_t(v_0, \cdot))
		\leq
		\kappa^j C  \gdist_{\sm/m} (u_0, v_0).
	\end{align*}
	Since $ t = jT + s < (j+1)T$, then $j > (t/T) - 1$ and, consequently, $\kappa^j < \kappa^{\frac{t}{T} - 1}$. Thus, 
	\begin{align*}
		\Wass_{\gdist_{\sm}}(P_t(u_0,\cdot), P_t(v_0, \cdot))
		\leq
		\kappa^{\frac{t}{T} - 1} C  \gdist_{\sm/m} (u_0, v_0)
		=
		\frac{C}{\kappa} e^{t \frac{\ln \kappa}{T}}  \gdist_{\sm/m} (u_0, v_0)
	\end{align*}
	for all $t \in \mI$ with $t \geq T$, and every $u_0, v_0 \in \gsp$. Therefore, \eqref{pt:sp:gap} holds with $C_1 = C/\kappa$ and $C_2 = - (\ln \kappa)/T$. This concludes the proof.
\end{proof}

\begin{Rmk}\label{rmk:exist:inv:meas}
	Clearly, if $\gP_t$, $t \in \mI$, is a Markov semigroup satisfying the assumptions of \cref{thm:gen:sp:gap}, and for which there exists an associated invariant measure $\mu_* \in \Pr(\gsp)$, i.e. $\mu_* \gP_t = \mu_*$ for all $t \in \mI$, then inequality \eqref{Wass:contr:gen} implies that $\mu_*$ must also be the unique invariant measure. Moreover, fixing $\gdist \in \gfd$ to be the distance-like function for which \eqref{Wass:contr:gen} holds, it follows similarly as in \cite[Corollary 4.11]{HairerMattinglyScheutzow2011} that if there exists a complete metric $\tilde{\rho}$ on $\gsp$ such that $\tilde{\rho} \leq \sqrt{\gdist}$ and such that $\gP_t$ is a Feller semigroup on $(\gsp,\tilde{\rho})$, then together with assumptions  \ref{A:1}, \ref{A:2} and \ref{A3:i} one can also guarantee the existence of such invariant measure. 
\end{Rmk}

\subsection{Uniform in time weak convergence}
\label{sec:gen:wk:conv}

The following general result provides the specific set of assumptions needed for achieving a long time bias estimate similar to \eqref{eq:finite:err:mom:intro}, though in a more general setting than considered in \cref{subsec:unif:contr:overview}. Indeed, our estimate is given with respect to the Wasserstein distance induced by the distance-like function $\gdista$ defined in \eqref{def:gdista} above, thus not necessarily a metric. The main assumptions are given by: a generalized triangle inequality satisfied by $\gdista$, \ref{H:1}; the existence of an invariant measure for each member of the given parametrized family of Markov kernels, \ref{H:2}; a finite-time error estimate for the approximating processes, \ref{H:4}; and a parameter-uniform Wasserstein contraction for the given family of Markov kernels. Under these assumptions, the proof follows essentially the same steps of argumentation leading to \eqref{eq:finite:err:mom:intro}.

\begin{Thm}\label{thm:gen:wk:conv}
	Let $\gsp$ be a separable Banach space with norm $\gn{\cdot}$. Fix a collection $\gfd$ of distance-like functions $\gdist: \gsp\times \gsp \to [0,1]$. 
	Consider a family of Markov kernels $\gP_t^\pmr$ on $\gsp$ indexed by $t \in \RR^+$ and a parameter $\pmr$ varying in some set $\gspmr$. 
	Assume the following set of conditions:
	\begin{enumerate}[label={(H\arabic*)}]
		\item\label{H:1} There exists a constant $\gamma \geq 1$ such that for every $\gdist \in \gfd$ and $\sm > 0$ the distance-like function $\gdista: \gsp \times \gsp \to \RR^+$ defined in \eqref{def:gdista} satisfies
		\begin{align}\label{gen:tri:ineq:gdista}
		\gdista(u, v) \leq C \left[ \gdist_{\gamma\sm}(u, w) + \gdist_{\gamma \sm}(w, v)\right]
		\end{align}
		for all $u, v, w \in \gsp$ and for some constant $C > 0$ (which may depend on $\gdist$, $\sm$ and $\gamma$).
		\item\label{H:2} For each $\pmr \in \gspmr$, there exists a probability measure $\mu_\pmr$ on $\gsp$ which is invariant under $\gP_t^\pmr$, $t \in \RR^+$.
		\item\label{H:4} 
		There exist $\pmr_0 \in \gspmr$ and $\sm' > 0$ such that for each $\sm \in (0, \sm']$ and for each $\gdist \in \gfd$ there exist functions $\pfunc: \gspmr \to \RR^+$, $\tgf: \RR^+ \to \RR^+$, and a measurable function $\gf: \gsp \to \RR^+ \cup \{\infty\}$ such that 
		\begin{align}\label{ineq:fin:time:0}
			\Wass_{\gdista} (\gP_t^\pmr(u, \cdot), \gP_t^{\pmr_0}(u, \cdot)) \leq \tgf(t) \gf(u) \pfunc(\pmr),
		\end{align}
		for all $\pmr \in \gspmr$, $t \in \RR^+$, and $u \in \gsp$. 
		\item\label{H:3} For every $m \geq 1$, there exists $\sm_m > 0$ such that for each $\sm \in (0, \sm_m]$ there exists $\gdist \in \gfd$, $T > 0$ and constants $C_1, C_2 > 0$ for which the following inequality holds:
		\begin{align}\label{sp:gap:pmr}
		\sup_{\pmr \in \gspmr \backslash \{\pmr_0\}} \Wass_{\gdista}(\mu \gP_t^\pmr, \tmu \gP_t^\pmr) \leq C_1 e^{-t C_2} \Wass_{\gdist_{\sm/m}} (\mu, \tmu)
		\end{align}
		for every $\mu, \tmu \in \Pr(\gsp)$ and $t \geq T$, with $\pmr_0 \in \gspmr$ as in \ref{H:4}.
	\end{enumerate}
	Then, there exists $\sm_* > 0$ such that for each fixed $\sm \in (0, \sm_*]$ there exists $\gdist \in \gfd$, $\tT > 0$ and a constant $C > 0$ for which it holds that 
	\begin{align}\label{W:inv:meas}
		\Wass_{\gdista}(\mu_\pmr, \mu_{\pmr_0}) \leq C \tgf(\tT) \pfunc(\pmr) \int_\gsp f(u) \mu_{\pmr_0}(du),
	\end{align}
	for every $\pmr \in \gspmr$.
\end{Thm}

\begin{Rmk}
	\label{rmk:gen:tri:cond}
    The crucial condition \eqref{gen:tri:ineq:gdista} of \cref{thm:gen:wk:conv} is not hard to verify in practice.  The main underlying condition for the collection 
	of distance like functions $\gfd$ is that they satisfy a generalized triangle inequality, namely that for $\rho \in \gfd$, we have the bound $\rho(u,w) \leq C (\rho(u,v) + \rho(u,w))$, 
	for a constant $C$ independent of any $u,v,w \in X$.     See \cref{prop:gen:tri:ineq:rhoesa} below for our precise formulation and \eqref{def:rhoes} in \cref{subsec:Wass:contr:NSE},
	\eqref{def:gfd:bdd} in \cref{sec:was:contr:bnd:dm} where we put
	this result into practice.
\end{Rmk}
\begin{proof}
	Due to assumptions \ref{H:1} and \ref{H:2}, together with \cref{prop:tri:ineq:Wass}, it follows that for each $\gdist \in \gfd$ and $\sm > 0$ there exists a constant $C > 0$ such that for all $t \in \RR^+$ and $\pmr \in \gspmr \backslash \{\pmr_0\}$
	\begin{align}\label{W:inv:1}
	\Wass_{\gdista}(\mu_\pmr, \mu_{\pmr_0}) 
	= \Wass_{\gdista}(\mu_\pmr \gP_t^\pmr, \mu_{\pmr_0} \gP_t^{\pmr_0})
	\leq C \left[ \Wass_{\gdist_{\gamma \sm}} (\mu_\pmr \gP_t^\pmr, \mu_{\pmr_0} \gP_t^\pmr) + \Wass_{\gdist_{\gamma \sm}} (\mu_{\pmr_0} \gP_t^\pmr, \mu_{\pmr_0} \gP_t^{\pmr_0} ) \right].
	\end{align}
	Now invoking assumption \ref{H:3} with $m = \gamma $ to estimate the first term in the right-hand side of \eqref{W:inv:1}, we obtain that for any fixed $\sm \in (0, \sm_\gamma/\gamma]$ and corresponding $\gdist \in \gfd$, $T > 0$ and constants $C_1, C_2 > 0$, we have
	\begin{align}\label{W:inv:2}
	\Wass_{\gdista} (\mu_\pmr, \mu_{\pmr_0}) 
	\leq
	C \left[ C_1 e^{-t C_2} \Wass_{\gdist_{\sm}} (\mu_\pmr, \mu_{\pmr_0}) + \Wass_{\gdist_{\gamma \sm}} (\mu_{\pmr_0} \gP_t^\pmr, \mu_{\pmr_0} \gP_t^{\pmr_0} ) \right]
	\end{align}
	for all $t \geq T$. Take $\tT \geq T$ such that 
	\begin{align}\label{cond:T}
	C C_1 e^{-\tT C_2} < \frac{1}{2}.
	\end{align}
	Thus, taking $t = \tT$ in \eqref{W:inv:2} and rearranging terms, we deduce that 
	\begin{align}\label{W:inv:meas:2}
		\Wass_{\gdista}(\mu_\pmr, \mu_{\pmr_0}) \leq 2 C \Wass_{\gdist_{\gamma\sm}}(\mu_{\pmr_0} \gP_{\tT}^\pmr, \mu_{\pmr_0} \gP_{\tT}^{\pmr_0}).
	\end{align}
	
	Moreover, similarly as in \eqref{Wass:convex}, it follows from \cite[Theorem 4.8]{villani2008} that
	\begin{align}\label{W:muinv:convex}
		\Wass_{\gdist_{\gamma \sm}} (\mu_{\pmr_0} \gP_{\tT}^\pmr, \mu_{\pmr_0} \gP_{\tT}^{\pmr_0})
		\leq
		\int_{\gsp \times \gsp} \Wass_{\gdist_{\gamma \sm}} (\gP_{\tT}^\pmr(u, \cdot), \gP_{\tT}^{\pmr_0}(v, \cdot)) \Gamma(du, d v),
	\end{align}
	for every coupling $\Gamma \in \Co(\mu_{\pmr_0}, \mu_{\pmr_0})$. Take $\Gamma(du, dv) = \delta_u(dv) \mu_{\pmr_0}(du)$, where $\delta_u \in \Pr(\gsp)$ denotes the Dirac measure concentrated at $u \in \gsp$. It is not difficult to check that $\Gamma \in \Co(\mu_{\pmr_0}, \mu_{\pmr_0})$. Therefore,
	\begin{align}\label{W:muinv:convex:2}
		\Wass_{\gdist_{\gamma \sm}} (\mu_{\pmr_0} \gP_{\tT}^\pmr, \mu_{\pmr_0} \gP_{\tT}^{\pmr_0})
		&\leq
		\int_{\gsp \times \gsp} \Wass_{\gdist_{\gamma \sm}} (\gP_{\tT}^\pmr(u, \cdot), \gP_{\tT}^{\pmr_0}(v, \cdot)) \delta_u(dv) \mu_{\pmr_0} (du)
		\notag\\
		&=
		\int_{\gsp} \Wass_{\gdist_{\gamma \sm}} (\gP_{\tT}^\pmr (u, \cdot), \gP_{\tT}^{\pmr_0}(u, \cdot)) \mu_{\pmr_0} (du).
	\end{align}
	Let us assume, if necessary, that $\sm$ varies in a smaller range so that inequality \eqref{ineq:fin:time:0} from \ref{H:4} holds with respect to $\gamma \sm$, namely
	\begin{align}\label{W:distga:u}
		\Wass_{\gdist_{\gamma \sm}} (\gP_t^\pmr(u, \cdot), \gP_t^{\pmr_0}(u, \cdot)) \leq \tgf(t) \gf(u) \pfunc(\pmr),
	\end{align}
	for all $\pmr \in \gspmr$, $t \in \RR^+$, $u \in \gsp$, and where we are fixing $\gdist$ as in \eqref{W:inv:2}.
	It thus follows from \eqref{W:inv:meas:2}, \eqref{W:muinv:convex:2}, and \eqref{W:distga:u} with $t = \tT$ that
	\begin{align}\label{W:inv:meas:2a}
		\Wass_{\gdista}(\mu_\pmr, \mu_{\pmr_0}) \leq 2 C \Wass_{\gdist_{\gamma \sm}} (\mu_{\pmr_0} \gP_{\tT}^\pmr, \mu_{\pmr_0} \gP_{\tT}^{\pmr_0})
		\leq
		2C \tgf(\tT) \pfunc(\pmr) \int_\gsp \gf(u) \, \mu_{\pmr_0} (du).
	\end{align}
	This shows \eqref{W:inv:meas}, and concludes the proof.
\end{proof}

Next we notice that a finite time error estimate as in \eqref{ineq:fin:time:0}, assuming $\pfunc(\pmr)\to 0$ as $\pmr \to \pmr_0$, combined with a uniform Wasserstein contraction for the approximating family $\{\gP_t^\pmr\}_{t \geq 0}$, $\pmr \in \gspmr \backslash \{\pmr_0\}$, as in \eqref{sp:gap:pmr} yields a Wasserstein contraction result for the limiting process $\{\gP_t^{\pmr_0}\}_{t \geq 0}$. This is made precise as follows.

\begin{Lem}\label{lem:sp:gap:lim:sys}
Fix the same setting from \cref{thm:gen:wk:conv} and assume that hypotheses \ref{H:1}-\ref{H:3} hold. Regarding \ref{H:4}, suppose additionally that for each $\sm \in (0, \sm']$ and $\gdist \in \gfd$ the corresponding function $\pfunc$ is such that $\lim_{\pmr \to \pmr_0} \pfunc(\pmr) = 0$. Then, for every $m \geq 1$ there exists $\sm_m > 0$ such that for each $\sm \in (0, \sm_m]$ there exists $\gdist \in \gfd$, $\tT > 0$ and constants $\tC_1, \tC_2 > 0$ for which the following inequality holds:
\begin{align}\label{Wass:contr:lim:proc}
\Wass_{\gdista}(\mu \gP_t^{\pmr_0}, \tmu \gP_t^{\pmr_0}) \leq \tC_1 e^{-t \tC_2} \Wass_{\gdist_{\sm/m}} (\mu, \tmu)
\end{align}
for every $t \geq \tT$ and all $\mu, \tmu \in \Pr(\gsp)$ satisfying
\begin{align}\label{cond:gf:mu:tmu}
\int_\gsp \gf(u) \mu(d u) + \int_\gsp \gf(u) \tmu(d u) < \infty,
\end{align}
where $\gf$ is the function from \ref{H:4}.
\end{Lem}
\begin{proof}
Fix any $m \geq 1$ and $\mu$, $\tmu$ satisfying \eqref{cond:gf:mu:tmu}. Invoking \ref{H:1} and \cref{prop:tri:ineq:Wass} twice, we obtain
\begin{align}\label{Wass:mu:tmu:P0}
\Wass_{\gdista}(\mu \gP_t^{\pmr_0}, \tmu \gP_t^{\pmr_0})
\leq C \left[ \Wass_{\gdist_{\gamma \sm}}(\mu \gP_t^{\pmr_0}, \mu \gP_t^\pmr) + \Wass_{\gdist_{\gamma^2 \sm}}(\mu \gP_t^\pmr, \tmu \gP_t^\pmr)  +  \Wass_{\gdist_{\gamma^2 \sm}}(\tmu \gP_t^\pmr, \tmu \gP_t^{\pmr_0})\right]
\end{align}
for any $\pmr \in \gspmr \backslash\{\pmr_0\}$. Now we assume $\sm > 0$ is sufficiently small, then proceed as in \eqref{W:muinv:convex}-\eqref{W:muinv:convex:2} and invoke \ref{H:4} to estimate the first and third terms in the right-hand side of \eqref{Wass:mu:tmu:P0}, and \ref{H:3} to estimate the second term. It thus follows that for any such $\sm$
there exists $\gdist \in \gfd$ and $T > 0$ such that 
\begin{align*}
\Wass_{\gdista}(\mu \gP_t^{\pmr_0}, \tmu \gP_t^{\pmr_0})
\leq C \left[ \tgf(t) \pfunc(\pmr) \left( \int_\gsp  \gf(u) \mu(du) + \int_\gsp  \gf(u) \tmu(du) \right)   +  C_1 e^{-t C_2} \Wass_{\gdist_{\sm/m}} (\mu, \tmu) \right]
\end{align*}
for all $t \geq T$, where $C_1,C_2$ are the same as in \eqref{sp:gap:pmr}. Thus, taking the limit as $\pmr \to \pmr_0$ and recalling the assumptions that $\lim_{\pmr \to \pmr_0} \pfunc(\pmr) = 0$ and \eqref{cond:gf:mu:tmu}, we deduce \eqref{Wass:contr:lim:proc}.
\end{proof}

Next, we show that, under the same assumptions from \cref{thm:gen:wk:conv} together with some natural conditions on the functions appearing in the right-hand side of the finite-time error estimate \eqref{ineq:fin:time:0}, it follows that the given parametrized family of Markov kernels $\gP_t^\pmr$, $\pmr \in \gspmr$, converges uniformly in $t \geq 0$ towards $\gP_t^{\pmr_0}$ in the Wasserstein topology determined by $\Wass_{\gdista}$. 

\begin{Thm}	
\label{thm:gen:wk:conv:2}
	Fix the same setting from \cref{thm:gen:wk:conv} and assume that hypotheses \ref{H:1}-\ref{H:3} hold. Additionally, regarding \ref{H:4}, suppose that for each $\sm \in (0, \sm']$ and $\gdist \in \gfd$ the corresponding functions $\tgf$, $\pfunc$, and $\gf$ satisfy:
	\begin{enumerate}[label={(H\arabic*)}]
		\setcounter{enumi}{4}
		\item\label{H:5} $\tgf$ is continuous and strictly increasing;
		\item\label{H:6} $\pfunc$ is bounded, and $\lim_{\pmr \to \pmr_0} \pfunc(\pmr) = 0$;
		\item\label{H:7} $\int_\gsp \gf(u) \mu_{\pmr_0}(du) < \infty$, with $\mu_{\pmr_0}$ as in \ref{H:2}.
	\end{enumerate}  
	Then, there exists $\hat{\sm} > 0$ such that for each fixed $\sm \in (0,\hat{\sm}]$ there exists $\gdist \in \gfd$ for which the following inequality holds for every $\pmr \in \gspmr$ and $\mu \in \Pr(\gsp)$ with $\int_\gsp \gf(u) \mu(du) < \infty$:
	\begin{align}\label{gen:w:conv}
	\sup_{t \in \RR^+} \Wass_{\gdista}(\mu \gP_t^\pmr, \mu \gP_t^{\pmr_0}) 
	\leq
	\tilde{g}(\pmr) \left[ \Wass_{\gdist_{\gamma \sm}}(\mu, \mu_{\pmr_0}) + \int_\gsp \gf(u) \mu(du) + \int_\gsp \gf(u) \mu_{\pmr_0}(du) \right],
	\end{align}
	where, if $\tgf$ is bounded,
	\begin{align*}
		\tilde{\pfunc}(\pmr) = C \pfunc(\pmr),
	\end{align*}
	and if $\tgf$ is unbounded
	\begin{align}\label{def:tildeg}
		\tilde{\pfunc}(\pmr) = C \max \left\{ \exp \left( - C_2 \tgf^{-1}(C \pfunc(\pmr)^{-q}) \right), \pfunc(\pmr), \pfunc(\pmr)^{1-q} \right\},
	\end{align}
	for $\pfunc(\pmr) \neq 0$, and $\tilde{\pfunc}(\pmr) = 0$ otherwise. Here, $q$ is any fixed number in $(0,1)$, $C > 0$ is a constant which is independent of $\pmr$, and $C_2$ is the constant from \eqref{sp:gap:pmr}. 
	
	Consequently, if additionally $\Wass_{\gdist_{\gamma \sm}}(\mu, \mu_{\pmr_0}) < \infty$, then
	\begin{align}\label{gen:w:conv:2}
		\lim_{\pmr \to \pmr_0} \sup_{t \in \RR^+} \Wass_{\gdista}(\mu \gP_t^\pmr, \mu \gP_t^{\pmr_0}) = 0.
	\end{align}
\end{Thm}
\begin{proof}

From \cref{prop:tri:ineq:Wass}, along with assumptions \ref{H:1} and \ref{H:2}, we obtain that for every $\mu \in \Pr(\gsp)$, $\sm > 0$, $\gdist \in \gfd$, $t \in \RR^+$ and $\pmr \in \gspmr \backslash \{\pmr_0\}$
\begin{align}
&\Wass_{\gdista}(\mu \gP_t^\pmr, \mu \gP_t^{\pmr_0}) 
\leq C \left[ \Wass_{\gdist_{\gamma \sm}}(\mu \gP_t^\pmr,  \mu_\pmr \gP_t^\pmr)  +   \Wass_{\gdist_{\gamma \sm}} (\mu_\pmr, \mu \gP_t^{\pmr_0}) \right]
\label{w:conv:1}
\\ 
&\quad\leq C \left[ \Wass_{\gdist_{\gamma^2 \sm}}(\mu \gP_t^\pmr, \mu_{\pmr_0} \gP_t^\pmr) +  \Wass_{\gdist_{\gamma^2 \sm}}(\mu_{\pmr_0} \gP_t^\pmr, \mu_\pmr \gP_t^\pmr) +  \Wass_{\gdist_{\gamma^2 \sm}} (\mu_\pmr, \mu_{\pmr_0}) +  \Wass_{\gdist_{\gamma^2 \sm}} (\mu_{\pmr_0} \gP_t^{\pmr_0}, \mu \gP_t^{\pmr_0}) \right]. \notag
\end{align}
Now we invoke assumption \ref{H:3} to estimate the first and second terms in the right-hand side of \eqref{w:conv:1}, \eqref{W:inv:meas} from \cref{thm:gen:wk:conv} to estimate the third term, and \cref{lem:sp:gap:lim:sys} to estimate the fourth term. Here we notice that we can apply \ref{H:3} and \cref{lem:sp:gap:lim:sys} with any choice of $m \geq 1$ to estimate the first, second, and fourth terms. But for estimating the third term via \eqref{W:inv:meas}, as we recall from the proof of \cref{thm:gen:wk:conv}, we must invoke \ref{H:3} with $m= \gamma$. Since the distance-like function $\gdist \in \gfd$ for which \eqref{sp:gap:pmr} in \ref{H:3} and \eqref{Wass:contr:lim:proc} in \cref{lem:sp:gap:lim:sys} hold depends in particular on the choice of $m$, we thus also estimate the first, second, and fourth terms under $m=\gamma$. This yields that for $\sm > 0$ sufficiently small there exists $\gdist \in \gfd$, $T > 0$, and constants $C_1, C_2 > 0$ such that
\begin{align}\label{w:conv:2}
	\Wass_{\gdista}(\mu \gP_t^\pmr, \mu \gP_t^{\pmr_0}) 
	&\leq C \left[ C_1 e^{-t C_2} \Wass_{\gdist_{\gamma \sm}}(\mu, \mu_{\pmr_0}) + C_1 e^{-t C_2} \Wass_{\gdist_{\gamma \sm}}(\mu_{\pmr_0}, \mu_\pmr) +  \Wass_{\gdist_{\gamma^2 \sm}}(\mu_\pmr, \mu_{\pmr_0})  \right]
	\notag\\
	&\leq C \left[ e^{-t C_2} \Wass_{\gdist_{\gamma \sm}}(\mu, \mu_{\pmr_0}) + \Wass_{\gdist_{\gamma^2 \sm}}(\mu_\pmr, \mu_{\pmr_0})  \right] \notag\\
	&\leq C \left[ e^{-t C_2} \Wass_{\gdist_{\gamma \sm}}(\mu, \mu_{\pmr_0}) + \pfunc(\pmr) \int_\gsp f(u) \mu_{\pmr_0}(du) \right],
\end{align}
for every $t \geq T$, $\pmr \in \gspmr$, and $\mu \in \Pr(\gsp)$ such that $\int_\gsp \gf(u) \mu(du) < \infty$. Here, $C>0$ is a constant depending on $\sm, T$ which is independent of $t, \pmr$.

On the other hand, proceeding similarly as in \eqref{W:muinv:convex}-\eqref{W:inv:meas:2a}, we obtain directly from assumption \ref{H:4} that, by taking $\sm$ smaller if necessary, 
\begin{align}\label{w:conv:4}
\Wass_{\gdista}(\mu \gP_t^\pmr, \mu \gP_t^{\pmr_0}) 
\leq
\tgf(t) \pfunc(\pmr) \int_\gsp f(u) \mu(du),
\end{align}
for all $\pmr \in \gspmr$ and $t \in \RR^+$. 
We now combine inequalities \eqref{w:conv:2} and \eqref{w:conv:4} to yield the desired uniform in $t \in \RR^+$ estimate \eqref{gen:w:conv}. 

Let us first suppose that $\tgf$ is a bounded function. In this case, it follows immediately from \eqref{w:conv:4} that
\begin{align*}
	\sup_{t \in \RR^+} \Wass_{\gdista}(\mu \gP_t^\pmr, \mu \gP_t^{\pmr_0}) \leq C \pfunc(\pmr) \int_\gsp f(u) \mu(du),
\end{align*}
 for every $\pmr \in \gspmr$ and for some constant $C > 0$, as desired.

Now let us assume that $\tgf$ is unbounded. Fix $q \in(0,1)$ and let $t_* > 0$ such that $\tgf(t_*) \neq 0$. Let $\tilde \tgf(\tau) \coloneqq \tgf(\tau)/ \tgf(t_*)$, $\tau \in \RR^+$, and $\tilde\pfunc(\pmr) \coloneqq \pfunc(\pmr)/ \overline{\pfunc}$, $\pmr \in \gspmr$, where $\overline{\pfunc} \coloneqq \sup_{\pmr \in \gspmr} \pfunc(\pmr)$. Here we recall our assumptions that $\tgf$ is strictly increasing and continuous, and $\pfunc$ is bounded, which implies that $\tilde{\tgf}$ is also strictly increasing and continuous, and $\overline{\pfunc} < \infty$. 

Fix $\pmr \in \gspmr$ and assume first that $\pfunc(\pmr) \neq 0$, so that $\tilde{\pfunc}(\pmr) \neq 0$. Notice that $\tilde{\tgf}(t_*) = 1$, and $\tilde{\pfunc}(\pmr) \leq 1$, so that $\tilde{\tgf}(t_*) \leq \tilde{\pfunc}(\pmr)^{-q}$. Since $\tilde{\tgf}: \RR^+ \to \RR^+$ is continuous and unbounded, there exists $\tau_* \in \RR^+$ such that $\tilde{\tgf}(\tau_*) = \tilde{\pfunc}(\pmr)^{-q}$. And since $\tgf$, $\tilde{\tgf}$ are strictly increasing, then their corresponding inverses $\tgf^{-1}$, $\tilde{\tgf}^{-1}$ are well-defined, so that $\tau_* = \tilde{\tgf}^{-1}(\tilde{\pfunc}(\pmr)^{-q}) = \tgf^{-1}(\pfunc(\pmr)^{-q} \tgf(t_*)/\overline{\pfunc})$. From \eqref{w:conv:4}, we thus obtain that for every $t \leq \tau_*$
\begin{align}\label{w:conv:5}
	\Wass_{\gdista}(\mu \gP_t^\pmr, \mu \gP_t^{\pmr_0}) 
	\leq
	\tgf(\tau_*) \pfunc(\pmr)  \int_\gsp f(u) \mu(du)
	&= \tilde{\tgf}(\tau_*) \tilde{\pfunc}(\pmr) \tgf(t_*) \overline{\pfunc} \int_\gsp f(u) \mu(du) \notag \\
	&= \tilde{\pfunc}(\pmr)^{1-q} \tgf(t_*) \overline{\pfunc} \int_\gsp f(u) \mu(du) \notag \\
	&\leq C \pfunc(\pmr)^{1-q} \int_\gsp f(u) \mu(du),
\end{align}
and for every $t \leq T$
\begin{align}\label{w:conv:5b}
	\Wass_{\gdista}(\mu \gP_t^\pmr, \mu \gP_t^{\pmr_0}) 
	\leq
	\tgf(T) \pfunc(\pmr) \int_\gsp f(u) \mu(du).
\end{align}

On the other hand, it follows from \eqref{w:conv:2} that for every $t > \max\{\tau_*,T\}$
\begin{multline}\label{w:conv:6}
	\Wass_{\gdista}(\mu \gP_t^\pmr, \mu \gP_t^{\pmr_0}) 
	\leq
	C \left[  e^{-\tau_* C_2} \Wass_{\gdist_{\gamma \sm}}(\mu, \mu_{\pmr_0}) +  \pfunc(\pmr) \int_\gsp f(u) \mu_{\pmr_0}(du) \right]
	\\
	\leq 
	C \max \{  \exp \left( - C_2 \tgf^{-1}(\pfunc(\pmr)^{-q} \tgf(t_*)/\overline{\pfunc}) \right),  \pfunc(\pmr)\}
	\left[  \Wass_{\gdist_{\gamma \sm}}(\mu, \mu_{\pmr_0}) +  \int_\gsp f(u) \mu_{\pmr_0}(du)  \right].
\end{multline}
From \eqref{w:conv:5}, \eqref{w:conv:5b} and \eqref{w:conv:6} we conclude that 
\begin{multline}\label{gen:w:conv:0}
	\sup_{t \in \RR^+} \Wass_{\gdista}(\mu \gP_t^\pmr, \mu \gP_t^{\pmr_0}) 
	\leq
	C \max \{ \exp \left( - C_2 \tgf^{-1}(\pfunc(\pmr)^{-q}  \tgf(t_*)/\overline{\pfunc}) \right) ,  \pfunc(\pmr), \pfunc(\pmr)^{1-q}\} \\
	\cdot \left[  \Wass_{\gdist_{\gamma \sm}}(\mu, \mu_{\pmr_0}) +  \int_\gsp f(u) \mu_{\pmr_0}(du) + \int_\gsp f(u) \mu(du)  \right],
\end{multline}
for all $\pmr \in \gspmr$ such that $\pfunc(\pmr) \neq 0$. Further, if $\pfunc(\pmr) = 0$, then it follows directly from \eqref{w:conv:4} that 
\begin{align*}
\Wass_{\gdista}(\mu \gP_t^\pmr, \mu \gP_t^{\pmr_0})= 0 \quad \mbox{ for all } t \in \RR^+.
\end{align*}
This concludes the proof of \eqref{gen:w:conv}. Finally, the validity of \eqref{gen:w:conv:2} is clear under \eqref{gen:w:conv} and the additional condition $\Wass_{\gdist_{\gamma \sm}}(\mu, \mu_{\pmr_0}) < \infty$.
\end{proof}

\begin{Rmk}\label{rmk:exp:growth}
	In the particular case that assumption \ref{H:4} holds with $\tgf(t) = \tC e^{t C'}$ for some positive constants $\tC, C'$, it follows that, under assumptions \ref{H:1}-\ref{H:4}, inequality \eqref{gen:w:conv} holds with 
	\begin{align*}
		\tilde{\pfunc}(\pmr) = C \max \{\pfunc(\pmr)^{\frac{C_2 q}{C'}}, \pfunc(\pmr), \pfunc(\pmr)^{1-q}\}
	\end{align*} 
	for any fixed $q < 1$ and for some constant $C > 0$. Indeed, this follows directly from the general expression of $\tilde{\pfunc}(\pmr)$ in \eqref{def:tildeg} for this specific case of $\tgf$. This particular situation appears in the application to a numerical discretization of the 2D SNSE presented in \cref{sec:app:SNSE} below.
\end{Rmk}

To conclude this section, we have the following immediate corollary of \cref{thm:gen:wk:conv:2} yielding uniform-in-time weak convergence for stochastic processes associated to the Markov kernels $\gP_t^\pmr$, $\pmr \in \gspmr$, with respect to Lipschitz test functions.

\begin{Cor}\label{cor:gen:wk:conv:obs}
	Fix the same setting and assumptions from \cref{thm:gen:wk:conv:2}. For each $u_0 \in \gsp$ and $\pmr \in \gspmr$, let $u_\pmr(t; u_0)$, $t \in \RR^+$, be a stochastic process such that $\mL(u_\pmr(t; u_0)) = \gP_t^\pmr(u_0,\cdot)$ for every $t \in \RR^+$, where $\mL(u_\pmr(t; u_0))$ denotes the law of $u_\pmr(t; u_0)$.
	Then, there exists $\hat{\sm} > 0$ such that for each $\sm \in (0, \hat{\sm}]$ there exists $\gdist \in \gfd$ for which the following inequality holds for every $\pmr \in \gspmr$, $u_0 \in \gsp$ such that $\gf(u_0) < \infty$, and every $\gdista$-Lipschitz function $\varphi: \gsp \to \RR$ with Lipschitz constant $L_\varphi$:
	\begin{align}\label{E:phi:diff:u:upmr}
		\sup_{t \in \RR^+} \left| \bE \left[ \varphi( u_\pmr(t;u_0)) -  \varphi(u_{\pmr_0}(t;u_0)) \right] \right|
		\leq
		L_{\varphi} \tilde{g}(\pmr) \left[ \Wass_{\gdist_{\gamma \sm}}(\delta_{u_0}, \mu_{\pmr_0}) +  \gf(u_0) + \int_\gsp \gf(u) \mu_{\pmr_0}(du) \right],
	\end{align}
	where $\tilde{g}(\pmr)$ is as given in \eqref{def:tildeg}. 
	
	Consequently, if $\Wass_{\gdist_{\gamma \sm}}(\delta_{u_0}, \mu_{\pmr_0}) < \infty$ then
	\begin{align}\label{lim:E:phi:diff:u:upmr}
		\lim_{\pmr \to \pmr_0} \sup_{t \in \RR^+} \left| \bE \left[ \varphi( u_\pmr(t;u_0)) -  \varphi(u_{\pmr_0} (t;u_0)) \right] \right| = 0.
	\end{align}
\end{Cor}
\begin{proof}
Let $\hat{\sm} > 0$ such that \eqref{gen:w:conv} holds for each $\sm \in (0,\hat{\sm}]$ and corresponding $\gdist \in \gfd$, and let us fix any such $\sm$ and $\gdist$. Fix also $\pmr \in \gspmr$, $u_0 \in \gsp$ such that $\gf(u_0) < \infty$, with $\gf$ as in \eqref{ineq:fin:time:0}, and let $\varphi: \gsp \to \RR$ be a $\gdista$-Lipschitz function with Lipschitz constant denoted as $L_\varphi$. 
Thus, for every $t \in \RR^+$ and coupling $\Gamma \in \Co(\gP_t^\pmr (u_0, \cdot), \gP_t^{\pmr_0}(u_0, \cdot))$, we have
\begin{align*}
\left| \bE \left[ \varphi ( u_\pmr(t;u_0)) -  \varphi (u_{\pmr_0}(t;u_0)) \right] \right|
&=
\left| \int_\gsp \varphi(u) \gP_t^\pmr(u_0, du) - \int_\gsp \varphi(\tilde{u}) \gP_t^{\pmr_0} (u_0, d\tilde{u}) \right| \\
&=
\left| \int_\gsp \left[ \varphi(u) - \varphi(\tilde{u}) \right] \Gamma(du, d\tilde{u}) \right|
\leq
L_\varphi \int_\gsp \gdista(u, \tilde{u}) \Gamma(du, d\tilde{u}).
\end{align*}
Taking the infimum over $\Gamma \in \Co(\gP_t^\pmr (u_0, \cdot), \gP_t^{\pmr_0}(u_0, \cdot))$, we deduce that
\begin{align}\label{ineq:obs:Wass}
\left| \bE \left[ \varphi ( u_\pmr(t;u_0)) -  \varphi (u_{\pmr_0}(t;u_0)) \right] \right|
\leq
L_\varphi  \Wass_{\gdista} (\gP_t^\pmr(u_0, \cdot), \gP_t^{\pmr_0}(u_0, \cdot)).
\end{align}

Next, we take the supremum over $t \in \RR^+$ and invoke \cref{thm:gen:wk:conv:2} to further estimate the right-hand side. It thus follows from \eqref{gen:w:conv} with $\mu = \delta_{u_0}$ that
\begin{align*}
	\sup_{t \in \RR^+} \left| \bE \left[ \varphi ( u_\pmr(t;u_0)) -  \varphi (u_{\pmr_0}(t;u_0)) \right] \right|
	\leq
	L_\varphi \tilde{g}(\pmr) \left[ \Wass_{\gdist_{\gamma \sm}}(\delta_{u_0}, \mu_{\pmr_0}) +  \gf(u_0) + \int_\gsp \gf(u) \mu_{\pmr_0}(du) \right],
\end{align*}
with $\tilde{g}(\pmr)$ as given in \eqref{def:tildeg}. This shows \eqref{E:phi:diff:u:upmr}. Clearly, \eqref{lim:E:phi:diff:u:upmr} follows immediately from \eqref{E:phi:diff:u:upmr} and the assumption that $\Wass_{\gdist_{\gamma \sm}}(\delta_{u_0}, \mu_{\pmr_0}) < \infty$. This concludes the proof.
\end{proof}

\section{Numerical approximation of the 2D stochastic Navier-Stokes equations}\label{sec:app:SNSE}

We now turn to the application of the abstract results from the previous section to the 2D stochastic Navier-Stokes equations (SNSE) and a corresponding space-time numerical discretization. In \cref{subsec:prelim:NSE}, we introduce some preliminary material regarding the form of the 2D stochastic Navier-Stokes equations that we consider here, along with the specific space-time discretization to be analyzed. In \cref{subsec:Wass:contr:NSE}, we verify the general set of assumptions from \cref{thm:gen:sp:gap} for a suitable class $\gfd$ of distance-like functions, defined in \eqref{def:rhoes} below, to prove Wasserstein contraction for the Markov semigroup generated by this discretization. Here, as mentioned above in the Introduction, we emphasize that the contraction coefficients obtained for the discretized system are \emph{independent} of any discretization parameters. This fact is crucial for obtaining a weak convergence result for the numerical scheme as a consequence of \cref{thm:gen:wk:conv:2}, which we present later in \cref{subsec:wk:conv:NSE}. We also provide in \cref{subsec:wk:conv:NSE} an estimate of the bias between the long time statistics of the discrete system and the continuous one as an application of \cref{thm:gen:wk:conv}. Before in \cref{subsubsec:fin:tim:err:L2}, we present some pathwise finite-time error estimates that are used to verify the required assumption \ref{H:4} from \cref{thm:gen:wk:conv} and \cref{thm:gen:wk:conv:2} for these last two results.

\subsection{Mathematical setting and moment bounds}\label{subsec:prelim:NSE}

\subsubsection{Two-dimensional stochastic Navier-Stokes equations}\label{subsubsec:2DSNSE}

Let $\bT^2 \simeq (\RR/2\pi \ZZ)^2$ be the two-dimensional torus. We consider the homogeneous Lebesgue space $\dL^2 = \dL^2(\bT^2) = \{ \xi \in L^2(\bT^2) \,:\, \int_{\bT^2} \xi(x)dx = 0\}$, endowed with the standard inner product and norm of $L^2(\bT^2)$, which we denote by $(\cdot,\cdot)$ and $|\cdot|$, respectively. Recall that any function $\xi \in \dL^2$ can be written as the Fourier expansion $\xi(x) = \sum_{\kappa \in \ZZ^2} \hat{\xi}(\kappa) e^{i \kappa \cdot x}$.

We also consider, for each $s\geq 0$, the homogeneous Sobolev space $\dH^s = \dH^s(\bT^2) = \{\xi \in \dL^2(\bT^2) \,:\, \|\xi\|_{\dH^s} < \infty\}$, where $\|\cdot\|_{\dH^s}$ is the norm induced by the inner product $(\cdot,\cdot)_{\dH^s}$, given by
\begin{align*}
	(\xi_1, \xi_2)_{\dH^s} = (2 \pi)^2 \sum_{\kappa \in \ZZ^2\backslash \{\mathbf{0}\}} |\kappa|^{2s}  \hat{\xi}_1(\kappa)\overline{\hat{\xi}_2(\kappa)}, \quad \hat{\xi}(\kappa) \coloneqq (2\pi)^{-2} \int_{\bT^2} e^{-i\kappa\cdot x}\xi(x) dx,
\end{align*}
where $\overline{\,\cdot\,}$ denotes complex conjugation. Clearly, for every $s_1 < s_2$, we have $\dH^{s_2} \subset \dH^{s_1}$. Moreover, note that $\dH^0$ coincides with $\dL^2$, with $|\xi| = \|\xi\|_{\dH^0}$. Also, $|\nabla \xi| = \|\xi\|_{\dH^1}$ and $|\Delta \xi| = \|\xi\|_{\dH^2}$.

Fix a stochastic basis $\mS = (\Omega, \mF, \{\mF_t\}_{t \geq 0}, \bP, \{W^k\}_{k=1}^d)$, i.e. a filtered probability space equipped with a finite family $\{W^k\}_{k=1}^d$ of standard independent real-valued Brownian motions on $\Omega$ that are adapted to the filtration $\{\mF_t\}_{t \geq 0}$. We consider the stochastically forced 2D Navier-Stokes equations (SNSE) in vorticity form in $\bT^2$ and driven by a white in time and colored in space additive noise, namely
\begin{align}\label{2DSNSEv}
\rd \xi + (- \nu \Delta \xi + \bu \cdot \nabla \xi ) \rd t = \sum_{k=1}^d \sigma_k \rd W^k, \quad \bu = \mK \ast \xi,
\end{align}
where $\xi = \xi(\bx,t)$, $(\bx, t) \in \bT^2 \times (0,\infty)$, represents the unknown random vorticity field; $\bu = \bu(\bx,t)$ represents the random velocity field, which is determined from the vorticity through the Biot-Savart kernel $\mK$ in \eqref{2DSNSEv}, so that $\nabla^\perp \cdot\bu = (-\partial_y, \partial_x)\cdot (u_1,u_2) = \xi$ and $\nabla \cdot \bu = 0$, see e.g. \cite{MajdaBertozzi}. Moreover, 
$\sigma_1, \ldots, \sigma_d$ are given functions in $\dL^2$. We will sometimes use the abbreviated notation $\sigma dW$ for $\sum_{k=1}^d \sigma_k d W^k$. Also, we assume that \eqref{2DSNSEv} is in nondimensional form, so that the parameter $\nu$ equals $Re^{-1}$, where $Re$ denotes the Reynolds number associated to the fluid flow. 

Equation \eqref{2DSNSEv} is sometimes also written in the following convenient functional form
\begin{align}\label{snse:func}
	\rd \xi + (\nu A \xi + B(\xi,\xi)) \rd t = \sigma d W, \quad \bu = \mK \ast \xi,
\end{align}
where $A = (-\Delta): \dH^2 \to \dL^2$,
and $B: \dH^1 \times \dH^1 \to (\dH^1)'$ is the bilinear mapping defined as $B(\xi,\xi) = \bu \cdot \nabla \xi =  (\mK \ast \xi) \cdot \nabla \xi $. Here, $(\dH^1)'$ denotes the dual space of $\dH^1$. For each $s \geq 0$, we define the corresponding power of A as $A^s: D(A^s) \to \dL^2$, given by $A^s \xi (x)= \sum_{\kappa \in \ZZ^2\backslash\{\mathbf{0}\}} |\kappa|^{2s} \hat{\xi}(\kappa) e^{i\kappa\cdot x}$, where $D(A^s) = \dH^{2s}$. Notice that $|A^s\xi| = \|\xi\|_{\dH^{2s}}$. Further, we recall that $A$ is a positive and self-adjoint operator with compact inverse. As such, it possesses a nondecreasing sequence of positive eigenvalues $\{\lambda_k\}_{k \in \mathbb{N}}$ with $\lambda_k \sim k$ asymptotically, so that $\lambda_k \to \infty$ as $k \to \infty$, associated to a sequence of eigenfunctions $\{e_k\}_{k \in \mathbb{N}}$ that form an orthonormal basis of $\dL^2$.

Regarding the noise term in \eqref{2DSNSEv}, we adopt the following additional notation. For each $s \geq 0$, we denote by $\bdH^s$ the d-fold product of $\dH^s$, and define, for each $\sigma = (\sigma_1, \ldots, \sigma_d) \in \bdH^s$, $\|\sigma\|^2_{\dH^s} \coloneqq \sum_{k=1}^d \|\sigma_k\|^2_{\dH^s}$. Similarly, we consider $\bdL^2 = \bdH^0$ and denote $|\sigma| \coloneqq \|\sigma\|_{\dH^0}$ for all $\sigma \in \bdL^2$. We then set $\sigma W \coloneqq \sum_{k=1}^d \sigma_k W^k$, so that, for any $\sigma \in \bdH^s$, $\sigma W$ is a Brownian motion on $\dH^s$ with covariance operator $tQ_s$, where $Q_s: \dH^s \to \dH^s$ is given by 
\begin{align}\label{def:Q:TrQ}
Q_s \xi = \sum_{k=1}^d (\sigma_k , \xi)_{\dH^s} \sigma_k, \quad \xi \in \dH^s.
\end{align}
We notice that $Q_s$ is a compact and symmetric operator with $\tr(Q_s) = \| \sigma\|_{\dH^s}^2$, where we recall that $\tr(Q_s) = \sum_{j=1}^\infty (Q_s \tilde{e}_j, \tilde{e}_j)_{\dH^s}$ for any orthonormal basis $\{\tilde{e}_j\}_{j \geq 1}$ of $\dH^s$.

With a slight abuse of notation, we also regard a given $\sigma \in \bdH^s$ as a mapping $\sigma: \RR^d \to \dH^s$, defined as $\sigma (w_1, \ldots, w_d) = \sum_{k=1}^d \sigma_k w_k $ for all $(w_1, \ldots, w_d) \in \RR^d$. Clearly, $\sigma$ is thus a bounded linear operator on $\RR^d$ with operator norm bounded from above by $\|\sigma\|_{\dH^s}$. Moreover, we denote by $\sigma^{-1}: \range{\sigma} \to \RR^d$ its corresponding pseudo-inverse, which is a bounded operator.\footnote{Notice that since $\range{\sigma}$ is a finite dimensional subset of $\dL^2$, then it is closed. This implies that the pseudo-inverse $\sigma^{-1}$ is a bounded operator (see e.g. \cite[Theorem 3.8]{Sheffield1956}).}

In what follows, we will be interested in pathwise, i.e. probabilistically strong, solutions of \eqref{2DSNSEv}, which are defined with respect to a fixed stochastic basis $\mS = (\Omega, \mF, \{\mF_t\}_{t \geq 0}, \bP, \{W^k\}_{k=1}^d)$ as considered above. 
We have the following well-posedness result regarding this type of solutions.

\begin{Prop}\label{prop:wellposed:cont}
	Let $\mS = (\Omega, \mF, \{\mF_t\}_{t \geq 0}, \bP, \{W^k\}_{k=1}^d)$ be a stochastic basis. Then, given any sequence $\{\sigma_k\}_{k=1}^d$ in $\dL^2$ and any $\mF_0$-measurable random variable $\xi_0 \in L^2(\Omega, \dL^2)$, there exists a unique $\dL^2$-valued random process $\xi$ with
	\begin{align*}
		\xi \in L^2 (\Omega; L^2_{loc}([0,\infty);\dH^1) \cap C ([0,\infty); \dL^2)),
	\end{align*}
	which is $\mF_t$-adapted, solves \eqref{2DSNSEv} weakly in $\dL^2$, and satisfies the initial condition $\xi(0) = \xi_0$ almost surely. Moreover, $\xi$ depends continuously on the initial data, i.e. for each $\xi_0 \in \dL^2$, the mapping $\xi_0 \mapsto \xi(t; \xi_0, \{W^k\}_{k=1}^d)$ is continuous in $\dL^2$ for any $t \in [0,\infty)$ and any fixed realization $\{W^k(\cdot,\omega)\}_k$, $\omega \in \Omega$.
\end{Prop}

Within this additive noise setting, a proof of \cref{prop:wellposed:cont} is given by following the standard argument of defining a change of variables $\xi = \overline{\xi} + v$, where 
\begin{align*}
	\frac{d\overline{\xi}}{dt} - \nu \Delta \overline{\xi} + (\mK \ast (\overline{\xi} + v)) \cdot \nabla (\overline{\xi} + v) = 0, \quad \mbox{and }\,\,
	dv - \nu \Delta v dt = \sigma dW.
\end{align*}
For each realization of $v$, $\overline{\xi}$ thus satisfies a deterministic equation for which one can show well-posedness by following similar arguments as for the 2D Navier-Stokes equations, see e.g. \cite{ConstantinFoias88,Temam2001}. Whereas $v$ satisfies a linear SPDE whose well-posedness is well-established, see e.g. \cite{DaPratoZabczyk2014}. We remark, however, that  well-posedness has also been established under much more general noise settings, see e.g. \cite{MiRoz2004,GHZ2009}.

With the notation introduced in \cref{subsec:prelim} and \cref{prop:wellposed:cont}, we denote the transition function associated to \eqref{2DSNSEv} by $\Pcont = \Pcont(\xi_0,\cO)$, for each $t \geq 0$, initial point $\xi_0 \in \dL^2$ and Borel set $\cO \in \mB(\dL^2)$, defined as
\begin{align}\label{def:Pcont:ker}
	\Pcont(\xi_0, \cO) := \bP (\xi(t; \xi_0) \in \cO),
\end{align}
where $\xi(t;\xi_0)$, $t \geq 0$, is the unique solution of \eqref{2DSNSEv} satisfying $\xi(0) = \xi_0$ almost surely, in the sense given in \cref{prop:wellposed:cont}. The corresponding Markov semigroup $\Pcont$, $t \geq 0$, is defined for each $\varphi \in \mM_b(\dL^2)$ as
\begin{align}\label{def:Pcont}	
	\Pcont \varphi (\xi_0) := \bE \varphi(\xi(t;\xi_0)), \quad \xi_0 \in \dL^2.
\end{align}
Since $\xi(\cdot;\xi_0)$ is continuous with respect to the initial data $\xi_0$, it follows that $\{\Pcont\}_{t \geq 0}$ is also a Feller Markov semigroup in $\dL^2$. I.e., denoting by $\mC_b(\dL^2)$ the space of real-valued, bounded and continuous functions on $\dL^2$, we have $\Pcont \varphi \in \mC_b(\dL^2)$ for every $\varphi \in \mC_b(\dL^2)$.

We recall that existence of an invariant measure $\mu_*$ with respect to the semigroup $\Pcont$, $t \geq 0$, is a well-established result, in fact valid for much more general noise structures than specified in \eqref{2DSNSEv}, see e.g. \cite{Flandoli1994}. On the other hand, showing uniqueness of the invariant measure requires extra assumptions on the noise term, see e.g. \cite{FlaMas1995,DaPratoZabczyk1996,Mattingly1999,BKL2001,EMatSi2001,KukShi2001,BKL2002,KPS2002,KukShi2002,Mattingly2002,Mattingly2003,HaiMa2006,HairerMattingly2008,HM11,KuksinShirikyan12,Debussche2013,GHRichMa2017}. Following a similar assumption from previous works, here we consider that the number $d$ of stochastically forced directions in \eqref{2DSNSEv} is sufficiently large depending on the ``size'' of the  parameter $\nu$ and the coefficients $\sigma_k$, see \eqref{cond:K:sigma} below. Nevertheless, we expect that similar results regarding the space-time discretization \eqref{disc2DSNSEv} below would also hold for a degenerate type of stochastic forcing as considered in \cite{HaiMa2006,HairerMattingly2008}, albeit with respect to a possibly different class of distance functions than \eqref{def:rhoes}. This would however require more sophisticated Malliavin calculus techniques that we intend to pursue in future work.

Let us also recall a few basic inequalities and properties of the bilinear term $\bu \cdot \nabla \xi$ in \eqref{2DSNSEv}. 
For any divergence-free $\bu \in (\dH^1)^2$, it follows with integration by parts that
\begin{align}\label{nonlin:ort:0}
(\bu \cdot \nabla \xi, \txi) = - (\bu \cdot \nabla \txi, \xi) \quad \mbox{ for all } \xi, \txi \in \dH^1,
\end{align}
which implies the orthogonality property
\begin{align}\label{nonlin:ort}
(\bu \cdot \nabla \xi, \xi) = ((\mK \ast \xi)  \cdot \nabla \xi, \xi) = 0 \quad \mbox{ for all } \xi \in \dH^1.
\end{align}
Moreover, the following inequalities follow by standard arguments involving H\"older and interpolation inequalities:
\begin{align}\label{ineq:nonlin:a:0}
	|\left( (\mK \ast \xi_1) \cdot \nabla \xi_2, \xi_3 \right)|  \leq  c |\xi_1| |\nabla \xi_2| |\xi_3|^{1 - a } |\nabla \xi_3|^{ a },
\end{align}
\begin{align}\label{ineq:nonlin:b}
	|\left( (\mK \ast \xi_1) \cdot \nabla \xi_2, \xi_3 \right)|  \leq  c |\xi_1|^{1/2} |\nabla \xi_1|^{1/2} |\nabla \xi_2| |\xi_3| ,
\end{align}
\begin{align}\label{ineq:nonlin:d}
	|\left( \nabla \left[ (\mK \ast \xi_1) \cdot \nabla \xi_2 \right] , \nabla \xi_2 \right) |
	\leq 
	c |\xi_1|^{3/4} |\Delta \xi_1|^{1/4} |\nabla \xi_2| |\xi_2|^{1/4} |\Delta \xi_2|^{3/4},
\end{align}
for some positive absolute constant $c$, and for all $\xi$ such that the norms above make sense.

\begin{Rmk}
Throughout the next sections, we adopt the following convention regarding constants. With lower-case letters $c$, $\tilde{c}$, we denote a positive absolute constant, i.e. independent of any parameters whatsoever. Whereas with upper-case letters $C, \tilde{C}, C_0, C_1, C_2$, we denote a positive constant which depends \emph{at most} on the parameters $\nu, |\sigma|, |\nabla \sigma|, |A \sigma|$, the parameters $\varepsilon > 0$ and $s \in (0,1]$ from the definition of the family of distances in \eqref{def:rhoes} below, along with other parameters that are specific to certain statements. These will be made explicit within each statement. Most importantly, these constants will always be independent of any discretization parameters. Under this convention regarding their dependences, we allow the values of these constants to vary from line to line. 
\end{Rmk}

\subsubsection{Spectral Galerkin discretization}

We start by fixing some notation. As before, we denote the eigenvalues and eigenfunctions of $A = (-\Delta): \dH^2 \to \dL^2$ by $\{\lambda_k\}_{k \in \NN}$ and $\{e_k\}_{k \in \NN}$, respectively. Then, for each $N \in \mathbb{N}$, we denote by $\Pi_N: \dL^2 \to \dL^2$ the projection operator onto the subspace $\HN$ of $\dL^2$ given by the span of the first $N$ eigenfunctions of $(- \Delta)$. Therefore, $I - \Pi_N$ is the projection operator onto the complement space $(I - \Pi_N )\dL^2$. We have the following Poincar\'e-type inequality:
\begin{align}\label{ineq:Poincare:N}
|\nabla (I - \Pi_N) \xi|^2 \geq \lambda_{N+1} |(I - \Pi_N) \xi|^2 \quad \mbox{ for all } \xi \in \dL^2.
\end{align}
The spectral Galerkin in space approximation of \eqref{2DSNSEv} in $\HN$ is given by
\begin{align}\label{2DSNSEvGal}
\rd \xi_N 
+ [- \nu \Delta \xi_N + \Pi_N (\bu_N \cdot \nabla \xi_N )]\rd t =  \Pi_N \sigma d W, \quad \bu_N = \mK \ast \xi_N.
\end{align}

The existence and uniqueness of probabilistically strong solutions of \eqref{2DSNSEvGal} satisfying a given initial condition follows analogously to the proof of \cref{prop:wellposed:cont}. For completeness, we state this result below.

\begin{Prop}\label{prop:wellposed:Gal}
	Fix a stochastic basis $\mS = (\Omega, \mF, \{\mF_t\}_{t \geq 0}, \bP, \{W^k\}_{k=1}^d)$. Then, given any family $\{\sigma_k\}_{k=1}^d$ of functions in $\dL^2$ and any $\mF_0$-measurable random variable $\xi_0 \in L^2(\Omega, \dL^2)$, there exists a unique $\HN$-valued random process $\xn$ with
	\begin{align*}
		\xn \in L^2 (\Omega; C ([0,\infty); \HN)),
	\end{align*}
	which is $\mF_t$-adapted, solves \eqref{2DSNSEvGal} weakly in $\dL^2$, and satisfies the initial condition $\xi(0) = \Pi_N \xi_0$ almost surely. Moreover, $\xn$ depends continuously on the initial data, i.e. for each $\xi_0 \in \dL^2$, the mapping $\xi_0 \mapsto \xn(t; \Pi_N \xi_0, \{W^k\}_{k=1}^d)$ is continuous in $\dL^2$ for any $t \in [0,\infty)$ and any fixed realization $\{W^k(\cdot,\omega)\}_k$, $\omega \in \Omega$.
\end{Prop}

We next state a collection of results regarding solutions of the Galerkin system \eqref{2DSNSEvGal} as well of the limiting system \eqref{2DSNSEv} that will be particularly useful in \cref{subsubsec:fin:tim:err:L2} and \cref{subsec:wk:conv:NSE} below. 

The following proposition presents some further moment bounds for solutions of the Galerkin scheme \eqref{2DSNSEvGal} and the fully continuous system \eqref{2DSNSEv}. The proof follows from similar arguments as in \cite[Corollary 2.4.11, Proposition 2.4.12]{KuksinShirikyan12}, where for handling the nonlinear term in each case we invoke \eqref{nonlin:ort}, \eqref{ineq:nonlin:d}, and the following inequality which follows similarly as in \cite[Lemma 2.1.20]{KuksinShirikyan12}
\begin{align*}
|\left( (\mK \ast \xi) \cdot \nabla \xi, A^k \xi \right)|
\leq c |\xi|^{\frac{k+2}{2(k+1)}} \|\xi\|_{\dH^1}^{\frac{k+2}{2k}} \|\xi\|_{\dH^k}^{\frac{k(4k+1) - 2}{2k(k+1)}},
\end{align*}
for all $\xi \in \dH^k$, $k \geq 2$.

\begin{Prop}\label{prop:sup:xn:nabla}
	Fix any $\xi_0 \in \dH^1$ and $\sigma \in \bdH^1$. Let $\xn(t)$, $t \geq 0$, be the solution of \eqref{2DSNSEvGal} satisfying $\xn(0) = \Pi_N \xi_0$ almost surely. Then, for every $T > 0$ and $m \in \NN$, it holds 
	\begin{align}\label{ineq:sup:nabla:xn}
		\sup_{N \in \NN} \bE \sup_{t \in [0,T]} \left( |\nabla \xn(t)|^2 + \nu \int_0^t |\Delta \xn(s)|^2 d s \right)^m \leq C (1 + |\xi_0|^{4m} + |\nabla \xi_0|^{2m}),
	\end{align}
	for some constant $C = C(m,\nu,T,|\nabla \sigma|)$.
	
	Moreover, given any $\xi_0 \in \dL^2$ and $\sigma \in \bdH^k$, $k \in \ZZ^+$, it follows that for every $T > 0$ and $m \in \NN$ there exists $p = p(k)$ such that 
	\begin{align}\label{ineq:sup:xn:Hk}
		\sup_{N \in \NN} \bE \sup_{t \in [0,T]} \left( t^k \|\xn(t)\|_{\dH^k}^2 + \nu \int_0^t s^k \|\xn(s)\|_{\dH^{k+1}}^2 ds \right)^m \leq \tC \left(1 + |\xi_0|^{2mp} \right),
	\end{align}
	for some constant $\tC = \tC(m,k,\nu,T,\|\sigma\|_{\dH^k})$. More precisely, $p(0) = 1$, $p(1) = 2$, and $p(k) = (3k+1)(k+2)/(3k+2)$ for every $k \geq 2$.
	
	Furthermore, let $\xi(t)$, $t \geq 0$, be the solution of  \eqref{2DSNSEv} satisfying $\xi(0) = \xi_0$ almost surely. Then, analogous inequalities to \eqref{ineq:sup:nabla:xn} and \eqref{ineq:sup:xn:Hk} hold with $\xn(t)$ replaced by $\xi(t)$. 
\end{Prop}

The next two propositions provide, respectively, some exponential moment bounds, and exponential Lyapunov inequalities for systems \eqref{2DSNSEvGal} and \eqref{2DSNSEv}. The proofs are available within similar contexts in e.g. \cite{HaiMa2006,HairerMattingly2008,KuksinShirikyan12,Debussche2013,GH2014,GHRichMa2017}, while also following as entirely analogous continuous versions of the proofs of \cref{prop:exp:bounds:sup} and \cref{prop:Lyap:exp:disc} below.

\begin{Prop}\label{prop:exp:mom:cont}
	Fix any $\sigma \in \bdL^2$ and $\xi_0 \in \dL^2$. Fix also $N \in \NN$ and let $\xn(t)$, $t \geq 0$, be the solution of  \eqref{disc2DSNSEv} satisfying $\xn(0) = \Pi_N \xi_0$ almost surely. Then, for all $\sm \in \RR$ satisfying
	\begin{align}\label{condgamma0}
	0 < \sm \leq \frac{\nu }{2 |\sigma|^2},
	\end{align}
	the following inequality holds
	\begin{align}\label{E:sup:exp:xi:T}
	\bE \sup_{t \geq 0 } \exp \left( \sm |\xn(t)|^2 + \sm \nu  \int_0^t |\nabla \xn(s) |^2 \rd s  - \sm |\sigma|^2 t \right) 
	\leq
	2 \exp \left( \sm |\xi_0|^2 \right).  
	\end{align}
	Moreover, let $\xi(t)$, $t \geq 0$, be the solution of  \eqref{2DSNSEv} satisfying $\xi(0) = \xi_0$ almost surely. Then, an analogous inequality to \eqref{E:sup:exp:xi:T} holds with $\xn(t)$ replaced by $\xi(t)$. 
\end{Prop}

\begin{Prop}\label{prop:exp:mom:cont:2}
	Fix any $\sigma \in \bdL^2$ and $\xi_0 \in \dL^2$. Fix also $N \in \NN$ and let $\xn(t)$, $t \geq 0$, be the solution of  \eqref{disc2DSNSEv} satisfying $\xn(0) = \Pi_N \xi_0$ almost surely. Consider $\sm \in \RR$ satisfying \eqref{condgamma0}. Then, the following inequality holds
	\begin{align}\label{bE:Lyap:gal}
	\bE \exp \left( \sm |\xn(t)|^2 \right) 
	\leq 
	\exp \left( \sm \left( e^{-\nu t}  |\xi_0|^2 +  \frac{|\sigma|^2}{\nu } \right) \right)  \quad \mbox{ for all } t \geq 0.
	\end{align}
Moreover, let $\xi(t)$, $t \geq 0$, be the solution of  \eqref{2DSNSEv} satisfying $\xi(0) = \xi_0$ almost surely. Then, an analogous inequality to \eqref{bE:Lyap:gal} holds with $\xn(t)$ replaced by $\xi(t)$. 
\end{Prop}

The following result shows H\"older regularity in time for solutions of the Galerkin system \eqref{2DSNSEvGal}. We note that a similar result was shown in \cite[Lemma 2.3]{CarelliProhl2012} involving the velocity formulation of \eqref{2DSNSEv} subject to a suitable multiplicative noise structure, and resulting in H\"older regularity for the associated solution with respect to a weaker norm than presented here. A proof is included in \cref{sec:app:holder:reg}.

\begin{Thm}\label{thm:holder:reg}
	Fix any $\sigma \in \bdH^1$ and $\xi_0 \in \dH^2$. Let $\xn = \xn(t)$ be the solution of \eqref{2DSNSEvGal} satisfying $\xn(0) = \Pi_N \xi_0$ almost surely. Then, for every $T > 0$, $m \in \NN$ and $\tp \in (0,1/2)$, 
	\begin{align}\label{holder:reg:L2}
		\sup_{N \in \NN} \bE |\xn(t) - \xn(s)|^m \leq C |t - s|^{m \tp} (1 +  |\xi_0|^{4m} + |\nabla \xi_0|^{2m})
	\end{align}
	and
	\begin{align}\label{holder:reg:H1}
		\sup_{N \in \NN} \bE |\nabla \xn(t) - \nabla \xn(s)|^m \leq 
		C |t - s|^{m \tp} (1 +  |\xi_0|^{4m} + |\nabla \xi_0|^{2m} +  |A \xi_0|^m)
	\end{align}
	for all $s, t \in [0,T]$, where $C = C (m, \tp, T, \nu, |\sigma|, |\nabla \sigma|)$.
	Moreover, let $\xi(t)$, $t \geq 0$, be the solution of  \eqref{2DSNSEv} satisfying $\xi(0) = \xi_0$ almost surely. Then, analogous inequalities to \eqref{holder:reg:L2} and \eqref{holder:reg:H1} hold with $\xi(t)$ replaced by $\xi(t)$.
\end{Thm}

\subsubsection{Space-time discretization}\label{subsubsec:sp:time:disc}

We now introduce, for each fixed time step $\dt > 0$, a fully space-time discrete approximation of \eqref{2DSNSEv} 
given by a semi-implicit in time Euler discretization of the Galerkin system \eqref{2DSNSEvGal}, namely
\begin{align}
  \xfd^n = \xfd^{n-1} + \dt [\nu \Delta \xfd^n - \Pi_N(\bu_{N,\dt}^{n-1} &\cdot \nabla \xfd^n)] + \sum_{k=1}^d \Pi_N \sigma_k (W^k(t_n) - W^k(t_{n-1})),
                                 \label{disc2DSNSEv}\\
  \bu_{N,\dt}^{n-1} &= \mK \ast \xfd^{n-1},
\notag
\end{align}
where each $\xfd^n$ represents the approximation of $\xi_N$, and thus of $\xi$, at time $t_n = n \dt$, for all $n \in \NN$.

Since $\{W^k\}_{k=1}^d$ is a sequence of independent real-valued Brownian motions, we can write
\begin{align}\label{def:etank}
W^k(t_n) - W^k(t_{n-1}) \stackrel{\mL}{=} \eta_n^k \dt^{1/2}
\end{align}
for a sequence $\eta_n^k: \Omega \to \mathbb{R}$, $n \in \mathbb{N}$, $k = 1, \ldots, d$, of independent and identically distributed Gaussian random variables with mean zero and covariance $1$. 

As in \eqref{snse:func} above, we adopt the abbreviated notation
\begin{align*}
	\eta_n = (\eta_n^1, \ldots, \eta_n^d), \quad
	\Pi_N\sigma = (\Pi_N \sigma_1, \ldots, \Pi_N \sigma_d), \quad \mbox{ and }\,\, 
	\Pi_N\sigma \eta_n = \sum_{k=1}^d \Pi_N \sigma_k \eta_n^k,
\end{align*}
so that \eqref{disc2DSNSEv} is compactly written as
\begin{align}\label{disc2DSNSEv2}
\xfd^n = \xfd^{n-1} + \dt [\nu \Delta \xfd^n - \Pi_N(\bu_{N,\dt}^{n-1} \cdot \nabla \xfd^n)] + \Pi_N\sigma \eta_n \dt^{1/2},
\quad
\bu_{N,\dt}^{n-1} = \mK \ast \xfd^{n-1}.
\end{align}
We notice that, for each $n \in \mathbb{N}$, $ \sigma \eta_n$ is a Gaussian random variable in $\dL^2$ with zero mean and covariance operator given by $Q_0$ defined in \eqref{def:Q:TrQ}.

\vspace{0.3cm}
\noindent{\bf Remark on notation:} To avoid overburdening notation, in the subsequent sections we will frequently denote $\xfd^n$ and $\ufd^n$ simply as $\xi^n$ and $\bu^n$, respectively, for any $n \in \NN$.
\vspace{0.3cm}

The following proposition establishes pathwise well-posedness of \eqref{disc2DSNSEv} for a given initial data. Its proof follows by standard arguments, so we omit the details.

\begin{Prop}\label{prop:wellposed:full:disc}
	Let $\mS = (\Omega, \mF, \{\mF_t\}_{t \geq 0}, \bP, \{W^k\}_{k=1}^d)$ be a stochastic basis. Then, given any family $\{\sigma_k\}_{k=1}^d$ of functions in $\dL^2$ and any $\dL^2$-valued and $\mF_0$-measurable random variable $\xi_0 \in L^2(\Omega, \dL^2)$, there exists a unique $\HN$-valued discrete random process $\{\xfd^n\}_{n \in \ZZ^+}$ with $\xfd^n \in L^2(\Omega, \dL^2)$, for all $n \in \ZZ^+$, and which is $\{\mF_{t_n}\}_{n\in\mathbb{N}}$-adapted, solves \eqref{disc2DSNSEv} in $\dL^2$ and satisfies the initial condition $\xfd^0 = \Pi_N \xi_0$ almost surely. Moreover, $\xfd^n$ depends continuously on the initial data, i.e. for each $\xi_0 \in \dL^2$, the mapping $\xi_0 \mapsto \xfd^n(\Pi_N \xi_0, \{W^k\}_{k=1}^d)$ is continuous in $\dL^2$ for any $n \in \ZZ^+$ and any fixed realization $\{W^k(\cdot,\omega)\}_k$, $\omega \in \Omega$.
\end{Prop}

For each fixed time step $\dt > 0$ and number $N$ of Galerkin modes, we denote the Markov transition function associated to $n$ steps of the discrete scheme \eqref{disc2DSNSEv2} by $\Pdisc = \Pdisc(\xi_0, \cO)$, $\xi_0 \in \dL^2$, $\cO \in \mB(\dL^2)$. This is defined as 
\begin{align}\label{def:Pdisc}
	\Pdisc (\xi_0, \cO) := \bP (\xfd^n(\Pi_N \xi_0) \in \cO),
\end{align}
where $\xfd^n(\Pi_N \xi_0)$ is the unique solution of \eqref{2DSNSEv} starting from the initial datum $\Pi_N \xi_0$, in the sense given in \cref{prop:wellposed:full:disc}. The corresponding Markov semigroup $\Pdisc$, $n \in \ZZ^+$, is thus defined for each $\varphi \in \mM_b(\dL^2)$ as
\begin{align}\label{def:Pdisc:0}
	\Pdisc \varphi(\xi_0) := \bE \varphi (\xfd^n(\Pi_N \xi_0)), \quad \xi_0 \in \dL^2, \quad n \in \ZZ^+.
\end{align}
Similarly as pointed out in \cref{subsubsec:2DSNSE} for the Markov semigroup $\Pcont$, $t \geq 0$, it follows as a consequence of the continuity of the solution $\xfd^n(\Pi_N \xi_0)$ with respect to the initial datum $\Pi_N \xi_0$, guaranteed by \cref{prop:wellposed:full:disc} above, that $\Pdisc$, $n \in \ZZ^+$, is a Feller Markov semigroup in $\dL^2$.

\subsection{Discretization-uniform Wasserstein contraction}\label{subsec:Wass:contr:NSE}

In this section, we apply \cref{thm:gen:sp:gap} to show a Wasserstein contraction result for the Markov semigroup $\Pdisc$, $n \in \ZZ^+$, associated to the numerical scheme \eqref{disc2DSNSEv2}, defined in \eqref{def:Pdisc:0}, for any fixed parameters $N \in \NN$, $\dt > 0$. Within the setting of \cref{thm:gen:sp:gap}, we consider $(\gsp, \gn{\cdot}) = (\dL^2, |\cdot|)$, $\mI = \dt \ZZ^+$, and $\{\gP_t\}_{t \in \mI}$ given by $\Pdisc$, $n \in \ZZ^+$. Moreover, we consider the class of distance-like functions $\Lambda = \{ \rhoes \,:\, \varepsilon > 0, \, 0 < s \leq 1 \}$, with each $\rhoes$ defined as
\begin{align}\label{def:rhoes}
\rhoes(\xi, \txi) = 1 \wedge \frac{|\xi -\txi|^s}{\varepsilon}, \quad \xi, \txi \in \dL^2.
\end{align}
Here, in fact, each $\rhoes$ is an actual metric on $\dL^2$, as it can be easily verified. The parameter $\varepsilon$ is appropriately tuned so as to produce a local contraction in \eqref{ineq:contr:Wass} below, in view of assumption \ref{A3:i} from \cref{thm:gen:sp:gap}. Thus, in a certain sense, $\varepsilon$ can be understood as representing the small spatial scales in the dynamics specified by \eqref{disc2DSNSEv} and \eqref{2DSNSEv}, respectively. 

As in \eqref{def:gdista}, for each $a > 0$ we denote the corresponding Lyapunov-weighted version of $\rhoes$ by $\rho_{\varepsilon,s,a}$, defined as
\begin{align}\label{def:rhoesa}
\rho_{\varepsilon,s,a}(\xi, \txi) = \rhoes(\xi, \txi)^{1/2} \exp \left( a |\xi|^2 + a |\txi|^2 \right), \quad \xi, \txi \in \dL^2.
\end{align}
Moreover, we denote the Wasserstein-like extensions to $\Pr(\dL^2)$ corresponding to $\rhoes$ and $\rho_{\varepsilon,s,a}$, as defined in \eqref{def:Wass:rhoe}, by $\Wes$ and $\Wass_{\varepsilon,s,a}$, respectively.

The validity of assumptions \ref{A:1}-\ref{A:3} from \cref{thm:gen:sp:gap} is verified in \cref{prop:Lyap:exp:disc}, \cref{prop:small:Wass}, and \cref{prop:contr:Wass} below. This leads us to the Wasserstein contraction result \cref{thm:sp:gap:disc:SNSE} below, whose proof we postpone to the end of this section. With the purpose of later applying \cref{thm:gen:wk:conv:2} to yield uniform weak convergence of the numerical scheme \eqref{disc2DSNSEv} towards the continuous system \eqref{2DSNSEv}, we state \cref{thm:sp:gap:disc:SNSE} in terms of a suitable continuous family of Markov kernels corresponding to the discrete semigroup $\Pdisc$, $n \in \ZZ^+$. Namely, we define for each $t \in \RR^+$
\begin{align}\label{def:Pdisct}
\Pdisct \coloneqq \Pdisc \quad \mbox{ if } t \in [n\dt, (n+1)\dt), \,\, n \in \ZZ^+.
\end{align}
We notice that the family $\Pdisct$, $t \in \RR^+$, may not define a Markov semigroup. However, the semigroup property is not required in the general weak convergence result, \cref{thm:gen:wk:conv:2}.

\begin{Thm}\label{thm:sp:gap:disc:SNSE}
	Fix $\dt_0 > 0$. Suppose there exists $K \in \NN$ and $\sigma \in \bdL^2$ such that
	\begin{align}\label{cond:K:sigma}
		\HK \subset \range{\sigma}, \quad \mbox{ and } \quad
		\lambda_{K+1} \geq 
		\frac{c}{\nu} \max \left\{ \frac{1}{\dt_0}, \frac{\dt_0^2 |\sigma|^4}{\nu^3 }, \frac{| \sigma|^4}{\nu^5 }  \right\}
	\end{align}
	for some absolute constant $c > 0$.	For each $N \in \NN$ and $0 < \dt \leq \dt_0$, let $\Pdisct$, $t \in \RR^+$, be the corresponding family of Markov kernels defined in \eqref{def:Pdisct}. Then, for every $m > 1$ there exists $\sm_m > 0$ such that for each $\sm \in (0, \sm_m]$ there exist $\varepsilon > 0$, $s \in (0,1]$, $T > 0$, and constants $C_1, C_2 > 0$ for which the following holds 
	\begin{align}\label{Wass:contr:Pdisc}
		\sup_{N \in \NN, \, 0< \dt \leq \dt_0} \Wesa(\mu \Pdisct, \tmu \Pdisct) \leq C_1 e^{-t C_2} \Wass_{\varepsilon, s, \sm/m} (\mu, \tmu)
	\end{align}
	for every $\mu, \tmu \in \Pr(\dL^2)$, and $t \geq T$. 
\end{Thm}

In view of \cref{rmk:exist:inv:meas}, \cref{thm:sp:gap:disc:SNSE} together with the Feller property of $\Pdisc$, $n \in \ZZ^+$, implies the existence of a unique associated invariant measure. We state this result below.
\begin{Cor}\label{cor:exist:inv:meas:disc}
	Consider the assumptions of \cref{thm:sp:gap:disc:SNSE}. Then, for each fixed discretization parameters $N \in \NN$ and $\dt > 0$, there exists a unique invariant measure $\mu_*^{N,\dt}$ of the discrete Markov semigroup $\Pdisc$, $n \in \ZZ^+$, and consequently of $\Pdisct$, $t \in \RR^+$.
\end{Cor}
\begin{proof}
The uniqueness of the invariant measure follows immediately from inequality \eqref{Wass:contr:Pdisc}. For the existence, as recalled in \cref{rmk:exist:inv:meas}, it follows similarly as in \cite[Corollary 4.11]{HairerMattinglyScheutzow2011} that it suffices to show there exists a complete metric $\tilde{\rho}$ on $\dL^2$ such that $\tilde{\rho} \leq \sqrt{\rhoes}$ and for which $\{\Pdisc\}_{n \in \ZZ^+}$ is a Feller semigroup on $(\dL^2, \tilde{\rho})$. Here, $\varepsilon > 0$ and $s \in (0,1]$ are any parameters such that \eqref{Wass:contr:Pdisc} holds. This is achieved, for example, by $\tilde{\rho} = \sqrt{\rhoes} = \rho_{\sqrt{\varepsilon},s/2}$, which is a metric on $\dL^2$ that is equivalent to the distance induced by the norm $|\cdot|$, so that the known Feller property of $\{\Pdisc\}_{n \in \ZZ^+}$ on $(\dL^2,|\cdot|)$ also holds in $(\dL^2, \sqrt{\rhoes})$. Clearly, if $\mu_*^{N,\dt}$ is an invariant measure for $\{\Pdisc\}_{n \in \ZZ^+}$, then from the definition \eqref{def:Pdisct} it follows immediately that $\mu_*^{N,\dt}$ is also an invariant measure for $\{\Pdisct\}_{t \in \RR^+}$.
\end{proof}

To prove \cref{thm:sp:gap:disc:SNSE}, we start by verifying the existence of an exponential Lyapunov structure as in assumption \ref{A:1} of \cref{thm:gen:sp:gap}.

\begin{Prop}\label{prop:Lyap:exp:disc}
	Fix any $N \in \NN$, $\dt, \dt_0 > 0$ with $\dt \leq \dt_0$, $\sigma \in \bdL^2$, and $\xi_0 \in \dL^2$.
	Let $\{\xfd^n\}_{n \in \ZZ^+}$ be the solution of \eqref{disc2DSNSEv2} corresponding to the parameters $N, \dt$ and satisfying $\xfd^0 = \Pi_N \xi_0$ almost surely. Then, for all
	$\sm \in \RR$ satisfying
	\begin{align}\label{cond:gamma:0}
		0 < \sm \leq \frac{\nu }{4 | \sigma|^2},
	\end{align}
	it holds that 
	\begin{align}\label{Lyap:disc:0}
	 	\bE \exp \left( \sm |\xfd^n|^2 \right) 
		\leq 
		\exp\left( \sm \left( \frac{2 |\xi_0|^2}{(1 + \nu \lambda_1 \dt)^n} + C \right)\right)
		\quad \mbox{  for all } n \in \ZZ^+,
	\end{align}
	for some positive constant $C$ depending only on $\nu, |\sigma|, \dt_0$.
	
	Consequently, recalling the definition of the Markov semigroup $\Pdisc$, $n \in \ZZ^+$, in \eqref{def:Pdisc:0}, it follows that for all $n \in \ZZ^+$ 
	\begin{align}\label{Lyap:disc:1}
		\Pdisc \exp \left( \sm |\xi_0|^2 \right) \leq \exp \left(  \sm \left(c e^{-n\dt \tC} |\xi_0|^2 + C \right)\right),
	\end{align}
	where $\tC$ is a positive constant depending only on $\nu, \dt_0$.
\end{Prop}
\begin{proof}
	Throughout the proof we adopt the simplified notation $\xfd^j = \xi^j$, $j \in \ZZ^+$, mentioned in \cref{subsubsec:sp:time:disc} above.
	
	Fix $n \in \NN$.
	For each $j \in \{1, \ldots, n\}$, we take the inner product of the first equation in \eqref{disc2DSNSEv2} with $\xi^j$ in $\dL^2$ and
	invoke the Hilbert space identity
	\begin{align}\label{eq:Hilb:sp}
	2 ( \xi - \txi, \xi) = |\xi |^2 +  |\xi - \txi|^2 - |\txi|^2 \quad \mbox{ for all } \xi, \txi \in \dL^2,
	\end{align}
	together with the orthogonality property \eqref{nonlin:ort}, to obtain that
	\begin{align}\label{energy:ineq:xn:1}
		|\xi^j|^2 + |\xi^j - \xi^{j-1}|^2 - |\xi^{j-1}|^2 + 2 \nu \dt |\nabla \xi^j|^2
		= 2 \dt^{1/2} (\Pi_N\sigma \eta_j, \xi^j)
		= 2 \dt^{1/2} (\sigma \eta_j, \xi^j).
	\end{align}
	In view of obtaining a well-defined martingale in \eqref{def:Mn} below, we add and subtract $2 \dt^{1/2} ( \sigma \eta_j, \xi^{j-1})$ in the right-hand side and estimate
	\begin{align*}
	2 \dt^{1/2} (\sigma \eta_j, \xi^j)
	= 2 \dt^{1/2} (\sigma \eta_j, \xi^j - \xi^{j-1}) + 2 \dt^{1/2} (\sigma \eta_j, \xi^{j-1})
	\notag \\
	\leq |\xi^j - \xi^{j-1}|^2 + \dt | \sigma \eta_j|^2 + 2 \dt^{1/2} ( \sigma \eta_j, \xi^{j-1}),
	\end{align*}
	so that, from \eqref{energy:ineq:xn:1},
	\begin{align}\label{energy:ineq:xn:2}
	(1 + \nu \dt) |\xi^j|^2 + \nu \dt |\nabla \xi^j|^2 \leq 
	|\xi^j|^2 + 2 \nu \dt |\nabla \xi^j|^2 
	\leq  |\xi^{j-1}|^2 +  \dt |\sigma \eta_j|^2 + 2 \dt^{1/2} (\sigma \eta_j, \xi^{j-1}),
	\end{align}
	for all $j \in \{1, \ldots,n\}$.
	
	Fix $m \in \NN$ with $m \geq n$. Denoting $\ra \coloneqq (1 + \nu \dt)^{-1}$, we obtain after multiplying both sides of \eqref{energy:ineq:xn:2} by $\ra^{m - j + 1}$ and summing over $j = 1,\ldots,n$ that
	\begin{align}\label{energyineq6}
	\ra^{m-n} |\xi^n|^2 + \nu \dt \sum_{j=1}^n \ra^{m-j+1} |\nabla \xi^j|^2 
	\leq \ra^m|\xi_0|^2 + \sum_{j=1}^n \dt \ra^{m-j+1} |\sigma \eta_j|^2 
	+ M_n 
	\end{align}
	where $\{M_n\}_{n \in \NN}$ is the martingale defined as 
	\begin{align}\label{def:Mn}
	M_n \coloneqq 2 \dt^{1/2} \sum_{j=1}^n \ra^{m-j+1} (\sigma \eta_j, \xi^{j-1}) \quad \mbox{ for all } n \in \NN,
	\end{align}
	with corresponding quadratic variation given by
	\begin{align}\label{Mnquadvar0}
	\langle M \rangle_n = 4 \dt \sum_{j=1}^n \sum_{k=1}^d \ra^{2(m-j+1)} ( \sigma_k, \xi^{j-1})^2.
	\end{align}
	We then estimate $\langle M \rangle_n$ as 
	\begin{align*}
	\langle M \rangle_n
	&\leq 4 \dt |\sigma|^2 \sum_{j=1}^n \ra^{2(m-j +1)} |\xi^{j-1}|^2 
	=  4 \dt |\sigma|^2 \sum_{j=0}^{n-1} \ra^{2(m-j)} |\xi^j|^2 
	\\
	&\leq  4 \dt |\sigma|^2 \left( \ra^{2m} |\xi_0|^2 + \sum_{j=1}^{n-1} \ra^{2 (m-j)} |\nabla \xi^j|^2 \right)
	\\
	&\leq 4 \dt |\sigma|^2 \left( \ra^{m+1} |\xi_0|^2 + \sum_{j=1}^n \ra^{m-j + 1} |\nabla \xi^j|^2 \right),
	\end{align*}
	where in the last line we used that $\ra \leq 1$ and $m \geq n \geq 1$. Thus, for all $\sm \in \RR$ satisfying \eqref{cond:gamma:0} we obtain that 
	\begin{align}\label{estMn0}
		\sm \langle M \rangle_n 
		\leq
		\ra^m |\xi_0|^2 + \nu \dt \sum_{j=1}^n \ra^{m-j + 1} |\nabla \xi^j|^2.
	\end{align}
	
	Now, adding and subtracting $\sm \langle M \rangle_n$ to the right-hand side of \eqref{energyineq6} and invoking \eqref{estMn0} it follows that
	\begin{align}\label{energyineq7}
		\ra^{m-n} |\xi^n|^2  \leq 2 \ra^m |\xi_0|^2 + \sum_{j=1}^n \dt \ra^{m-j+1} |\sigma \eta_j|^2 + M_n - \sm \langle M \rangle_n.
	\end{align}
	Multiplying by $\sm$, taking exponentials and expected values on both sides of \eqref{energyineq7}, we deduce that
	\begin{align}\label{expbound00}
		\bE \exp \left( \sm \ra^{m-n} |\xi^n|^2 \right) 
		\leq \exp\left( 2 \sm \ra^m |\xi_0|^2 \right)
		\left[ \prod_{j=1}^n \bE \exp \left(  2 \sm  \dt \ra^{m-j+1} | \sigma \eta_j|^2 \right)\right]^{1/2}
		\left( \bE \tM_n \right)^{1/2},
	\end{align}
	where 
	\begin{align}\label{def:tMn}
		\tM_n = \exp \left(2 \sm  M_n - 2\sm^2 \langle M \rangle_n \right),
	\end{align}
	and we used the independence of the random variables $\sigma \eta_j$, $j = 1, \ldots, n$, to write the second factor in the right-hand side of \eqref{expbound00}.
	
	From \eqref{def:Mn}, let us denote $z_n \coloneqq  \ra^{m-n +1} (\sigma \eta_n, \xi^{n-1})$, and consider the regular conditional probability of $z_n$ given $\mF_{t_{n-1}}$, i.e. $\mu_n(\omega, A) \coloneqq \bP(z_n(\omega) \in A \,|\, \mF_{t_{n-1}}) = \bE[\ind_{z_n^{-1}(A)} \,|\, \mF_{t_{n-1}}]$, for $\omega \in \Omega$, $A \in \mB(\RR)$, see e.g. \cite[Section 10.2]{Dudley2002}. It is not difficult to show that, for each fixed $\omega \in \Omega$, $\mu_n(\omega, \cdot)$ is a Gaussian probability measure on $\mB(\RR)$ with zero mean and variance $\ra^{2(m-n+1)} (Q_0 \xi^{n-1}, \xi^{n-1}) = \sum_{k=1}^d \ra^{2(m-n+1)}(\sigma_k, \xi^{n-1})^2$, where $Q_0$ is as defined in \eqref{def:Q:TrQ}. Using this fact, one can easily show that $\{\tM_n\}_{n \in \mathbb{N}}$ is a martingale with respect to the filtration $\{\mF_{t_n}\}_{n \in \mathbb{N}}$, and $\bE \tM_n = 1$ for all $n$ (see e.g. \cite[Appendix]{LMS2007}). Moreover, since $ \sigma \eta_j \sim \mN(0, Q_0)$, $j = 1, \ldots, n$, from a general result on Gaussian probability measures on Hilbert spaces \cite[Proposition 2.17]{DaPratoZabczyk2014} it follows that
	\begin{align}\label{ineq:gauss:prob}
		\bE \exp(\gamma |\sigma \eta_j|^2 ) \leq \frac{1}{( 1 - 2 \gamma | \sigma|^2  )^{1/2}} \quad \mbox{  for  all  } \gamma < \frac{1}{2 |\sigma|^2},
	\end{align}
	where we recall that $| \sigma|^2 = \tr (Q_0)$.  In particular, since $\sm \leq \nu /(4 | \sigma|^2)$ by assumption \eqref{cond:gamma:0}, and since $\ra^{m - j + 1} \leq \ra = 1/ (1 + \nu \dt)$, we have that \eqref{ineq:gauss:prob} holds with $\gamma = 2 \sm \dt \ra^{m - j + 1} $. Thus, from \eqref{expbound00},
	\begin{align}
	\bE \exp \left( \sm \ra^{m-n} |\xi^n|^2 \right) 
	\leq \exp\left(2 \sm \ra^m |\xi_0|^2 \right) \prod_{j=1}^n \frac{1}{(1 - 4 \sm \dt \ra^{m-j+1} |\sigma|^2)^{1/4}} \quad \mbox{ for all } m \geq n \geq 1.
	\end{align}
	In particular, if $m=n$ then
	\begin{align}\label{ineq:exp:r:xn}
		\bE \exp \left( \sm |\xi^n|^2 \right) 
		&\leq \exp\left( 2 \sm \ra^n |\xi_0|^2 \right) \prod_{j=1}^n \frac{1}{(1 - 4 \sm \dt \ra^j | \sigma|^2)^{1/4}} 
		\notag\\
		&= \exp\left( 2 \sm \ra^n |\xi_0|^2 \right) \exp \left( -\frac{1}{4} \sum_{j=1}^n \ln \left( 1 - 4 \sm \dt \ra^j | \sigma|^2 \right) \right).
	\end{align}
	
	Since $- \ln(1 - x) \leq x (1 - x)^{-1}$ for all $x \in (0,1)$, we obtain 
	\begin{align}
		-\frac{1}{4} \sum_{j=1}^n \ln\left( 1 - 4 \sm \dt \ra^j | \sigma|^2 \right)
		\leq \frac{1}{4} \sum_{j=1}^n \frac{4 \sm \dt \ra^j | \sigma|^2 }{ 1 -  4 \sm \dt \ra^j | \sigma|^2}.
	\end{align}
	Moreover, since $\sm \leq \nu /(4 |\sigma|^2)$ and $\ra^j \leq \ra = 1/ (1 + \nu  \dt)$, it follows that $[1 -  4 \sm \dt \ra^j | \sigma|^2]^{-1} \leq 1 + \nu \dt$, so that
	\begin{align}\label{ineq:term:exp:r:xn}
	\frac{1}{4} \sum_{j=1}^n \frac{4 \sm \dt \ra^j |\sigma|^2 }{ 1 -  4 \sm \dt \ra^j | \sigma|^2}
	&\leq (1 + \nu  \dt) \sm \dt | \sigma|^2 \sum_{j=1}^n  \ra^j \notag \\
	&\leq (1 + \nu  \dt) \sm \dt | \sigma|^2  \frac{\ra}{1 - \ra}
	= (1 + \nu  \dt) \frac{ \sm | \sigma|^2}{\nu }
	\leq (1 + \nu  \dt_0) \frac{ \sm | \sigma|^2}{\nu }.
	\end{align}
	Therefore, from \eqref{ineq:exp:r:xn}-\eqref{ineq:term:exp:r:xn}, it follows that
	\begin{align*}
	\bE \exp \left( \sm |\xi^n|^2 \right) 
	\leq \exp\left( \sm (2 \ra^n |\xi_0|^2  + C)\right),
	\end{align*}
	where $C = (1 + \nu  \dt_0)|\sigma|^2/(\nu )$. This shows \eqref{Lyap:disc:0}.
	
	For the remaining inequality, \eqref{Lyap:disc:1}, we use the fact that for any constant $0 < a < 1$ we have $\ln(1 + x) \geq ax $ for all $x \in [0, (1/a) - 1]$. Since $\dt \leq \dt_0$, we take $a \coloneqq 1/(1 + \nu  \dt_0)$ and obtain that $\ln (1 + \nu  \dt) \geq \nu  \dt /(1 + \nu  \dt_0)$, so that
	\begin{align}\label{ineq:exp:0}
		\frac{2}{(1 + \nu  \dt)^n} = 2 \exp \left( - n \ln (1 + \nu  \dt )\right)
		\leq 2 \exp \left( - \frac{\nu }{1+ \nu  \dt_0} n \dt \right).
	\end{align} 
	From \eqref{Lyap:disc:0} and \eqref{ineq:exp:0}, it thus follows that, for every $0 < \sm \leq \nu /(4 | \sigma|^2)$,
	\begin{align*}
		\Pdisc \exp \left( \sm |\xi_0|^2 \right)
		=
		\bE \exp \left( \sm |\xi^n|^2 \right) 
		&\leq 
		\exp \left( \sm \left( \frac{2 |\xi_0|^2 }{(1 + \nu  \dt)^n} + C \right)\right)
		\\
		&\leq
		\exp \left( \sm \left( 2 \exp \left( - \frac{\nu }{1 + \nu  \dt_0} n \dt \right) |\xi_0|^2 + C \right)\right),
	\end{align*}
	which shows \eqref{Lyap:disc:1} and concludes the proof.
\end{proof}

For showing the remaining assumptions \ref{A:2} and \ref{A:3} from \cref{thm:gen:sp:gap}, we follow a similar asymptotic coupling strategy from previous works, see e.g. \cite{BKL2001,EMatSi2001,KukShi2001,Mattingly2002,KukShi2002,Mattingly2003,Hai02,WL02,HaiMa2006,HairerMattingly2008,DO05,HM11,HairerMattinglyScheutzow2011,KuksinShirikyan12,FGHRT15,ButkovskyKulikScheutzow2019}. The idea consists in introducing the following modified equation for a given $\xi_0 \in \dL^2$ and corresponding solution $\xfd^n = \xfd^n( \Pi_N \xi_0) = \xi^n( \Pi_N \xi_0)$, $n \in \NN$, of \eqref{disc2DSNSEv}. Namely, we consider $\txfd^n = \txi^n$, $n \in \NN$, satisfying
\begin{align}\label{eq:nudging:1}
\txi^n = \txi^{n-1} + \dt [\nu \Delta \txi^n - \Pi_N(\tu^{n-1} \cdot \nabla \txi^n) - \beta \Pi_K (\txi^n - \xi^n( \Pi_N \xi_0))]
+ \sum_{k=1}^d \Pi_N \sigma_k (W^k(t_n) - W^k(t_{n-1})) ,
\end{align}
\begin{align}\label{eq:nudging:2}
\tu^{n-1} = \mK \ast \txi^{n-1}.
\end{align}
Here, the extra term $- \beta \dt \Pi_K (\txi^n - \xi^n( \Pi_N \xi_0))]$ has the purpose of enforcing a suitable control over ``large'' scales, with $K \in \NN$ representing the number of controlled modes, to be appropriately chosen in \eqref{cond:K:beta} below. 

Analogously as in \cref{prop:wellposed:full:disc}, we can show that system \eqref{eq:nudging:1}-\eqref{eq:nudging:2} is well-posed in the pathwise sense. We omit the technical details. Therefore, for each $N \in \NN$, $\dt > 0$ and $\xi_0 \in \dL^2$, we may define
\begin{align}\label{def:control:P}
	\tPdisc (\txi_0, \cO) = \bP (\txfd^n(\Pi_N \txi_0; \Pi_N \xi_0) \in \cO) \quad \mbox{ for all } n \in \ZZ^+, \,\, \txi_0 \in \dL^2 \, \mbox{ and } \cO \in \mB(\dL^2),
\end{align}
where $\txfd^n(\Pi_N \txi_0; \Pi_N \xi_0)$ is the unique (strong) solution of \eqref{eq:nudging:1}-\eqref{eq:nudging:2} with respect to a fixed stochastic basis $(\Omega, \mF, \{\mF_t\}_{t \geq 0}, \bP, \{W^k\}_{k=1}^d)$, and which satisfies the initial condition $\txfd^ 0 = \Pi_N\txi_0$ almost surely. Moreover, for every bounded and measurable function $\varphi: \dL^2 \to \RR$, we denote
\begin{align}
	\tPdisc \varphi(\txi_0) = \bE \varphi(\txfd^n(\Pi_N \txi_0; \Pi_N \xi_0)) \quad \mbox{ for all } n \in \ZZ^+ \mbox{ and } \txi_0 \in \dL^2.
\end{align}

Given any $\xi_0, \txi_0 \in \dL^2$, the idea consists in utilizing the family $\tPdisc$, $n \in \ZZ^+$, to estimate the Wasserstein distance $\Wes$ between $\Pdisc(\xi_0, \cdot)$ and $\Pdisc(\txi_0, \cdot)$ as
\begin{align}\label{trg:ineq:wass}
	\Wes (\Pdisc(\xi_0, \cdot), \Pdisc(\txi_0, \cdot)) 
	\leq \Wes (\Pdisc(\xi_0, \cdot), \tPdisc(\txi_0, \cdot))  +  \Wes ( \tPdisc(\txi_0, \cdot), \Pdisc(\txi_0, \cdot)),
\end{align}
which holds since $\Wes$ is a metric in $\Pr(\dL^2)$. We then estimate each term on the right-hand side of \eqref{trg:ineq:wass} by analyzing system \eqref{eq:nudging:1}-\eqref{eq:nudging:2} under two different perspectives. The first term is estimated by establishing a suitable contraction between the solution $\txfd^n(\Pi_N \txi_0;\Pi_N \xi_0)$ of  \eqref{eq:nudging:1}-\eqref{eq:nudging:2} and the solution $\xfd^n(\Pi_N \xi_0)$ of \eqref{disc2DSNSEv}. This is possible due to the presence of the control term $- \beta \Pi_K (\txfd^n - \xfd^n(\Pi_N \xi_0))$ in \eqref{eq:nudging:1}, and provided the number $K \in \NN$  of controlled modes and the tuning parameter $\beta > 0$ are chosen sufficiently large (see \eqref{cond:K:beta} below).

For the second term in the right-hand side of \eqref{trg:ineq:wass}, due to uniqueness of pathwise strong solutions of \eqref{disc2DSNSEv} we deduce that the solution $\txfd^n(\Pi_N\txi_0; \Pi_N \xi_0, W)$ of \eqref{eq:nudging:1}-\eqref{eq:nudging:2} corresponding to the Wiener process $W$ coincides with the solution $\xfd^n(\Pi_N \txi_0; \hW)$ of \eqref{disc2DSNSEv} corresponding to the following shifted process
\begin{align}\label{def:shifted:W:1}
\hW(t) = W(t) + \int_0^t \sum_{j=1}^\infty \psi_j \ind_{[t_{j-1},t_j)}(\tau) \rd \tau,
\end{align}
where
\begin{align}\label{def:shifted:W:2}
\psi_j = - \beta \sigma^{-1} \Pi_K (\txfd^j(\Pi_N\txi_0; \Pi_N \xi_0,W) - \xfd^j( \Pi_N \xi_0;W)) \quad \forall j.
\end{align}
Here we recall that $\sigma^{-1}$ denotes the pseudo-inverse of $ \sigma$ (see \cref{subsubsec:2DSNSE}). The second term in the right-hand side of \eqref{trg:ineq:wass} can then be estimated by the \textit{total variation distance} (see \eqref{def:TV} below) between the laws of the processes $W$ and $\hW$. This is in turn estimated via a Girsanov-type result. We note carefully that in order to have the expression in \eqref{def:shifted:W:2} well-defined, particularly in what concerns the domain of definition of $\sigma^{-1}$, we assume that $\HK \subset \range{\sigma}$.

Under this approach, we prove here the following results validating assumptions \ref{A:2} and \ref{A:3} of \cref{thm:gen:sp:gap} for the Markov semigroup $\Pdisc$, $n \in \ZZ^+$, and the class of distances $\Lambda = \{\rhoes \,:\, \varepsilon > 0, s \in (0,1]\}$ defined in \eqref{def:rhoes} above.

\begin{Prop}\label{prop:small:Wass}
Fix $\dt_0 > 0$ and suppose there exists $K \in \NN$ and $\sigma \in \bdL^2$ such that \eqref{cond:K:sigma} holds.
Then, for every $M > 0$, $\varepsilon > 0$ and $s \in (0,1]$, there exist a time $T_1 = T_1(M, \varepsilon,s) > 0$ and a coefficient $\kappa_1 = \kappa_1(M) \in (0,1)$, which is independent of $\varepsilon,s$, such that 
\begin{align}\label{ineq:small:Wass}
	\sup_{N \in \NN, \, 0 < \dt \leq \dt_0} \sup_{ n \geq T_1/\dt} \Wes(\Pdisc(\xi_0, \cdot), \Pdisc(\txi_0, \cdot)) \leq 1 - \kappa_1
\end{align}
for all $\dt_0 > 0$, 
and for every $\xi_0, \txi_0 \in \dL^2$ with $|\xi_0| \leq M$ and $|\txi_0| \leq M$.
\end{Prop}

\begin{Prop}\label{prop:contr:Wass} 
Fix $\dt_0 > 0$ and suppose there exists $K \in \NN$  and $\sigma \in \bdL^2$ such that  \eqref{cond:K:sigma} holds.
Then, for every $\kappa_2 \in (0,1)$ and for every $r > 0$ there exists $s \in (0,1]$ for which the following holds:
\begin{itemize}
    \item[(i)] For every $\varepsilon > 0$ and $\dt_0 > 0$, there exists a constant $C = C(\varepsilon,s, \dt_0) > 0$  such that
		\begin{align}\label{ineq:contr:Wass:0}
			\sup_{N \in \NN, \, 0 < \dt \leq \dt_0, \, n \in \ZZ^+} \Wes(\Pdisc(\xi_0,\cdot), \Pdisc(\txi_0,\cdot)) \leq C \exp(r |\xi_0|^2) \rhoes(\xi_0,\txi_0)
		\end{align}
		for every $\xi_0, \txi_0 \in \dL^2$ with $\rhoes(\xi_0, \txi_0) < 1$.
	\item[(ii)] For every $\dt_0 > 0$, there exist a parameter $\varepsilon=  \varepsilon(\kappa_2,r)  >0$ and a time $T_2 = T_2(\kappa_2, r) > 0$
	such that
		\begin{align}\label{ineq:contr:Wass}
			\sup_{N \in \NN, \, 0 < \dt \leq \dt_0} \sup_{ n \geq T_2/\dt} \Wes(\Pdisc(\xi_0, \cdot), \Pdisc(\txi_0, \cdot)) \leq \kappa_2 \exp(r |\xi_0|^2) \rhoes(\xi_0, \txi_0)
		\end{align}
		for every $\xi_0, \txi_0 \in \dL^2$ with $\rhoes(\xi_0, \txi_0) < 1$.
\end{itemize}
\end{Prop}

\begin{Rmk}
    We notice that item $(i)$ of \cref{prop:contr:Wass} gives a slightly stronger result than required in the general assumption \ref{A3:ii} of \cref{thm:gen:sp:gap}. Indeed, inequality \eqref{ineq:contr:Wass:0} is valid over all $n \in \ZZ^+$ and $\varepsilon > 0$. In contrast, \eqref{fin:time:Wass} concerns only a finite time interval $[0,\tau]$ and a particular choice of distance-like function $\rhoes$, which thus entails both a particular choice of $s \in (0,1]$ and $\varepsilon > 0$.
\end{Rmk}

Before proceeding with the proofs of \cref{prop:small:Wass} and \cref{prop:contr:Wass}, we establish some preliminary facts and terminology that are necessary for following the outline described under \eqref{trg:ineq:wass} above. We start with the following result establishing suitable exponential moment bounds for solutions of \eqref{disc2DSNSEv2}.

\begin{Lem}\label{prop:exp:bounds:sup}
	Fix any $N \in \NN$, $\dt, \dt_0 > 0$ with $\dt \leq \dt_0$, $\sigma \in \bdL^2$ and $\xi_0 \in \dL^2$. Let $\{\xfd^n\}_{n \in \ZZ^+}$ be the solution of \eqref{disc2DSNSEv2} corresponding to the parameters $N,\dt$, and satisfying $\xfd^0 = \Pi_N \xi_0$ almost surely. Then, there exists an absolute constant $c > 0$ such that for all  $\sm \in \RR$ satisfying
	\begin{align}\label{condgammasup}
		0 < \sm \leq \frac{c}{|\sigma|^2} \min\left\{ \nu , \frac{1}{\dt_0}\right\},
	\end{align}
	the following inequality holds
	\begin{align}\label{expboundssup}
		\bE \sup_{n \geq 1} \exp\left( \sm |\xfd^n|^2 + \sm \nu \dt \sum_{j=1}^n |\nabla \xfd^j|^2 + \frac{n}{4} \ln (1 - 4 \sm \dt |\sigma|^2) \right) 
		\leq \tilde{c} \exp \left( C \sm |\xi_0|^2 \right),
	\end{align}
	and, consequently, 
	\begin{align}\label{expbounds:1}
	\bE \exp\left( \sm |\xfd^n|^2 + \sm \nu \dt \sum_{j=1}^n |\nabla \xfd^j|^2\right) 
	\leq \tilde{c} \exp \left( C \sm |\xi_0|^2 \right) \exp \left( \tilde{c} \sm |\sigma|^2 n \dt \right) \quad \mbox{ for all } n \in \mathbb{N}.
	\end{align}
	Here, $C = \tilde{c} (1 + \nu \dt_0)$ and $\tilde{c} > 0$ is an absolute constant.
\end{Lem}
\begin{proof}
Proceeding as in \eqref{energy:ineq:xn:1}-\eqref{energy:ineq:xn:2} above, and summing \eqref{energy:ineq:xn:2} over $j = 1, \ldots, n$, we obtain
\begin{align}\label{energyineq3}
	|\xi^n|^2 - |\xi_0|^2 + 2 \nu \dt \sum_{j=1}^n |\nabla \xi^j|^2 \leq \dt \sum_{j=1}^n | \sigma \eta_j|^2 + M_n,
\end{align}
where $\{M_n\}_{n \in \mathbb{N}}$ is the martingale defined as
\begin{align}\label{defMn}
M_n := 2 \dt^{1/2} \sum_{j=1}^n (\sigma \eta_j, \xi^{j-1})
\end{align}
with corresponding quadratic variation given by
\begin{align}\label{quadvar:Mn}
\langle M \rangle_n = 4 \dt \sum_{j=1}^n \sum_{k=1}^d (\sigma_k, \xi^{j-1})^2.
\end{align}
We estimate $\langle M \rangle_n$ as
\begin{align}\label{estMn}
\langle M \rangle_n 
\leq 4 \dt |\sigma|^2 \sum_{j=1}^n |\xi^{j-1}|^2 
&= 4 \dt |\sigma|^2 \left( |\xi_0|^2 + \sum_{j=1}^{n-1} |\xi^{j}|^2 \right) 
\notag\\
&\leq 4 \dt |\sigma|^2 |\xi_0|^2 + 4 \dt |\sigma|^2 \sum_{j=1}^{n-1} |\nabla \xi^j|^2.
\end{align}

Thus, under assumption \eqref{condgammasup} on $\sm$ with a suitable absolute constant $c$ it follows that
\begin{align*}
	\sm \langle M \rangle_n \leq \nu  \dt |\xi_0|^2 + \nu \dt \sum_{j=1}^{n-1} |\nabla \xi^j|^2.
\end{align*}
Adding and subtracting $\sm \langle M\rangle_n$ in \eqref{energyineq3}, yields
\begin{align}\label{energyineq4}
|\xi^n|^2 + \nu \dt \sum_{j=1}^n |\nabla \xi^j|^2 \leq (1 + \nu \dt)  |\xi_0|^2 + \dt \sum_{j=1}^n | \sigma \eta_j|^2 +  M_n - \sm \langle M \rangle_n.
\end{align}
We now subtract $R_n \coloneqq - \frac{n}{2\sm} \tr (\ln(1 - 2 \sm \dt Q_0))$ from both sides of \eqref{energyineq4}, where $Q_0$ is defined in \eqref{def:Q:TrQ}. Then, multiplying by $\sm/2$, taking exponentials, the supremum over $n \in \{1, \ldots,m\}$ for some $m \in \mathbb{N}$, and expected values, it follows that
\begin{multline}\label{supexpmoment}
\bE \sup_{1 \leq n \leq m} \exp\left( \frac{\sm}{2} |\xi^n|^2 + \frac{\sm \nu \dt}{2} \sum_{j=1}^n |\nabla \xi^j|^2  - \frac{\sm}{2} R_n \right)\\
\leq \exp \left( \frac{\sm}{2}(1 + \nu \dt )|\xi_0|^2 \right) 
\left[ \bE \sup_{1 \leq n \leq m} \exp\left(\frac{D_n}{2}\right) \right]^{1/2}
\left[ \bE \sup_{1 \leq n \leq m} \exp\left(\frac{E_n}{2}\right) \right]^{1/2}, 
\end{multline}
where 
\begin{align}\label{def:Dn}
D_n = 2 \sm M_n - 2 \sm^2 \langle M \rangle_n
\end{align}
and
\begin{align}\label{def:El}
E_n =   \sm \dt \sum_{j=1}^n | \sigma \eta_j|^2 -  \sm R_n.
\end{align}

Similarly as in \eqref{def:tMn}, we have that $\{\exp(D_n)\}_{n \in \mathbb{N}}$ is a martingale with respect to the filtration $\{\mF_{t_n}\}_{n \in \mathbb{N}}$, and $\bE \exp(D_n) = 1$ for all $n$.

Clearly, each $\exp(E_n)$ is measurable with respect to $\mF_{t_n}$. To conclude that $\{\exp(E_n)\}_{n \in \NN}$ is a martingale, it remains to show that $\bE |\exp(E_n)| = \bE \exp(E_n) < \infty$ and $\bE (\exp(E_{n+1}) | \mF_{t_n}) = \exp(E_n)$ for all $n \in \mathbb{N}$. Since $\sm < c (\dt_0 |\sigma|^2)^{-1}$ and $\sigma \eta_n \sim \mN(0,Q_0)$, it follows by invoking once again \cite[Proposition 2.17]{DaPratoZabczyk2014} that for all $\dt \leq \dt_0$
\begin{align*}
	\bE \exp (\sm \dt |\sigma \eta_n|^2) = \exp\left(-\frac{1}{2} \tr \ln(1 - 2 \sm \dt Q_0) \right) \quad \mbox{ for all } n \in \mathbb{N}.
\end{align*}
Hence,
\begin{align*}
\bE(\exp(E_{n+1}) \,|\, \mF_{t_n}) &= \exp\left( \sm \dt \sum_{j=1}^{n} |\sigma \eta_j|^2\right) \bE\exp \left(\sm \dt |\sigma \eta_{n+1}|^2 + \frac{(n+1)}{2} \tr \ln(1 - 2\sm \dt Q_0)\right)\\
&=  \exp\left( \sm \dt \sum_{j=1}^{n} |\sigma \eta_j|^2  +\frac{n}{2} \tr \ln (1 - 2 \sm \dt Q_0)\right) = \exp(E_{n}).
\end{align*}
This implies that, for all $n \in \NN$, 
\begin{align}\label{eq:E:expEn:1}
\bE[\exp(E_n)] 
= \bE[\bE[\exp(E_n) \,|\, \mF_{t_1}]] 
= \bE[\exp(E_1)]
= \exp \left( \sm \dt |\sigma \eta_1|^2 + \frac{1}{2} \tr \ln(1 - 2\sm \dt Q_0)  \right) = 1.
\end{align} 
Therefore, $\{\exp(E_n)\}_{n \in \NN}$ is a martingale and, moreover, $\bE \exp(E_n) = 1$ for all $n$.

With these facts, we proceed to further estimate the right-hand side of \eqref{supexpmoment} by noticing that 	
\begin{align}\label{estEk}
\bE \sup_{1\leq n \leq m} \exp\left(\frac{E_n}{2}\right) 
&= \int_0^\infty \bP \left(  \sup_{1\leq n \leq m} \exp(E_n) \geq z^2 \right) d z
\notag \\
&\leq 1 + \int_1^\infty \bP \left( \sup_{1\leq n \leq m} \exp(E_n) \geq z^2\right) d z
\leq 1 + \int_1^\infty \frac{\bE \exp(E_m)}{z^2} d z = 2,
\end{align}
where the last inequality follows from Doob's martingale inequality, while in the final equality we used \eqref{eq:E:expEn:1}. Analogously, we can show that
\begin{align}\label{estDk}
\bE \sup_{1\leq n \leq m} \exp\left(\frac{D_n}{2}\right)  \leq 2.
\end{align}
Plugging estimates \eqref{estEk} and \eqref{estDk} into \eqref{supexpmoment}, it follows that
\begin{align}\label{expboundssup0}
	\bE \sup_{1 \leq n \leq m} \exp\left( \frac{\sm}{2} |\xi^n|^2 + \frac{\sm \nu \dt}{2} \sum_{j=1}^n |\nabla \xi^j|^2 - \frac{\sm}{2} R_n \right) \leq 2 \exp \left( \frac{\sm}{2}(1 + \nu \dt )|\xi_0|^2 \right) ,
\end{align}
for all $m \in \mathbb{N}$. Replacing $\sm/2$ by $\sm$, and noticing that $ \tr (\ln(1 - 4 \sm \dt Q_0)) \geq \ln (1 - 4 \sm \dt \tr(Q_0))$ (see \cite[Proposition 2.17]{DaPratoZabczyk2014} ) and $\tr(Q_0) = |\sigma|^2$, we obtain that for all $N \in \NN$ and $\dt \leq \dt_0$
\begin{align}\label{expboundssup0:1}
	\bE \sup_{1 \leq n \leq m} \exp\left( \sm |\xi^n|^2 + \sm \nu \dt \sum_{j=1}^n |\nabla \xi^j|^2 + \frac{n}{4} \ln (1 - 4 \sm \dt |\sigma|^2) \right) 
	\leq 2 \exp \left( \frac{\sm}{2}(1 + \nu \dt_0 )|\xi_0|^2 \right) 
\end{align}
for all $m \in \NN$. Now we conclude \eqref{expboundssup} from \eqref{expboundssup0:1} by invoking the Monotone Convergence theorem.

For the final inequality \eqref{expbounds:1}, first notice that \eqref{expboundssup} clearly implies
\begin{align}\label{expbounds}
\bE \exp\left( \sm |\xi^n|^2 + \sm \nu \dt \sum_{j=1}^n |\nabla \xi^j|^2  \right) 
\leq \frac{ \tilde{c} \exp \left( C \sm |\xi_0|^2 \right)}{ ( 1 - 4 \sm \dt | \sigma|^2  )^{n/4}} \quad \mbox{ for all } n \in \NN.
\end{align}
Now we use the elementary fact that $\ln(1 - x) \geq - e x$ for every $0 \leq x \leq 1/e$. Thus, choosing the constant $c$ in \eqref{condgammasup} appropriately so that $4 \sm \dt_0 |\sigma|^2 \leq 1/e$, we obtain that for all $\dt \leq \dt_0$
\begin{align*}
(1 - 4 \sm \dt |\sigma|^2)^{-n/4} =
\exp \left( - \frac{n}{4}  \ln (1 - 4 \sm \dt |\sigma|^2)   \right)
\leq 
\exp( \tilde{c} \sm |\sigma|^2 n \dt).
\end{align*}
Plugging this inequality into \eqref{expbounds}, we deduce \eqref{expbounds:1}. This concludes the proof.
\end{proof}

Next, we have the following contraction result.

\begin{Lem}\label{lem:bound:diff:xi:txi}
	Fix any $N \in \NN$, $\dt, \dt_0 > 0$ with $\dt \leq \dt_0$, $\sigma \in \bdL^2$, and $\xi_0, \txi_0 \in \dL^2$. 
	Let $\txfd^n = \txfd^n(\Pi_N \txi_0; \Pi_N \xi_0)$, $n \in \ZZ^+$, be the solution of \eqref{eq:nudging:1}-\eqref{eq:nudging:2} corresponding to the parameters $N, \dt, \xi_0$, and satisfying $\txfd^0 = \Pi_N \txi_0$ almost surely. Suppose $K$ and $\beta$ from \eqref{eq:nudging:1} satisfy 
	\begin{align}\label{cond:K:beta}
	\nu \lambda_{K+1} \geq 2 \beta
	\end{align}
	and
	\begin{align}\label{cond:beta:0}
	\beta \geq c \max \left\{ \frac{1}{\dt_0}, \frac{\dt_0^2 |\sigma|^4}{\nu^3 }, \frac{| \sigma|^4}{\nu^5 }  \right\}
	\end{align}
	for some absolute constant $c > 0$.
	Then, for every $n \in \ZZ^+$
	\begin{align}\label{bound:diff:xi:txi}
	\bE |\txfd^n(\Pi_N \txi_0; \Pi_N \xi_0) - \xfd^n(\Pi_N \xi_0)|^2 
	\leq
	\tilde{c} \frac{ \exp\left(C (\nu^3 \beta)^{-1/2} |\xi_0|^2 \right)}{(1 + \beta \dt)^{3n/4}} |\txi_0 - \xi_0|^2 ,
	\end{align}
	where $C > 0$ is a constant depending only on $\nu, \dt_0$.
\end{Lem}
\begin{proof}
	Denote $\zeta^n = \zfd^n := \txfd^n(\Pi_N \txi_0; \Pi_N \xi_0) - \xfd^n(\Pi_N \xi_0)$ and $\bv^n = \bv_{N,\dt}^n = \tufd^n - \ufd^n$. Subtracting \eqref{disc2DSNSEv} from \eqref{eq:nudging:1}, we obtain
	\begin{align}\label{eq:zeta}
	\zeta^n = \zeta^{n-1}  + \dt [\nu \Delta \zeta^n - \Pi_N(\bv^{n-1} \cdot \nabla \zeta^n) - \Pi_N(\bv^{n-1} \cdot \nabla \xi^n) - \Pi_N (\bu^{n-1} \cdot \nabla \zeta^n) - \beta \Pi_K \zeta^n ].
	\end{align}
	Taking the inner product of \eqref{eq:zeta} with $\zeta^n$ and invoking \eqref{nonlin:ort} and \eqref{eq:Hilb:sp}, it follows that
	\begin{align}\label{ineq:zeta:0}
	|\zeta^n|^2 - |\zeta^{n-1}|^2 + |\zeta^n - \zeta^{n-1}|^2 + 2 \nu \dt |\nabla \zeta^n |^2 
	= -2 \dt (\bv^{n-1} \cdot \nabla \xi^n, \zeta^n) - 2 \beta \dt |\Pi_K \zeta^n|^2.
	\end{align}
	Invoking \eqref{ineq:nonlin:a:0} with $a = 1/2$ and Young's inequality, we estimate the nonlinear term above as
	\begin{align}\label{ineq:nonlin:term:per}
	2 \dt |(\bv^{n-1} \cdot \nabla \xi^n, \zeta^n)| 
	&\leq \tilde{c} \dt |\zeta^{n-1}| |\nabla \xi^n | |\zeta^n|^{1/2} |\nabla \zeta^n |^{1/2} 
	\notag \\
	&\leq \tilde c \frac{\dt}{(\nu \beta)^{1/2}} |\zeta^{n-1}|^2 |\nabla \xi^n|^2 + \beta \dt |\zeta^n|^2 + \nu \dt |\nabla \zeta^n |^2,
	\end{align}
	for some absolute constant $\tilde c > 0$. Thus, from \eqref{ineq:zeta:0},
	\begin{align*}
	|\zeta^n|^2 - |\zeta^{n-1}|^2 + |\zeta^n - \zeta^{n-1}|^2 +  \nu \dt |\nabla \zeta^n |^2
	\leq \tilde c \frac{\dt}{(\nu  \beta)^{1/2}} |\zeta^{n-1}|^2 |\nabla \xi^n |^2 + \beta \dt |\zeta^n|^2 - 2 \beta \dt |\Pi_K \zeta^n|^2.
	\end{align*}
	
	With inequality \eqref{ineq:Poincare:N}, we estimate the last term in the left-hand side as
	\begin{align}\label{ineq:PiK}
	\nu \dt |\nabla \zeta^n |^2 
	= \nu \dt (|\nabla \Pi_K \zeta^n |^2 + |\nabla (I - \Pi_K) \zeta^n |^2)
	&\geq \nu \dt (|\nabla \Pi_K \zeta^n |^2 + \lambda_{K+1} |(I - \Pi_K) \zeta^n|^2) 
	\notag\\
	&\geq \nu \dt |\nabla \Pi_K \zeta^n |^2 + 2 \beta \dt |(I - \Pi_K) \zeta^n|^2,
	\end{align}
	where in the last inequality we invoked the hypotheses that $\nu \lambda_{K+1} \geq 2 \beta$, \eqref{cond:K:beta}.
	After rearranging terms, we deduce that
	\begin{align}\label{ineq:zeta}
	|\zeta^n|^2 - |\zeta^{n-1}|^2 + |\zeta^n - \zeta^{n-1}|^2 + \nu \dt |\nabla \Pi_K \zeta^n |^2 + \beta \dt |\zeta^n|^2 
	\leq \tilde c \frac{\dt}{(\nu  \beta)^{1/2}} |\zeta^{n-1}|^2 |\nabla \xi^n |^2.
	\end{align}
	In particular, after ignoring the third and fourth terms from the left-hand side of \eqref{ineq:zeta}, we obtain
	\begin{align*}
	(1 + \beta \dt) |\zeta^n|^2 
	\leq \left( 1 + \tilde c \dt\frac{|\nabla \xi^n |^2}{(\nu  \beta)^{1/2}} \right)|\zeta^{n-1}|^2 \quad \mbox{ for all } n \in \NN.
	\end{align*}
	Therefore, by induction,
	\begin{align}\label{ineq:zeta:2}
	|\zeta^n|^2 
	\leq \frac{|\zeta_0|^2}{(1 + \beta \dt)^n} \prod_{j=1}^n \left( 1 + \tilde c \dt \frac{|\nabla \xi^j |^2}{(\nu  \beta)^{1/2}} \right) 
	\leq \frac{|\zeta_0|^2}{(1 + \beta \dt)^n} \exp\left( \sum_{j=1}^n \tilde c \dt \frac{|\nabla \xi^j |^2}{(\nu  \beta)^{1/2}} \right), 
	\end{align}
	where in the last inequality we used that $1 + x \leq e^x$, for all $x \in \mathbb{R}$. Taking expected values on both sides of \eqref{ineq:zeta:2}, we thus obtain
	\begin{align}\label{ineq:zeta:3}
	\bE |\zeta^n|^2 
	\leq
	\frac{|\zeta_0|^2}{(1 + \beta \dt)^n} \bE \exp\left( \sum_{j=1}^n \tilde{c} \dt \frac{|\nabla \xi^j |^2}{(\nu  \beta)^{1/2}} \right).
	\end{align}
	
	Now let $\sm = \tilde c (\nu^3  \beta)^{-1/2}$. From assumption \eqref{cond:beta:0} on $\beta$ it is clear that $\sm$ satisfies condition \eqref{condgammasup} from  \cref{prop:exp:bounds:sup}. Thus, from \eqref{expbounds} and \eqref{ineq:zeta:3} we obtain that
	\begin{align}\label{E:diff:xi:txi:1}
	\bE |\zeta^n|^2 
	\leq \frac{| \zeta_0 |^2}{(1 + \beta \dt)^n} \frac{ \tilde{c} \exp \left( C \alpha |\xi_0|^2 \right)}{(1 - 4 \alpha \dt | \sigma|^2)^{n/4}}.
	\end{align}
	
	Moreover, we can assume that the constant $c$ in assumption \eqref{cond:beta:0} on $\beta$ is large enough so that $\sm =  \tilde c (\nu^3  \beta)^{-1/2} \leq 1/(8 \dt_0 |\sigma|^2)$ and $\beta \geq 1/\dt_0$. 
	We then estimate
	\begin{align*}
		1 - 4 \alpha \dt | \sigma|^2 &\geq 1 - \frac{\dt}{2 \dt_0}
		= 1 - \frac{1}{2 \dt_0 \beta} \beta \dt \\
		&\geq 1 - \frac{1}{1 + \dt_0 \beta} \beta \dt
		\geq 1 - \frac{1}{1 + \beta \dt} \beta \dt
		= \frac{1}{1 + \beta \dt},
	\end{align*}
	where in the last inequality we used that $\dt \leq \dt_0$.
	Therefore,
	\begin{align*}
	(1 + \beta \dt )^n (1 - 4 \sm \dt |\sigma|^2 )^{n/4} 
	\geq (1 + \beta \dt )^n (1 + \beta \dt)^{-n/4} 
	= (1 + \beta \dt)^{3n/4},
	\end{align*}
	so that from \eqref{E:diff:xi:txi:1} we deduce
	\begin{align}\label{E:diff:xi:txi:2}
	\bE |\zeta^n|^2 
	\leq |\zeta_0|^2  \frac{\tilde{c} \exp \left( C \alpha |\xi_0|^2 \right)}{(1 + \beta \dt)^{3n/4}}.
	\end{align}
	This shows \eqref{bound:diff:xi:txi} and concludes the proof.
\end{proof}

We next recall some additional notions of distance in the space of probability measures on any measurable space $(X, \Sigma_X)$, along with some useful related inequalities. These will be particularly helpful in further estimating the second term in the right-hand side of \eqref{trg:ineq:wass}, i.e. the cost-of-control term. First, we recall that the \textit{total variation} distance between any two measures $\mu, \tmu \in \Pr(X)$ is defined as 
\begin{align}\label{def:TV}
\tv{\mu - \tmu} \coloneqq \sup_{A \in \Sigma_X} |\mu(A) - \tmu(A)|.
\end{align}
Given another measurable space $(Y, \Sigma_Y)$ and a measurable function $\phi: X \to Y$, it follows immediately from definition \eqref{def:TV} that 
\begin{align}\label{ineq:pfwd:tv}
\tv{\phi^*\mu - \phi^*\tmu} \leq \tv{\mu - \tmu},
\end{align}
where here $\phi^*\mu \in \Pr(Y)$ denotes the \emph{pushforward measure} of $\mu$ by the function $\phi$, i.e. $\phi^*\mu(A) \coloneqq \mu(\phi^{-1}(A))$ for all $A \in \Sigma_Y$.

Secondly, we recall that the \textit{Kullback-Leibler divergence} is defined as 
\begin{align}\label{def:KL}
\KL{\tmu}{\mu} := \int_X \ln \left( \frac{d \tmu }{d \mu} (\xi )\right) \tmu (d\xi),
\end{align}
for any $\mu, \tmu \in \Pr(X)$ such that $\tmu$ is absolutely continuous with respect to $\mu$, so that the Radon-Nikodym derivative $d \tmu/ d\mu$ is well-defined. When $\tmu$ is not absolutely continuous with respect to $\mu$, we set $\KL{\tmu}{\mu} := + \infty$.

Regarding these two notions of distance, we will make use of two useful inequalities from \cite{ButkovskyKulikScheutzow2019} providing estimates on the distance between the law of a $d$-dimensional Wiener process $W$ and the corresponding shifted process 
\begin{align}\label{hW:W:phi}
    \hW(t) = W(t) + \int_0^t \varphi(\tau) d\tau 
\end{align}
for some progressively measurable process $\varphi(t)$, $t \geq 0$. In the proofs below, these inequalities will be applied with $\varphi (t)=  \sum_{j=1}^\infty \psi_j \ind_{[t_{j-1},t_j)}(t)$, for $\psi_j$ as given in \eqref{def:shifted:W:2}. Specifically, denoting by $\mL(W)$ and $\mL(\hW)$ the laws of $W$ and $\hW$, respectively, it follows from \cite[Theorem A.2]{ButkovskyKulikScheutzow2019} that
\begin{align}\label{ineq:kl:phi:0}
    \KL{\mL(\hW)}{ \mL(W)}
    \leq 
    \frac{1}{2} \bE \int_0^\infty |\varphi(t)|^2 \rd t.
\end{align}
And from \cite[Theorem A.5, (A.13)]{ButkovskyKulikScheutzow2019}, we have that for any $a \in (0,1]$\footnote{In \cite[Theorem A.5, (A.13)]{ButkovskyKulikScheutzow2019}, it is actually assumed $a \in (0,1)$. In fact, inequality \eqref{ineq:tv:phi:0} also holds with $a=1$, although a slightly sharper bound is valid in this case due to \eqref{ineq:kl:phi:0} and Pinsker's inequality (see e.g. \cite[Lemma 2.5.(i)]{Tsybakov2009}). Namely, $\tv{\mL(\hW) -  \mL(W)} \leq \sqrt{\frac{1}{2} \KL{\mL(\hW)}{\mL(W)}} \leq \frac{1}{2} \left( \bE \int_0^\infty |\varphi(t)|^2 \rd t\right)^{1/2}$.}
\begin{align}\label{ineq:tv:phi:0}
    \tv{\mL(\hW) -  \mL(W)}
	 &\leq 2^{\frac{1-a}{1+a}} \left[ \bE \left( \int_0^\infty |\varphi(t)|^2 \rd t \right)^a \right]^{\frac{1}{1+a}}.
\end{align}

We also recall the following inequality providing an explicit relation between these two definitions (see e.g. \cite[inequality (2.25)]{Tsybakov2009}):
\begin{align}\label{ineq:tv:kl:2}
\tv{ \mu - \tmu } \leq 1 - \frac{1}{2} \exp \left( - \KL{\tmu}{\mu} \right)
\end{align}
for all $\mu, \tmu \in \Pr(X)$.

To further connect these definitions with the Wasserstein-like distances defined in \eqref{def:Wass:rhoe} on $\Pr(\dL^2)$, we notice that for any distance-like function $\rho: \dL^2 \times \dL^2 \to \RR^+$ such that $\rho(\xi,\txi) \leq 1$ for all $\xi,\txi \in \dL^2$, it follows as an immediate consequence of the coupling lemma \cite[Lemma 1.2.24]{KuksinShirikyan12} that 
\begin{align}\label{ineq:wass:TV}
\Wass_\rho (\mu, \tmu) \leq \tv{\mu - \tmu} \quad \mbox{ for all } \mu, \tmu \in \Pr(\dL^2). 
\end{align}

With these notations and facts in place, we now proceed with the proofs of \cref{prop:small:Wass} and \cref{prop:contr:Wass}.

\begin{proof}[Proof of \cref{prop:small:Wass}]
Fix $M > 0$, $\varepsilon >0$, $s \in (0,1]$, and let $\xi_0, \txi_0 \in \dL^2$ such that $|\xi_0| \leq M$ and $|\txi_0| \leq M$. 
	
We start with the triangle inequality as in \eqref{trg:ineq:wass} and provide an estimate of each term in the right-hand side by following the strategy described in the introduction to this section. For the first term, it follows from the definition of $\Wes$ according to \eqref{def:Wass:rhoe} and \eqref{def:rhoes}, along with H\"older's inequality, that
\begin{align}\label{W:dP:dtP:0}
	\Wes (\Pdisc(\xi_0,\cdot), \tPdisc(\txi_0,\cdot)) 
	&\leq \bE \left( 1 \wedge \frac{|\xi^n(\Pi_N \xi_0) - \txi^n( \Pi_N \xi_0; \Pi_N \txi_0)|^s}{\varepsilon} \right) 
	\notag\\
	&\leq \frac{1}{\varepsilon} \left( \bE |\xi^n( \Pi_N \xi_0) - \txi^n(\Pi_N \xi_0; \Pi_N \txi_0)|^2 \right)^{s/2}.
\end{align}
By assumption, there exists $K \in \NN$ such that \eqref{cond:K:sigma} holds. In particular, the second condition in \eqref{cond:K:sigma} implies that we can take $\beta > 0$ satisfying assumptions \eqref{cond:K:beta} and \eqref{cond:beta:0} of \cref{lem:bound:diff:xi:txi}.
It thus follows from \eqref{bound:diff:xi:txi} and \eqref{W:dP:dtP:0} that 
\begin{align}
	\Wes (\Pdisc(\xi_0,\cdot), \tPdisc(\txi_0,\cdot))
	&\leq \frac{|\xi_0 - \txi_0|^s}{\varepsilon} \frac{\tilde{c}^s \exp\left( C s (\nu^3 \beta)^{-1/2} |\xi_0|^2 \right)}{(1 + \beta \dt)^{3n s/8}}
	\label{W:dP:dtP:1:0}\\
	&= \tilde{c}^s \frac{M^s}{\varepsilon}  \frac{ \exp\left( C s (\nu^3  \beta)^{-1/2} M^2 \right)}{(1 + \beta \dt)^{3n s/8}},
	\label{W:dP:dtP:1}
\end{align}
for some absolute constant $\tilde{c} > 0$ and some constant $C> 0$ depending only on $\nu, \dt_0$.

We proceed to estimate the second term in the right-hand side of \eqref{trg:ineq:wass}. Let us denote by $\xi^n(\Pi_N \txi_0;W)$ and $\txi^n(\Pi_N \txi_0; \Pi_N \xi_0, W)$ the solutions of \eqref{disc2DSNSEv} and \eqref{eq:nudging:1}-\eqref{eq:nudging:2}, respectively, starting from $\Pi_N \txi_0 \in \HN$ and corresponding to the family $W = \{W^k\}_{k=1}^d$ of independent real-valued Brownian motions $W^k$, $k=1, \ldots, d$. Then, denoting by $\mL(Z)$ the law of a random variable $Z$, we can equivalently write the Markov transition kernels defined in \eqref{def:Pdisc} and \eqref{def:control:P} as 
\begin{align*}
\Pdisc(\txi_0, \cdot) = \mL(\xi^n(\Pi_N \txi_0;W)) \quad \mbox{ and } \quad \tPdisc(\txi_0, \cdot) = \mL(\txi^n(\Pi_N \txi_0;\Pi_N \xi_0,W)),
\end{align*}
respectively. From inequality \eqref{ineq:wass:TV}, we thus have
\begin{align}\label{est:sec:term}
	\Wes ( \tPdisc(\txi_0, \cdot), \Pdisc(\txi_0, \cdot))
	&\leq 
	\tv{\tPdisc(\txi_0, \cdot) -  \Pdisc(\txi_0, \cdot)} \notag\\
	&= \tv{\mL(\txi^n(\Pi_N \txi_0; \Pi_N \xi_0, W))  -  \mL(\xi^n(\Pi_N \txi_0;W)) }.
\end{align}

Let $\hW = \{\hW^k\}_{k=1}^d$ be the family of shifted independent Brownian motions defined in \eqref{def:shifted:W:1}-\eqref{def:shifted:W:2}. Here notice that, in the definition of $\psi_j$ in \eqref{def:shifted:W:2}, $\sigma^{-1} \Pi_K$ is well-defined due to the assumption that $\HK \subset \range{\sigma}$ in \eqref{cond:K:sigma}.
Moreover, due to the uniqueness of pathwise strong solutions of \eqref{disc2DSNSEv}, as shown in \cref{prop:wellposed:full:disc}, it follows that $\xi^n(\Pi_N \txi_0;\hW) = \txi^n(\Pi_N \txi_0; \Pi_N \xi_0, W)$ for all $n \in \NN$ almost surely. Hence, from \eqref{est:sec:term}, 
\begin{align}\label{bound:sec:term}
	\Wes ( \tPdisc(\txi_0, \cdot), \Pdisc(\txi_0, \cdot))
	\leq
	\tv{\mL(\xi^n(\Pi_N \txi_0;\hW)) -  \mL(\xi^n(\Pi_N \txi_0; W))}. 
\end{align}

It is not difficult to show that $W \in \RR^d \mapsto \xi^n(\Pi_N \txi_0;W)$ is a continuous mapping. It thus follows from \eqref{ineq:pfwd:tv} that
\begin{align}\label{ineq:tv:hW:W}
	\tv{\mL(\xi^n(\Pi_N \txi_0;\hW))  -   \mL(\xi^n(\Pi_N \txi_0; W))}
	&=
	\tv{\xi^n(\Pi_N \txi_0;\cdot)^* \mL(\hW) - \xi^n(\Pi_N \txi_0;\cdot)^*\mL(W)} \notag\\
	&\leq 
	\tv{\mL(\hW) -  \mL(W)}.
\end{align}
Together with inequality \eqref{ineq:tv:kl:2}, we thus have
\begin{align}\label{ineq:tv:hW:W:2}
	\tv{\mL(\xi^n(\Pi_N \txi_0;\hW))  -   \mL(\xi^n(\Pi_N \txi_0; W))}
	\leq
	1 - \frac{1}{2} \exp \left( -\KL{\mL(\hW)}{ \mL(W)} \right).
\end{align}

Recalling from \eqref{def:shifted:W:1} that
\begin{align*}
\hW(t) = W(t) + \int_0^t \varphi(\tau) d \tau, \quad \mbox{ where }  \varphi(\tau) \coloneqq \sum_{j=1}^\infty \psi_j \ind_{[t_{j-1},t_j)}(\tau),
\end{align*}
with $\psi_j$ as defined in \eqref{def:shifted:W:2}, we now invoke \eqref{ineq:kl:phi:0} and obtain that 
\begin{align}\label{ineq:kl:phi}
\KL{\mL(\hW)}{ \mL(W)}
\leq 
\frac{1}{2} \bE \int_0^\infty |\varphi(t)|^2 \rd t
=
\bE \dt \sum_{j=1}^\infty |\psi_j|^2 .
\end{align}

Now invoking the fact that $\sigma^{-1}$ is bounded, and once again \cref{lem:bound:diff:xi:txi}, we further estimate the right-hand side above as
\begin{align}\label{ineq:int:phi:b:0}
\bE \dt \sum_{j=1}^\infty |\psi_j|^2 
&\leq
\dt \beta^2 \|\sigma^{-1}\|^2 \sum_{j=1}^\infty \bE | \txi^j(\Pi_N \txi_0; \Pi_N \xi_0,W) - \xi^j(\Pi_N \xi_0;W) |^2
\notag \\
&\leq
\dt \beta^2 \|\sigma^{-1}\|^2 \sum_{j=1}^\infty |\xi_0 - \txi_0|^2 \frac{ \tilde{c} \exp\left( C (\nu^3 \beta)^{-1/2} |\xi_0|^2 \right)}{(1 + \beta \dt)^{3j/4}}
\end{align}
for some constant $C > 0$, and where $\|\sigma^{-1}\|$ denotes the operator norm of $\sigma^{-1}$. Notice that
\begin{align*}
\beta \dt \sum_{j=1}^\infty \frac{1}{(1 + \beta \dt)^{3j/4}} 
=
\beta \dt \frac{1}{1 -(1 + \beta \dt)^{-3/4}}
&\leq
\beta \dt \frac{1}{1 -(1 + \beta \dt)^{-1/2}}
= \beta \dt \frac{(1 + \beta \dt)^{1/2}}{(1 + \beta \dt)^{1/2} - 1} \\
&= \beta \dt \frac{(1 + \beta \dt)^{1/2}  [(1 + \beta \dt)^{1/2} + 1] }{\beta \dt}
\leq 2  (1 + \beta \dt), 
\end{align*}
so that, from \eqref{ineq:int:phi:b:0},
\begin{align}
\frac{1}{2} \bE \int_0^\infty |\varphi(t)|^2 \rd t = \bE \dt \sum_{j=1}^\infty |\psi_j|^2 
\leq
\tilde{c} \beta (1 + \beta \dt) \|\sigma^{-1}\|^2 |\xi_0 - \txi_0|^2 \exp\left( C (\nu^3  \beta)^{-1/2} |\xi_0|^2 \right)
\label{ineq:int:phi:c:0:0} \\
\leq 
\tilde{c} \beta (1 + \beta \dt) \|\sigma^{-1}\|^2 M^2 \exp\left( C (\nu^3  \beta)^{-1/2} M^2 \right).
\label{ineq:int:phi:c:0}
\end{align}
From \eqref{bound:sec:term}-\eqref{ineq:kl:phi} and \eqref{ineq:int:phi:c:0}, we thus have
\begin{align}\label{bound:sec:term:b}
	\Wes& ( \tPdisc(\txi_0, \cdot), \Pdisc(\txi_0, \cdot))
	\notag\\
	&\leq
	1 - \frac{1}{2} \exp \left\{ - \tilde{c} \beta (1 + \beta \dt) \|\sigma^{-1}\|^2 M^2 \exp\left( C (\nu^3  \beta)^{-1/2} M^2 \right) \right\}.
\end{align}

Hence, combining inequality \eqref{trg:ineq:wass} with the estimates \eqref{W:dP:dtP:1} and \eqref{bound:sec:term:b}, it follows that
\begin{align}\label{ineq:small:Wass:0}
    \Wes ( \Pdisc(\txi_0, \cdot), \Pdisc(\txi_0, \cdot))
	&\leq
	\tilde{c}^s \frac{M^s}{\varepsilon}  \frac{\exp\left( C s (\nu^3 \beta)^{-1/2} M^2 \right)}{(1 + \beta \dt)^{3n s/8}}
	\\
	&\qquad +
	1 - \frac{1}{2} \exp \left\{ - \tilde{c} \beta \left(1 + \beta \dt_0 \right) \|\sigma^{-1}\|^2 M^2 \exp\left( C (\nu^3  \beta)^{-1/2} M^2 \right) \right\},
	\notag
\end{align}
where we have used that $\dt \leq \dt_0$ to further estimate the right-hand side of \eqref{bound:sec:term:b}. 

To arrive at \eqref{ineq:small:Wass}, we use the fact that for any constant $0 < a < 1$ we have $\ln(1 + x) \geq ax $ for all $x \in [0, (1/a) - 1]$. In particular, taking $a = 1/(1 + \beta \dt_0)$ we have $\beta \dt \leq \beta \dt_0 = (1/a) - 1$, so that
\begin{align}\label{ineq:n:dt}
\frac{1}{(1 + \beta \dt)^{3ns/8}} = \exp \left( - \frac{3ns}{8} \ln (1 + \beta \dt) \right)
\leq \exp \left( - \frac{3s}{8} \frac{\beta}{1 + \beta \dt_0} n \dt \right). 
\end{align}
Hence, we can fix a time $T_1 > 0$ depending on $M, s, \varepsilon, \beta, \sigma, \dt_0$ such that for all $n \in \NN$ with $n \dt \geq T_1$ the first term in the right-hand side of \eqref{ineq:small:Wass:0} can be estimated as
\begin{align}\label{est:first:small}
    \tilde{c}^s \frac{M^s}{\varepsilon}  \frac{\exp\left( C s (\nu^3  \beta)^{-1/2} M^2 \right)}{(1 + \beta \dt)^{3n s/8}}
    &\leq
    \tilde{c}^s \frac{M^s}{\varepsilon} \exp\left( C s (\nu^3 \beta)^{-1/2} M^2 \right)  \exp \left( - \frac{3s}{8} \frac{\beta}{1 + \beta \dt_0} n \dt \right)
    \notag \\
    &\leq 
    \frac{1}{4}  \exp \left\{ - \tilde{c} \beta \left( 1 + \beta \dt_0 \right)  \|\sigma^{-1}\|^2 M^2 \exp\left( C (\nu^3 \beta)^{-1/2} M^2 \right) \right\}.
\end{align}

From \eqref{ineq:small:Wass:0} and \eqref{est:first:small}, we thus conclude
\begin{align*}
     \Wes ( \Pdisc(\txi_0, \cdot), \Pdisc(\txi_0, \cdot))
	&\leq
	1 -  \frac{1}{4}  \exp \left\{ - \tilde{c} \beta \left(1 + \beta \dt_0 \right) \|\sigma^{-1}\|^2 M^2 \exp\left( C (\nu^3  \beta)^{-1/2} M^2 \right) \right\}
	\\
	&\quad \eqqcolon 1 - \kappa_1.
\end{align*}
This shows \eqref{ineq:small:Wass} and concludes the proof. 
\end{proof}

\begin{proof}[Proof of \cref{prop:contr:Wass}]
Fix $\kappa_2 \in (0,1)$ and $r > 0$. Let $\xi_0, \txi_0 \in \dL^2$ satisfying $\rhoes(\xi_0, \txi_0) < 1$, for $\rhoes$ as defined in \eqref{def:rhoes}. Note that this implies $\rhoes(\xi_0, \txi_0) = |\xi_0 - \txi_0|^s/\varepsilon$. 

We proceed analogously as in the proof of \cref{prop:small:Wass}, starting with the inequality \eqref{trg:ineq:wass}. For the first term in the right-hand side of \eqref{trg:ineq:wass}, we first estimate as in \eqref{W:dP:dtP:1:0}. Then, choose $s \in (0,1]$ such that
\begin{align}\label{choice:s}
    \frac{ C s}{(\nu^3  \beta)^{1/2}} \leq r,
\end{align}
with $C > 0$ as in  \eqref{W:dP:dtP:1:0} and \eqref{ineq:int:phi:c:0:0}.
Here we recall that $\beta > 0$ is fixed so that assumptions \eqref{cond:K:beta} and \eqref{cond:beta:0} of \cref{lem:bound:diff:xi:txi} hold, which is possible due to the second condition in the standing assumption \eqref{cond:K:sigma}.

With this choice of $s$, it follows from \eqref{W:dP:dtP:1:0} that
\begin{align}\label{est:first:term:contr}
    \Wes (\Pdisc(\xi_0,\cdot), \tPdisc(\txi_0,\cdot))
    \leq
    \tilde{c}^s \frac{|\xi_0 - \txi_0|^s}{\varepsilon} \frac{\exp\left( r |\xi_0|^2 \right)}{(1 + \beta \dt)^{3n s/8}}.
\end{align}

For the second term in the right-hand side of \eqref{trg:ineq:wass}, we proceed as in \eqref{est:sec:term}-\eqref{ineq:tv:hW:W}, and then invoke \eqref{ineq:tv:phi:0} to obtain that for any $a \in (0,1]$
\begin{align}\label{est:sec:term:contr}
    \Wes ( \tPdisc(\txi_0, \cdot), \Pdisc(\txi_0, \cdot))
	\leq 
	\tv{\mL(\hW) -  \mL(W)}
	 &\leq 2^{\frac{1-a}{1+a}} \left[ \bE \left( \int_0^\infty |\varphi(t)|^2 \rd t \right)^a \right]^{\frac{1}{1+a}}.
\end{align}
By H\"older's inequality, together with estimate \eqref{ineq:int:phi:c:0:0}, we have
\begin{multline}\label{E:phi:a}
    2^{\frac{1-a}{1+a}} \left[ \bE \left( \int_0^\infty |\varphi(t)|^2 \rd t \right)^a \right]^{\frac{1}{1+a}} 
	\leq
	2^{\frac{1-a}{1+a}} \left[ \bE \int_0^\infty |\varphi(t)|^2 \rd t  \right]^{\frac{a}{1+a}}
	\\
	\leq 
	2^{\frac{1-a}{1+a}} \left( \tilde{c} \beta (1 + \beta \dt) \|\sigma^{-1}\|^2 \right)^{\frac{a}{1+a}} |\xi_0 - \txi_0|^{\frac{2a}{1+a}} \exp\left( \frac{C a}{(1+a) (\nu^3 \beta)^{1/2}} |\xi_0|^2 \right).
\end{multline}
In particular, choosing $a \in (0,1]$ such that $2a /(1 + a) = s$, with $s \in (0,1]$ as fixed in \eqref{choice:s}, it follows from \eqref{est:sec:term:contr}, \eqref{E:phi:a} and \eqref{choice:s} that 
\begin{align}\label{est:sec:term:contr:1}
	\Wes ( \tPdisc(\txi_0, \cdot), \Pdisc(\txi_0, \cdot))
	&\leq 	
	2^{1-s} \left( \tilde{c} \beta (1 + \beta \dt) \|\sigma^{-1}\|^2 \right)^{s/2} |\xi_0 - \txi_0|^s \exp \left( r |\xi_0|^2  \right)
	\notag\\
	&\leq
	2^{1-s} \left( \tilde{c} \beta \left(1 + \beta \dt_0 \right) \|\sigma^{-1}\|^2 \right)^{s/2} |\xi_0 - \txi_0|^s \exp \left( r |\xi_0|^2  \right).
\end{align}
Thus, from \eqref{trg:ineq:wass}, \eqref{est:first:term:contr}, \eqref{est:sec:term:contr:1}, and since $\rhoes(\xi_0, \txi_0) = |\xi_0 - \txi_0|^s/\varepsilon$, it follows that
\begin{multline}\label{ineq:contr}
    \Wes (\Pdisc(\xi_0, \cdot), \Pdisc(\txi_0, \cdot)) 
    \\
    \leq 
    \left[ \frac{\tilde{c}^s}{(1 + \beta \dt)^{3ns/8}} + \varepsilon 2^{1-s} \left( \tilde{c} \beta \left(1 + \beta \dt_0 \right) \|\sigma^{-1}\|^2 \right)^{s/2} \right]  \exp \left( r |\xi_0|^2  \right) \rhoes(\xi_0, \txi_0)
\end{multline}
for every $n \in \NN$.

In particular, by estimating $1/(1 + \beta \dt)^{3ns/8} \leq 1$, we deduce that \eqref{ineq:contr:Wass:0} holds with 
\begin{align*}
    C(\varepsilon,s) = \tilde{c}^s +  \varepsilon 2^{1-s} \left( \tilde{c} \beta \left(1 + \beta \dt_0 \right) \|\sigma^{-1}\|^2 \right)^{s/2}.
\end{align*}

Moreover, proceeding as in \eqref{ineq:n:dt} and choosing $T_2 = T_2(\kappa_2,r) > 0$ and $\varepsilon = \varepsilon(\kappa_2,r) > 0$ such that
\begin{align*}
	\tilde{c}^s \exp \left( - \frac{3s}{8} \frac{ \beta}{1 + \beta \dt_0} T_2  \right) + \varepsilon 2^{1-s} \left( \tilde{c} \beta \left( 1 + \beta \dt_0 \right) \|\sigma^{-1}\|^2 \right)^{s/2} 
	\leq 
	\kappa_2,
\end{align*}
we conclude from \eqref{ineq:contr} that \eqref{ineq:contr:Wass} holds for every $n \in \NN$ with $n \dt \geq T_2$, as desired.
\end{proof}

We conclude this section by combining the above results to deduce a proof of \cref{thm:sp:gap:disc:SNSE}.

\begin{proof}[Proof of \cref{thm:sp:gap:disc:SNSE}]
Fix $N \in \NN$ and $\dt \leq \dt_0$. Following the notation from \cref{thm:gen:sp:gap}, we take $(\gsp, \gn{\cdot}) = (\dL^2, |\cdot|)$, $\mI = \dt \ZZ^+$, $\{\gP_t\}_{t \in \mI}$ given by $\Pdisc$, $n \in \ZZ^+$, and $\gfd$ as the class of distance functions defined in \eqref{def:rhoes}. It follows from \cref{prop:Lyap:exp:disc}, \cref{prop:small:Wass} and \cref{prop:contr:Wass} that assumptions \ref{A:1}, \ref{A:2} and \ref{A:3} of \cref{thm:gen:sp:gap} are satisfied in this setting. Thus, from \eqref{Wass:contr:gen:2} we obtain that for every $m > 1$ there exists $\sm_m > 0$ such that for each $\sm \in (0, \sm_m)$ there exist $\varepsilon > 0$, $s \in (0,1]$, $T > 0$, and constants $C_1, C_2 > 0$ for which the following holds 
\begin{align}\label{sp:gap:disc:1}
	\Wesa(\mu \Pdisc, \tmu \Pdisc) \leq C_1 e^{-n \dt C_2} \Wass_{\varepsilon, s, \sm/m} (\mu, \tmu)
\end{align}
for every $\mu, \tmu \in \Pr(\dL^2)$, and all $n \in \ZZ^+$ such that $n \dt \geq T$. 

Now take any $t \in \RR^+$ with $t \geq \dt_0 + T$, and let $n_0 \coloneqq \inf_{n \in \ZZ^+} \{ n \dt_0 \geq T\}$. It follows that $(n_0 - 1)\dt_0 < T \leq t - \dt_0$, and hence $t \geq n_0 \dt_0 \geq n_0 \dt$ for all $\dt \leq \dt_0$. Thus there exists $n \in \ZZ^+$ for which $t \in [n\dt, (n+1) \dt)$, so that from definition \eqref{def:Pdisct} we have $\Pdisct = \Pdisc$. Therefore, 
\begin{align}\label{sp:gap:disc:2}
	\Wesa(\mu \Pdisct, \tmu \Pdisct) = \Wesa(\mu \Pdisc, \tmu \Pdisc) 
	&\leq C_1 e^{-n \dt C_2} \Wass_{\varepsilon, s, \sm/m} (\mu, \tmu) \notag \\
	&\leq C_1 e^{\dt_0 C_2} e^{-t C_2} \Wass_{\varepsilon, s, \sm/m} (\mu, \tmu)
\end{align}
for every $\mu, \tmu \in \Pr(\dL^2)$. Moreover, according to the dependence of the constants $C_1$, $C_2$ made explicit in the statement of \cref{thm:gen:sp:gap}, it follows from \cref{prop:Lyap:exp:disc} that $C_1$ and $C_2 $ depend \emph{only} on $m$, $\sm$, $T$, $\nu$, $|\sigma|$, $\dt_0$. Hence, we may take the supremum in \eqref{sp:gap:disc:2} with respect to $N \in \NN$ and $\dt$ in $(0,\dt_0]$ to conclude that \eqref{Wass:contr:Pdisc} holds for every $t \geq \dt_0 + T$. 
\end{proof}

\subsection{Finite-time strong error estimates for the numerical scheme}\label{subsubsec:fin:tim:err:L2}

In this section, we present an estimate of the error between a solution $\xi(t)$, $t  \geq 0$, of \eqref{2DSNSEv}, and a solution $\xfd^n$, $n \in \ZZ^+$, of the numerical scheme \eqref{disc2DSNSEv}, in a suitable strong sense. This will be used later in \cref{subsec:wk:conv:NSE} to show a uniform weak convergence result for the family of Markov semigroups $\{\Pdisc\}_{n \in \ZZ^+}$, defined in \eqref{def:Pdisc:0}, as an application of \cref{thm:gen:wk:conv:2}. Specifically, it will be used to verify assumption \ref{H:4} in \cref{thm:gen:wk:conv}.

For this purpose, we split the error $|\xi(n \dt) - \xfd^n|$ into the spatial discretization error $|\xi(n\dt) - \xn(n\dt)|$ and the time discretization error $|\xn(n\dt) - \xfd^n|$. Here we recall that $\xn(t)$, $t \geq 0$, denotes a solution of the spectral Galerkin discretization scheme \eqref{2DSNSEvGal}. Concretely, we obtain a strong $L^2(\Omega)$ estimate of the spatial discretization error with respect to the topology in $L^\infty_{\text{loc},t} L^2_x$. For the time discretization error, due to limitations associated to the nonlinear terms in \eqref{2DSNSEvGal} and \eqref{disc2DSNSEv}, we are only able to obtain strong convergence in $L^\smx(\Omega;L^\infty_{\text{loc},t} L^2_x)$ for sufficiently small $\smx > 0$. As we show later in \cref{prop:fin:time:err:wass}, this is however compensated in the Wasserstein error estimate by the presence of the Lyapunov function $\xi \mapsto \exp(\sm |\xi|^2)$ in the definition of $\rhoesa$, thanks to the associated Lyapunov inequalities from \cref{prop:Lyap:exp:disc} and \cref{prop:exp:mom:cont:2}.

We start by providing an estimate of the spatial discretization error.

\begin{Prop}\label{prop:spatial:error}
	Fix any $N \in \NN$, $\sigma \in \bdH^1$, and $\xi_0 \in \dH^1$. Let $\xi = \xi(t)$ and $\xn = \xn(t)$ be the solutions of \eqref{2DSNSEv} and \eqref{2DSNSEvGal}, satisfying $\xi(0) = \xi_0$ and $\xn(0) = \Pi_N \xi_0$ almost surely, respectively. Then, for every $\sm$ satisfying
	\begin{align*}
	0 < \sm \leq \frac{\nu }{2|\sigma|^2},
	\end{align*}
	it follows that, for every $T > 0$,
	\begin{align}\label{ineq:lem:spatial:error}
		\bE \sup_{t \in [0,T]} | \xi(t) - \xn(t)|^2 \leq	\frac{C}{N} \left[  \exp(c \sm |\xi_0|^2 ) +  |\nabla \xi_0|^2 \right],
	\end{align}
	for some positive constant $C$ depending only on $\nu,|\sigma|, |\nabla \sigma|, \sm, T$.
\end{Prop}
\begin{proof}
	We start by estimating the spatial discretization error by its low mode and high mode components, namely
	\begin{align}\label{split:err:high:low}
	\bE \sup_{t \in [0,T]} | \xi(t) - \xn(t)|^2 
	&= \bE \sup_{t \in [0,T]} | (\Pi_N \xi(t) - \xn(t)) + (I - \Pi_N) \xi(t) |^2 \notag \\
	&\leq 2\bE \sup_{t \in [0,T]} | \Pi_N \xi(t) - \xn(t)|^2 + 2 \bE \sup_{t \in [0,T]} |(I - \Pi_N) \xi(t)|^2.
	\end{align}
	For the second term, it follows from \eqref{ineq:Poincare:N} and the analogous version of the bound \eqref{ineq:sup:nabla:xn} for $\xi(t)$, $t \geq 0$, that 
	\begin{align}\label{ineq:E:QN:xi}
	\bE \sup_{t \in [0,T]} |(I - \Pi_N) \xi(t)|^2
	\leq
	\lambda_{N+1}^{-1} \bE \sup_{t \in [0,T]} |\nabla \xi(t)|^2
	\leq N^{-1} C (1 + |\xi_0|^4 + |\nabla \xi_0|^2 ),
	\end{align}
	for some positive constant $C = C(\nu, T, |\sigma|, |\nabla \sigma|)$, where we have also used that $\lambda_j \sim j$ as recalled in \cref{subsubsec:2DSNSE}.
	
	We proceed to estimate the first term in \eqref{split:err:high:low}. Let us denote $\zn = \Pi_N \xi - \xn$ and $\bv_N = \mK \ast \zn$. Applying the projection $\Pi_N$ to \eqref{2DSNSEv}, we have
	\begin{align}\label{2DSNSEv:proj}
	\rd \Pi_N \xi + \left[ - \nu \Delta \Pi_N \xi + \Pi_N (\bu \cdot \nabla \xi) \right] \rd t = \sum_{k=1}^d \Pi_N \sigma_k \rd W^k.
	\end{align}
	Subtracting \eqref{2DSNSEvGal} from \eqref{2DSNSEv:proj}, we obtain that $\zn$ satisfies
	\begin{align*}
	\rd \zn + \left[ - \nu \Delta \zn + \Pi_N \left( \bu \cdot \nabla \xi - \bu_N \cdot \nabla \xn \right) \right] \rd t = 0.
	\end{align*}
	Hence, it follows by It\^o formula that
	\begin{align}\label{eq:energy:zN}
	\rd |\zn|^2 + 2 \nu | \nabla \zn |^2 \rd t = - 2 \left( \bu \cdot \nabla \xi - \bu_N \cdot \nabla \xn , \zn \right) \rd t.
	\end{align}
	Denote $\QN = I - \Pi_N$. Notice that
	\begin{align*}
	\bu \cdot \nabla \xi - \bu_N \cdot \nabla \xn = 
	\bu \cdot \nabla \QN \xi + \QN \bu \cdot \nabla \Pi_N \xi + \Pi_N \bu \cdot \nabla \zn + \bv_N \cdot \nabla \Pi_N \xi - \bv_N \cdot \nabla \zn.
	\end{align*}
	Thus, due to the orthogonality property \eqref{nonlin:ort},
	\begin{align}\label{eq:nonl:term}
	\left( \bu \cdot \nabla \xi - \bu_N \cdot \nabla \xn , \zn \right) = \left( \bu \cdot \nabla \QN \xi , \zn \right) + \left( \QN \bu \cdot \nabla \Pi_N \xi ,\zn \right) + \left( \bv_N \cdot \nabla \Pi_N \xi, \zn \right).
	\end{align}
	We proceed to estimate each term in the right-hand side of \eqref{eq:nonl:term}. Invoking \eqref{nonlin:ort:0}, \eqref{ineq:nonlin:b}, and \eqref{ineq:Poincare:N}, we obtain
	\begin{multline}\label{est:nonlin:zN:1}
	|\left( \bu \cdot \nabla \QN \xi , \zn \right)| = |\left( \bu \cdot \nabla \zn, \QN \xi \right)|  \leq c |\xi|^{1/2} |\nabla \xi|^{1/2} |\nabla \zn| | \QN \xi| \\
	\leq \frac{c}{ \lambda_{N+1}^{1/2}} |\xi|^{1/2} |\nabla \xi|^{3/2} |\nabla \zn| 
	\leq \frac{\nu}{6} |\nabla \zn|^2 + \frac{c}{\nu \lambda_{N+1}} |\xi| \, |\nabla \xi|^3.
	\end{multline}
	
	Moreover, it follows from \eqref{ineq:nonlin:a:0} with $a = 1/2$ that
	\begin{align}\label{est:nonlin:zN:2}
	\left| \left( \bv_N \cdot \nabla \Pi_N \xi , \zn \right) \right|  \leq  c  |\nabla \Pi_N \xi | \,  |\zn|^{3/2} \,  |\nabla \zn|^{1/2}  
	\leq \frac{\nu}{6} |\nabla \zn|^2  +  \frac{c}{\nu^{1/3}} |\nabla \xi|^{4/3} |\zn|^2,
	\end{align}
	and
	\begin{multline} \label{est:nonlin:zN:3}
	\left|  \left( \QN \bu \cdot \nabla \Pi_N \xi ,\zn \right)   \right|   
	=   \left|  \left( \QN \bu \cdot \nabla \zn ,  \Pi_N \xi  \right)   \right|  
	\leq c  |\QN \xi | \, |\nabla \zn| \,  |\Pi_N \xi |^{1/2} \,  |\nabla \Pi_N \xi |^{1/2}  \\
	\leq \frac{c}{\lambda_{N+1}^{1/2}}  |\xi|^{1/2} |\nabla \xi|^{3/2} |\nabla \zn| 
	\leq \frac{\nu}{6} |\nabla \zn|^2 + \frac{c}{\nu \lambda_{N+1}} |\xi| \, |\nabla \xi|^3.
	\end{multline}
	Hence, from \eqref{eq:energy:zN}, we obtain
	\begin{align*}
	\frac{d}{dt} |\zn|^2  +  \nu  |\nabla \zn|^2  \leq  \frac{c}{\nu^{1/3}} |\nabla \xi|^{4/3}  |\zn|^2   + \frac{c}{\nu \lambda_{N+1}} |\xi| \, |\nabla \xi|^3.
	\end{align*}
	Ignoring the second term in the right-hand side and applying Gronwall's inequality, recalling that $\zn(0) = 0$, it follows that
	\begin{align}\label{ineq:Gronw:zN}
	|\zn (t)|^2   \leq   \frac{c}{\nu  \lambda_{N+1}}  \int_0^t  |\xi(s)| \, |\nabla \xi(s)|^3  \exp \left( \frac{c}{\nu^{1/3}} \int_s^t  |\nabla \xi(\tau) |^{4/3}  d \tau  \right)  ds.
	\end{align}
	For some $\sm > 0 $ to be appropriately chosen later, we estimate
	\begin{align}\label{ineq:grad:xi}
	\frac{c}{\nu^{1/3}}  \int_s^t  |\nabla \xi(\tau) |^{4/3}  d \tau   
	\leq  \int_s^t  \left(  \frac{\sm \nu }{2}  |\nabla \xi(\tau)|^2  + C_\sm  \right)  d \tau 
	\leq  \frac{\sm \nu}{2}  \int_0^T |\nabla \xi(\tau)|^2  d \tau   +   C_\sm T,
	\end{align}
	where $C_\sm = c(\sm^2 \nu^3)^{-1}$ for some positive absolute constant $c$. 
	
	Plugging \eqref{ineq:grad:xi} into \eqref{ineq:Gronw:zN}, and taking the supremum over $t \in [0,T]$, expected values, and applying H\"older's inequality, it follows that
	\begin{align}\label{ineq:bE:zN}
	\bE&  \sup_{t \in [0,T]} |\zn(t)|^2  \notag\\
	&\leq  c  \frac{ e^{C_\sm T}}{\nu \lambda_{N+1}}   \left(  \bE \exp \left(  \sm \nu \int_0^T  |\nabla \xi(\tau)|^2 d \tau   \right)  \right)^{1/2}  \left(   \bE  \left(  \int_0^T |\xi(s)| \, |\nabla \xi(s)|^3  ds   \right)^2   \right)^{1/2}.
	\end{align}
	
	Choosing $0 < \sm \leq \nu /(2| \sigma|^{2})$ and invoking \cref{prop:exp:mom:cont}, we estimate the first term between parentheses above as 
	\begin{align}\label{E:exp:nablaxi}
	\bE \exp \left(  \sm \nu \int_0^T  |\nabla \xi(\tau)|^2 d \tau   \right) 
	\leq 2 \exp \left( \sm|\xi_0|^2 \right) \exp \left( \sm |\sigma|^2 T \right).
	\end{align}
	
	For the last term in \eqref{ineq:bE:zN}, 
	we estimate 
	\begin{align*}	
	\bE  \left(  \int_0^T |\xi(s)| \, |\nabla \xi(s)|^3  ds   \right)^2
	&\leq
	\bE \left[ \sup_{t \in [0,T]} |\xi(t)|^2 |\nabla \xi(t)|^2 \left( \int_0^T |\nabla \xi(s)|^2 ds \right)^2 \right] \\
	&\leq
	C \bE \left[ \left( \sup_{t \in [0,T]} \left(  |\xi(t)|^2  + \nu \int_0^t |\nabla \xi(s)|^2 ds \right)^3 \right) \left( \sup_{t \in [0,T]}| \nabla \xi(t)|^2 \right) \right].
	\end{align*}
	Thus, after applying H\"older's inequality, we obtain from the analogous versions of inequalities \eqref{ineq:sup:nabla:xn} and \eqref{ineq:sup:xn:Hk} with $k=0$ satisfied by $\xi(t)$ that
	\begin{align}\label{E:xi:nablaxi:3}
	\bE  \left(  \int_0^T |\xi(s)| \, |\nabla \xi(s)|^3  ds   \right)^2
	&\leq
	C (1 + |\xi_0|^6) (1 + |\xi_0|^4 + |\nabla \xi_0|^2) \notag \\
	&\leq C \exp(c\sm |\xi_0|^2 ) (1 + |\nabla \xi_0|^2 ), 
	\end{align}
	for some positive constant $C$ depending on $\nu$, $\sm$, $T$, $|\nabla \sigma|$.
	
	Plugging \eqref{E:exp:nablaxi} and \eqref{E:xi:nablaxi:3} into \eqref{ineq:bE:zN}, we deduce that
	\begin{align}\label{ineq:E:zn}
	\bE  \sup_{t \in [0,T]} |\zn(t)|^2  
	\leq
	\frac{C}{N} \exp(c\sm |\xi_0|^2 ) (1 + |\nabla \xi_0| ),
	\end{align}
	where we again used that $\lambda_j \sim j$.
	
	Now combining \eqref{ineq:E:zn} with \eqref{ineq:E:QN:xi}, it follows from \eqref{split:err:high:low} that
	\begin{align*}
	\bE \sup_{t \in [0,T]} | \xi(t) - \xn(t)|^2 
	&\leq
	\frac{2 C}{N} \exp(c\sm |\xi_0|^2 ) (1 + |\nabla \xi_0| )
	+ \frac{2 C}{N}  (1 + |\xi_0|^4 + |\nabla \xi_0|^2 ) \\
	&\leq
	\frac{C}{N} \left[  \exp(c \sm |\xi_0|^2 ) + |\nabla \xi_0|^2 \right].
	\end{align*}
	This finishes the proof.
\end{proof}

We proceed by showing an estimate of the time discretization error as mentioned above. We note that a related result is obtained in \cite{BessaihMillet2019} (see also \cite{BessaihMillet2021,BessaihMillet2022}), where the authors consider instead the velocity formulation of the 2D stochastic Navier-Stokes equations subject to periodic boundary conditions and either multiplicative or additive noise as in \eqref{2DSNSEv}. In particular, for a semi-implicit Euler time discretization under additive noise analogously as in \eqref{disc2DSNSEv2}, their result yields a strong $L^2(\Omega)$ estimate of the discretization error for the approximating velocity fields under the topology of $L^\infty_{\text{loc},t} L^2_x$ and with order of convergence $1/4$. Our result below provides instead a strong $L^p(\Omega)$ bound for the approximating vorticity fields in $L^\infty_{\text{loc},t} L^2_x$, which implies an error bound for the corresponding velocity fields in $L^\infty_{\text{loc},t} H^1_x$, for $p$ sufficiently small, and with higher order of convergence $1/2$. Notably, according to \eqref{cond:sm:smx} below it follows that $p$ can be almost $2$ as $\nu \to \infty$. The main difference in our proof in relation to \cite{BessaihMillet2019} concerns the definition of the appropriate localization set in the sample space. Here, we consider a sequence of localization sets which are related to a suitable sequence of discrete stopping times, see \eqref{def:kappal} and \eqref{def:OmegalK} below.

\begin{Prop}\label{prop:time:disc:error}
	Fix any $N \in \NN$, $\dt, \dt_0 > 0$ with $\dt \leq \dt_0$, $\sigma \in \bdH^1$, and $\xi_0 \in \dH^2$. Let $\xn(t)$, $t \geq 0$, and $\xfd^n$, $n \in \ZZ^+$, be the solutions of \eqref{2DSNSEvGal} and \eqref{disc2DSNSEv} satisfying $\xn(0) = \Pi_N \xi_0$ and $\xfd^0 = \Pi_N \xi_0$, respectively. Then there exist positive absolute constants $c_1,c_2$ such that if
	\begin{align}\label{cond:sm:smx}
	0 < \sm \leq \frac{c_1}{|\sigma|^2} \min\left\{ \nu , \frac{1}{\dt_0}\right\}, 
	\quad \mbox{ and } \quad 
	0 < \smx <  \frac{2 \nu^2 \sm}{c_2 + \nu^2 \sm}, 
	\end{align}
	then for every $K \in \NN$ and $\tp \in (0,1/2)$
	\begin{align}\label{ineq:time:disc:error}
	\bE \sup_{k \leq K}|\xfd^k(\xi_0) - \xn(t_k; \xi_0)|^\smx \leq 
	C \dt^{\tp\smx} (1 + |\nabla \xi_0|^4 + |A\xi_0|^2)^{\smx/2}  \exp \left( \tC \sm |\xi_0|^2 \right),
	\end{align}
	where $\tC = c(1 + \nu \dt_0)$ 
	and $C = C( \tp, \smx, \nu, \dt_0, T,|\sigma|, |\nabla \sigma|,\sm)$, with $T \geq (K + 1)\dt$. Notably, $C$ and $\tC$ are independent of $N$ and $\dt$.
\end{Prop}
\begin{proof}
	For each $j \in \mathbb{N}$, let $\z^j_{N,\dt} = \z^j := \xfd^j - \xn(t_j) = \xi^j - \xn(t_j)$ and $\bv^j = \mK \ast \z^j$. Integrating \eqref{2DSNSEvGal} with respect to $t \in [t_j, t_{j+1}]$,
	\begin{align}\label{eq:xn:tn}
	\xn(t_{j+1}) - \xn(t_j) &+  \int_{t_j}^{t_{j+1}} \left[ - \nu \Delta \xn(s) + \Pi_N (\bu_N(s) \cdot \nabla \xn(s)) \right] \rd s  \notag\\
		&= \Pi_N \sigma (W(t_{j+1})  -  W(t_j)).
	\end{align}
	Thus, subtracting \eqref{eq:xn:tn} from \eqref{disc2DSNSEv} with $n=j+1$, we obtain
	\begin{align*}
	\z^{j+1} - \z^j - \nu \int_{t_j}^{t_{j+1}} (\Delta \xi^{j+1} - \Delta \xn(s)) \rd s  +  \int_{t_j}^{t_{j+1}}  \left[  \Pi_N  \left( \bu^j \cdot \nabla \xi^{j+1}  \right)  -   \Pi_N  (\bu_N(s) \cdot \nabla \xn(s))  \right] \rd s  =  0
	\end{align*}
	Taking the inner product with $\z^{j+1}$ in $\dL^2$ yields
	\begin{multline}\label{eq:energy:zn}
	|\z^{j+1}|^2 + |\z^{j+1} -  \z^j|^2  -  |\z^j|^2   = 
	2 \nu \int_{t_j}^{t_{j+1}}  (\Delta \xi^{j+1} - \Delta \xn(s), \z^{j+1}) \rd s  \\
	- 2  \int_{t_j}^{t_{j+1}}  \left( ( \bu^j \cdot \nabla \xi^{j+1} ) -   (\bu_N(s) \cdot \nabla \xn(s)) ,  \z^{j+1}  \right)\rd s
	\end{multline}
	Notice that
	\begin{align}\label{diff:dissp}
	2 \nu \int_{t_j}^{t_{j+1}} (\Delta \xi^{j+1} - \Delta \xn(s),  &\z^{j+1} ) \rd s \notag\\
	&=
	- 2 \nu \dt |\nabla \z^{j+1}|^2   +  2 \nu \int_{t_j}^{t_{j+1}} \left( \Delta \xn(t_{j+1}) - \Delta \xn(s),  \z^{j+1} \right) \rd s .
	\end{align}
	Integrating by parts the second term in the right-hand side of \eqref{diff:dissp}, then applying Cauchy-Schwarz and Young's inequalities, it follows that
	\begin{align}\label{diff:int}
	2 \nu &\left| \int_{t_j}^{t_{j+1}} \left( \Delta \xn(t_{j+1}) - \Delta \xn(s),  \z^{j+1} \right) \rd s \right| \notag\\
	&\qquad \qquad \leq  2\nu \int_{t_j}^{t_{j+1}} \! \! \! \! |\nabla \xn(t_{j+1}) - \nabla \xn(s)|  \, |\nabla \z^{j+1}|  \rd s \notag \\
	&\qquad \qquad \leq  \frac{\nu \dt}{4} |\nabla \z^{j+1}|^2  +  c \nu \int_{t_j}^{t_{j+1}}  \! \! \! \!  |\nabla \xn(t_{j+1}) - \nabla \xn(s)|^2 \rd s
	\end{align}
	
	Now for the second term in the right-hand side of \eqref{eq:energy:zn}, first notice that
	\begin{align}\label{zn:nonlin:term}
	&\left( ( \bu^j \cdot \nabla \xi^{j+1} )   -    (\bu_N(s) \cdot \nabla \xn(s)) , \z^{j+1} \right) \\
	&=
	\left( \bv^j \cdot \nabla  \xi^{j+1},  \z^{j+1} \right)  
	+   \left(  (\bu_N(t_j)  -  \bu_N(s)) \cdot \nabla \xn(t_{j+1}) ,  \z^{j+1}  \right)   -   \left( \bu_N(s) \cdot \nabla (\xn(s)  -   \xn(t_{j+1})),  \z^{j+1}   \right)
	\notag
	\end{align}
	
	We proceed to estimate each term in the right-hand side of \eqref{zn:nonlin:term}.  With \eqref{ineq:nonlin:a:0} and Young's inequality, we obtain
	\begin{align}\label{zn:nonlin:a}
	\left|  \left( \bv^j \cdot \nabla  \xi^{j+1},  \z^{j+1} \right)  \right| 
	&\leq  c  |\z^j| \, |\nabla \xi^{j+1}|  \,   |\z^{j+1}|^{1/2} |\nabla \z^{j+1}|^{1/2} \notag \\
	&\leq  c |\z^j| \, |\nabla \xi^{j+1}| \,  |\nabla \z^{j+1}|  \notag \\
	&\leq  \frac{\nu}{8} |\nabla \z^{j+1}|^2   +   \frac{c}{\nu}  |\z^j|^2 |\nabla \xi^{j+1}|^2.
	\end{align}
	Similarly,
	\begin{align}\label{zn:nonlin:c}
	\left|   \left( \bu_N(s) \cdot \nabla (\xn(s)  -   \xn(t_{j+1})),  \z^{j+1}   \right)  \right|
	\leq  \frac{\nu}{8} |\nabla \z^{j+1}|^2  +  \frac{c}{\nu} |\xn(s)|^2   |\nabla \xn(s)  -  \nabla \xn(t_{j+1}))|^2.
	\end{align}
	
	Now from \eqref{nonlin:ort:0} and \eqref{ineq:nonlin:b} we obtain
	\begin{align}\label{zn:nonlin:b}
	\left|  \left(  (\bu_N(t_j)  -  \bu_N(s)) \cdot \nabla \xn(t_{j+1}) ,  \z^{j+1}  \right)  \right| 
	&= \left|  \left(  (\bu_N(t_j)  -  \bu_N(s)) \cdot \nabla \z^{j+1}, \xn(t_{j+1})  \right)  \right|
	\notag \\ 
	&\leq c |\nabla \xn(t_j) - \nabla \xn(s)| \, |\nabla \z^{j+1}| \,  |\xn(t_{j+1})|
	\notag \\
	&\leq  \frac{\nu}{8} |\nabla \z^{j+1}|^2  +  \frac{c}{\nu}  | \nabla \xn(t_j)  -  \nabla \xn(s) |^2  |\xn(t_{j+1})|^2.
	\end{align}
	
	With \eqref{diff:dissp}-\eqref{zn:nonlin:b}, we obtain from \eqref{eq:energy:zn} the following inequality valid for every $j \in \mathbb{N}$
	\begin{multline}\label{eq:energy:zn:b} 
	|\z^{j+1}|^2 + |\z^{j+1} -  \z^j|^2  -  |\z^j|^2  +  \nu \dt |\nabla \z^{j+1}|^2  \\
	\leq
	c \nu \int_{t_j}^{t_{j+1}}   |\nabla \xn(t_{j+1}) - \nabla \xn(s)|^2 \rd s
	+  \frac{c \dt }{\nu}  |\z^j|^2 |\nabla \xi^{j+1}|^2\\
	+  \frac{c}{\nu}  |\xn(t_{j+1})|^2  \int_{t_j}^{t_{j+1}} | \nabla \xn(t_j)  -  \nabla \xn(s) |^2  \rd s 
	+  \frac{c}{\nu}  \int_{t_j}^{t_{j+1}} |\xn(s)|^2   |\nabla \xn(s)  - \nabla \xn(t_{j+1}))|^2 \rd s.
	\end{multline}
	
	Fix $K \in \mathbb{N}$. For each $l \in \mathbb{N}$, we define the following discrete stopping time
	\begin{align}\label{def:kappal}
	\kappa_l := \min \left\{ k \geq 1 \,:\, \sup_{t \in [0, t_{k+2}]} \frac{|\xn(t)|^2}{\nu^2} + \frac{3 \dt}{\nu} \sum_{j=1}^{k+2} |\nabla \xi^j|^2 \geq l  \right\} \wedge K.
 	\end{align}
	We also define a corresponding family of discrete stopping times $\kappa_l^i$, $i = 0,1,\ldots, L$, for some $L = L(l) \in \mathbb{N}$ to be suitably chosen, given by
	\begin{align*}
	\kappa_l^0 := 0; \quad
	\kappa_l^i := \min \left\{  k \geq \kappa_l^{i-1} + 1 \,:\,  \frac{\dt}{\nu}  \sum_{j = \kappa_l^{i-1} + 1}^{k+2} |\nabla \xi^j|^2 \geq \frac{l}{L}  \right\} \wedge \kappa_l ,
	\quad  i = 1, \ldots, L.
	\end{align*}
	It is not difficult to show that $\kappa_l^{i-1} < \kappa_l^i$ for all $i = 1,\ldots, L$, and $\kappa_l^L = \kappa_l$.
	
	Ignoring the second and fourth terms in the left-hand side of \eqref{eq:energy:zn:b} and summing over $j = \kappa_l^{i-1}, \ldots, k$, for $\kappa_l^{i-1} \leq k \leq \kappa_l^i$ and $i \in \{1, \ldots, L\}$, it follows that
	\begin{multline}\label{eq:energy:z:sum}
	|\z^{k+1}|^2 - |\z^{\kappa_l^{i-1}}|^2 
	\leq 
	c \nu \sum_{j = \kappa_l^{i-1}}^{\kappa_l^i} \int_{t_j}^{t_{j+1}}   |\nabla \xn(t_{j+1}) - \nabla \xn(s)|^2 \rd s
	+ \frac{c \dt}{\nu} \sup_{\kappa_l^{i-1} \leq j \leq \kappa_l^i } |\z^j|^2 \sum_{j = \kappa_l^{i-1} + 1}^{\kappa_l^i+1} | \nabla \xi^j|^2 
	\\
	+ \frac{c}{\nu} \sup_{t \in [0, t_{\kappa_l^i + 1}] } |\xn(t)|^2 \sum_{j = \kappa_l^{i-1}}^{\kappa_l^i} \int_{t_j}^{t_{j+1}} \left[ | \nabla \xn(t_j)  -  \nabla \xn(s) |^2  +       |\nabla \xn(s)  - \nabla \xn(t_{j+1}))|^2 \right] \rd s.
	\end{multline}
	
	Notice that, by the definition of $\kappa_l^i$, $i = 1, \ldots, L$, we have
	\begin{align*}
	\frac{\dt}{\nu}  \sum_{j = \kappa_l^{i-1} + 1}^{\kappa_l^i + 1} |\nabla \xi^j|^2 \leq \frac{l}{L} \quad \mbox{ for all } i = 1, \ldots, L.
	\end{align*}
	Choose $L \in \mathbb{N}$ as
	\begin{align}\label{choice:L}
	L = \min \left\{ j \in \mathbb{N} \,:\, cl/j \leq 1/2  \right\},
	\end{align}
	with $c > 0$ as in the second term in the right-hand side of \eqref{eq:energy:z:sum}. We then estimate this term as
	\begin{align}\label{est:term:sup:z}
	\frac{c \dt}{\nu} \sup_{\kappa_l^{i-1} \leq j \leq \kappa_l^i } |\z^j|^2 \sum_{j = \kappa_l^{i-1} + 1}^{\kappa_l^i+1} | \nabla \xi^j|^2 
	\leq 
	c \frac{l}{L}\sup_{\kappa_l^{i-1} \leq j \leq \kappa_l^i } |\z^j|^2 
	\leq
	\frac{1}{2} \sup_{\kappa_l^{i-1} \leq j \leq \kappa_l^i + 1} |\z^j|^2. 
	\end{align}
	Notice that \eqref{eq:energy:z:sum} is in fact valid for all $k = (\kappa_l^{i-1} - 1), \kappa_l^{i-1}, \ldots, \kappa_l^i$. Thus, taking in \eqref{eq:energy:z:sum} the supremum over $k =(\kappa_l^{i-1} - 1), \kappa_l^{i-1}, \ldots, \kappa_l^i$ and invoking \eqref{est:term:sup:z}, we obtain
	\begin{multline}\label{eq:energy:z:sup}
	\sup_{\kappa_l^{i-1} \leq j \leq \kappa_l^i + 1} |\z^j|^2 
	\leq
	2 |\z^{\kappa_l^{i-1}}|^2 
	+ c \nu \sum_{j = \kappa_l^{i-1}}^{\kappa_l^i } \int_{t_j}^{t_{j+1}}   |\nabla \xn(t_{j+1}) - \nabla \xn(s)|^2 \rd s
	\\
	+ \frac{c}{\nu} \sup_{t \in [0, t_{\kappa_l^i + 1}] } |\xn(t)|^2 \sum_{j = \kappa_l^{i-1}}^{\kappa_l^i} \int_{t_j}^{t_{j+1}} \left[ | \nabla \xn(t_j)  -  \nabla \xn(s) |^2  +       |\nabla \xn(s)  - \nabla \xn(t_{j+1}))|^2 \right] \rd s.
	\end{multline}
	
	From the definition of $\kappa_l$ and since $\kappa_l^i \leq \kappa_l^L = \kappa_l$ for all $i=1,\ldots,L$, we have
	\begin{align}\label{est:term:sup:xi}
	\sup_{t \in [0, t_{\kappa_l^i + 1}]} \frac{|\xn(t)|^2}{\nu^2}
	\leq 
	\sup_{t \in [0, t_{\kappa_l + 1}] } \frac{|\xn(t)|^2}{\nu^2}
	\leq
	\sup_{t \in  [0, t_{\kappa_l + 1}] } \frac{|\xn(t)|^2}{\nu^2}
	+ \frac{3 \dt}{\nu} \sum_{j=1}^{\kappa_l + 1} |\nabla \xi^j|^2 
	\leq l.
	\end{align}
	Using \eqref{est:term:sup:xi} to estimate the third term in the right-hand side of \eqref{eq:energy:z:sup}, it follows that
	\begin{align*}
	\sup_{\kappa_l^{i-1} \leq j \leq \kappa_l^i + 1} |\z^j|^2 
	\leq
	2 |\z^{\kappa_l^{i-1}}|^2 + \mA_i,
	\end{align*}
	where
	\begin{align}\label{def:Ai}
	\mA_i := 
	c \nu l  \sum_{j = \kappa_l^{i-1}}^{\kappa_l^i} \int_{t_j}^{t_{j+1}} \left[ | \nabla \xn(t_j)  -  \nabla \xn(s) |^2  +       |\nabla \xn(s)  - \nabla \xn(t_{j+1}))|^2 \right] \rd s.
	\end{align}
	
	Hence, for every $i = 1, \ldots, L$,
	\begin{align*}
	\sup_{j \leq \kappa_l^i} |\z^j|^2 
	&\leq \sup_{j \leq \kappa_l^{i - 1} } |\z^j|^2 + \sup_{ \kappa_l^{i -1 } \leq j \leq \kappa_l^i + 1} |\z^j|^2 
	\\
	&\leq  \sup_{j \leq \kappa_l^{i - 1} } |\z^j|^2 + 2  |\z^{\kappa_l^{i-1}}|^2 + \mA_i
	\\
	&\leq 3 \sup_{j \leq \kappa_l^{i - 1} } |\z^j|^2 + \mA_i.
	\end{align*}
	By induction, it follows that for every $l \in \mathbb{N}$ and $L = L(l) \in \mathbb{N}$ as in \eqref{choice:L},
	\begin{align}\label{ineq:sup:z}
	\sup_{j \leq \kappa_l} |\z^j|^2 = \sup_{j \leq \kappa_l^L } |\z^j|^2 
	\leq 3^L |\z^0|^2 + \sum_{i = 1}^L 3^{L - i} \mA_i 
	=  \sum_{i = 1}^L 3^{L - i} \mA_i,
	\end{align}
	since $\z^0 = 0$. 
	
	Now, for $K \in \mathbb{N}$ as in \eqref{def:kappal}, we define for every $l \in \mathbb{N}$
	\begin{align}\label{def:OmegalK} 
	\Omega_l^K := \left\{  \omega \in \Omega \,:\,  l - 1 \leq \sup_{t \in [0, t_{K+2}]} \frac{|\xn(t)|^2}{\nu^2}  +  \frac{3 \dt}{\nu} \sum_{j=1}^{K+2} |\nabla \xi^j|^2 < l  \right\}.
	\end{align}
	Since $\Omega = \bigcup_{l \in \mathbb{N}} \Omega_l^K$, we obtain for any $\smx >0$
	\begin{align}\label{ineq:sup:z:smx}
	\bE \sup_{j \leq K} |\z^j|^\smx
	= \bE \left( \sup_{j \leq K} |\z^j|^\smx \sum_{l=1}^\infty \ind_{\Omega_l^K} \right)
	&= \sum_{l = 1}^\infty \bE \left( \sup_{j \leq K} |\z^j|^\smx \ind_{\Omega_l^K}  \right)
	\notag \\
	&\leq \sum_{l=1}^\infty \left(  \bE \sup_{j \leq K} |\z^j|^2  \ind_{\Omega_l^K} \right)^{\smx/2} \bP(\Omega_l^K)^{\frac{2 - \smx}{2}},
	\end{align}
	where the last step follows from H\"older's inequality. Moreover, from the definition of $\kappa_l$ and $\Omega_l^K$, it follows that if $\omega \in \Omega_l^K$ then $\kappa_l(\omega) = K$. Thus, from \eqref{ineq:sup:z} and \eqref{def:Ai},
	\begin{multline*}
	\bE \sup_{j \leq K} |\z^j|^2  \ind_{\Omega_l^K} 
	=
	\bE \sup_{j \leq \kappa_l} |\z^j|^2 \ind_{\Omega_l^K}
	\leq \sum_{i = 1}^L 3^{L-i} \bE ( \ind_{\Omega_l^K}  \mA_i )
	\\
	\leq c \nu l 3^{L-1} \bE  \left( \ind_{\Omega_l^K}  \sum_{i=1}^L \sum_{j = \kappa_l^{i-1}}^{\kappa_l^i} \int_{t_j}^{t_{j+1}} \left[ | \nabla \xn(t_j)  -  \nabla \xn(s) |^2  +   |\nabla \xn(s)  - \nabla \xn(t_{j+1})|^2 \right] \rd s \right)
	\\
	\leq c \nu l 3^{L-1} \bE \sum_{j=0}^{K} \int_{t_j}^{t_{j+1}} \left[ | \nabla \xn(t_j)  -  \nabla \xn(s) |^2  +   |\nabla \xn(s)  - \nabla \xn(t_{j+1})|^2 \right] \rd s 
	\\
	= c \nu l 3^{L-1} \sum_{j=0}^{K} \int_{t_j}^{t_{j+1}} \left[ \bE | \nabla \xn(t_j)  -  \nabla \xn(s) |^2  +   \bE |\nabla \xn(s)  - \nabla \xn(t_{j+1})|^2 \right] \rd s .
	\end{multline*}
	
	Fix any $T \geq (K+1) \dt$. From \cref{thm:holder:reg} with $m = 2$, we have that for every $\tp \in (0,1/2)$ and $s, t \in [0,T]$ 
	\begin{align}\label{ineq:Holder:reg:2}
	\bE |\nabla \xn(t) - \nabla \xn(s)|^2 \leq C R_0 |t - s|^{2 \tp},
	\end{align}
	where $R_0 = (1 + |\xi_0|^8 + |\nabla \xi_0|^4 + |A\xi_0|^2 )$, and $C$ is a positive constant depending on  $\tp, T, \nu, |\sigma|, |\nabla \sigma|$. 
	Hence, 
	\begin{align*}
	\bE \sup_{j \leq K} |\z^j|^2  \ind_{\Omega_l^K} 
	&\leq
	c \nu l 3^{L-1} \sum_{j=0}^{K} C R_0 \dt^{2 \tp + 1} 
	\leq C R_0 \dt^{2 \tp} l 3^{L-1},
	\end{align*}
	where $C = C(\tp, T, \nu,  |\sigma|, |\nabla \sigma|)$.
	
	Moreover, from the definition of $L$ in \eqref{choice:L} it follows that $2cl \leq L \leq 2cl + 1$, so that $l 3^{L-1} \leq 3^{cl}$ and
	\begin{align}\label{ineq:sup:z:ind}
	\bE \sup_{j \leq K} |\z^j|^2  \ind_{\Omega_l^K}  \leq  C R_0 \dt^{2 \tp} 3^{cl}.
	\end{align}
	
	From the definition of the set $\Omega_l^K$ in \eqref{def:OmegalK} and invoking Markov's inequality, it follows that for any $\tsm > 0$
	\begin{align*}
	\bP(\Omega_l^K) 
	&\leq \mathbb{P} \left( \sup_{t \in [0, t_{K+2}]} \frac{|\xn(t)|^2}{\nu^2}  +  \frac{3 \dt}{\nu} \sum_{j=1}^{K+2} |\nabla \xi^j|^2 
	\geq l - 1  \right)
	\\
	&= \mathbb{P} \left( \exp \left( \tsm \sup_{t \in [0, t_{K+2}]} \frac{|\xn(t)|^2}{\nu^2}  +  \tsm \frac{3 \dt}{\nu} \sum_{j=1}^{K+2} |\nabla \xi^j|^2 \right) \geq e^{\tsm (l - 1)}  \right)
	\\
	&\leq e^{-\tsm (l-1)} \bE \left(  \exp \left( \tsm \sup_{t \in [0, t_{K+2}]} \frac{|\xn(t)|^2}{\nu^2}  +  \tsm \frac{3 \dt}{\nu} \sum_{j=1}^{K+2} |\nabla \xi^j|^2 \right)
	\right)
	\\
	&\leq e^{-\tsm (l-1)} \left( \bE \exp \left( 2 \tsm \sup_{t \in [0, t_{K+2}]} \frac{|\xn(t)|^2}{\nu^2}  \right) \right)^{1/2}
	\left( \bE \exp \left( c \tsm \frac{\dt}{\nu} \sum_{j=1}^{K+2} |\nabla \xi^j|^2  \right)  \right)^{1/2}.
	\end{align*}
	Now we assume that $0 < \tsm \leq \tilde{c} \nu^2 |\sigma|^{-2} \min\{\nu, \dt_0^{-1}\}$ for some absolute constant $\tilde{c} > 0$ that is small enough so we can invoke the bounds \eqref{expbounds:1} from  \cref{prop:exp:bounds:sup} and \eqref{E:sup:exp:xi:T} from \cref{prop:exp:mom:cont} to obtain that
	\begin{align}\label{ineq:P:OmegalK}
	\bP(\Omega_l^K) 
	&\leq 
	c e^{-\tsm (l-1)}  \exp \left( \tC \frac{\tsm}{\nu^2} |\xi_0|^2 \right) \exp \left(c  \frac{\tsm}{\nu^2} |\sigma|^2 (K + 2) \dt \right)  
	\notag \\
	&\leq C e^{-\tsm (l-1)}  \exp \left( \tC \frac{\tsm}{\nu^2} |\xi_0|^2 \right),
	\end{align}
where $C = C(\nu, \dt_0, T, |\sigma|)$ and $\tC = c(1 + \nu \dt_0)$.
	
	Therefore, it follows from from \eqref{ineq:sup:z:smx}, \eqref{ineq:sup:z:ind}, and \eqref{ineq:P:OmegalK} that
	\begin{align}\label{ineq:E:sup:z}
	\bE \sup_{j \leq K} |\z^j|^\smx 
	&\leq  
	\sum_{l=1}^\infty \left( C R_0 \dt^{2 \tp} 3^{cl} \right)^{\smx/2}  \left( C e^{-\tsm (l-1)} \exp \left(  \tC \frac{\tsm}{\nu^2} |\xi_0|^2 \right) \right)^{\frac{2 - \smx}{2}}
	\notag\\
	&\leq 
	C R_0^{\smx/2} \exp \left( \tC \frac{\tsm}{\nu^2} |\xi_0|^2 \right) \dt^{\tp \smx} \left(\sum_{l=1}^\infty 3^{c \smx l} e^{-\frac{\tsm (2 - \smx)}{2} l } \right) e^{\tsm \frac{(2 - \smx)}{2}}.
	\end{align}
	For the terms depending on the initial datum $\xi_0$, we have by the definition of $R_0$ in \eqref{ineq:Holder:reg:2} that
	\begin{align}\label{term:xi0}
	R_0^{\smx/2}  \exp \left( \tC \frac{\tsm}{\nu^2} |\xi_0|^2 \right)
	&=
	(1 + |\xi_0|^8 + |\nabla \xi_0|^4 + |A\xi_0|^2 )^{\smx/2}  \exp \left( \tC \frac{\tsm}{\nu^2} |\xi_0|^2 \right)
	\notag\\
	&\leq
	C (1 + |\nabla \xi_0|^4 + |A\xi_0|^2)^{\smx/2}  \exp \left( \tC \frac{\tsm}{\nu^2} |\xi_0|^2 \right),
	\end{align}
	for $C = C(\smx,\nu, \dt_0, |\sigma|, \tsm)$.
	
	Moreover, for the term involving the sum in \eqref{ineq:E:sup:z}, notice that
	\begin{align}\label{term:sum}
	\sum_{l=1}^\infty 3^{c \smx l} e^{-\tsm \frac{ (2 - \smx)}{2} l } = 	\sum_{l=1}^\infty e^{\ln (3) c \smx l} e^{-\tsm \frac{ (2 - \smx)}{2} l } = 	\sum_{l=1}^\infty e^{- \gamma l},
	\end{align}
	where $\gamma = [\tsm (2 - \smx)/2]  -  c \smx$. We choose $\smx < 2\tsm / (c + \tsm)$, so that $\gamma > 0$ and, consequently, the last sum in \eqref{term:sum} is finite. We thus conclude from \eqref{ineq:E:sup:z} and \eqref{term:xi0} that
	\begin{align*}
	\bE \sup_{j \leq K} |\z^j|^\smx 
	\leq  
	C \dt^{\tp\smx} (1 + |\nabla \xi_0|^4 + |A\xi_0|^2)^{\smx/2}  \exp \left( \tC \frac{\tsm}{\nu^2} |\xi_0|^2 \right),
	\end{align*}
	for some positive constant $C = C( \tp, \smx, \nu, \dt_0, T,|\sigma|, |\nabla \sigma|, \tsm)$. 
	This shows \eqref{ineq:time:disc:error} by denoting $\sm \coloneqq \tsm \nu^{-2}$, and finishes the proof.
\end{proof}

\subsection{Uniform in time weak convergence of the numerical scheme}\label{subsec:wk:conv:NSE}

This section focuses on the application of \cref{thm:gen:wk:conv} and \cref{thm:gen:wk:conv:2} to the space-time discretization of the 2D stochastic Navier-Stokes equations introduced in \eqref{disc2DSNSEv}. Following the notation from \cref{thm:gen:wk:conv}, similarly as in \cref{subsec:Wass:contr:NSE} we take $(\gsp, \gn{\cdot}) = (\dL^2, |\cdot|)$ and again consider $\gfd$ to be the class of distance functions $\rhoes$, $\varepsilon > 0$, $s \in (0,1]$, defined in \eqref{def:rhoes}. We let $\gspmr$ be the set of pairs $\{(N,\dt)\,:\, N \in \NN, \,\, \dt > 0\} \cup \{(\infty,0)\}$. Then, for every $\pmr = (N,\dt)$ with $N \in \NN$ and $\dt > 0$, we let $\{\gP_t^\pmr\}_{t \geq 0}$ be the family of Markov kernels $\{\Pdisct\}_{t \geq 0}$ associated to the numerical scheme \eqref{disc2DSNSEv2}, defined in \eqref{def:Pdisct} and \eqref{def:Pdisc:0}. For $\pmr_0 \coloneqq (\infty,0)$, we let $\{\gP_t^{\pmr_0}\}_{t\geq 0}$ be the Markov semigroup $\{\Pcont\}_{t \geq 0}$ associated to the 2D SNSE \eqref{2DSNSEv}, and defined in \eqref{def:Pcont}.

Regarding assumptions \ref{H:1}-\ref{H:3} of \cref{thm:gen:wk:conv:2}, only \ref{H:4} requires extra work to be verified. This is done in the following proposition, whose proof follows crucially from the strong error estimates obtained in \cref{prop:spatial:error} and \cref{prop:time:disc:error} above, combined with the exponential Lyapunov inequalities from \cref{prop:Lyap:exp:disc} and \cref{prop:exp:mom:cont:2}.

\begin{Prop}\label{prop:fin:time:err:wass}
Fix any $N \in \NN$, $\dt, \dt_0 > 0$ with $\dt \leq \dt_0$, and $\sigma \in \bdH^1$. Let $\{\Pcont\}_{t\geq 0}$ and $\{\Pdisct\}_{t \geq 0}$ be the corresponding family of Markov kernels associated to systems \eqref{2DSNSEv} and \eqref{disc2DSNSEv2}, respectively, as defined in \eqref{def:Pcont} and \eqref{def:Pdisct}. Then, there exists a positive absolute constant $c_1$ such that if
\begin{align}\label{cond:sm:smx:2}
	\sm' = \frac{c_1}{|\sigma|^2} \min\left\{ \nu , \frac{1}{\dt_0}\right\}
\end{align}
then for every $\sm \in (0, \sm']$, $s \in (0,1]$, $\varepsilon > 0$, and $\tp \in (0,1/2)$, it holds that
\begin{align}\label{fin:tim:err:wass}
&\Wesa(\Pdisct(\xi_0, \cdot), \Pcont (\xi_0, \cdot)) \\
&\qquad \qquad\qquad \qquad\leq
\frac{C e^{C' t}}{\varepsilon^{1/2}} \left[  \max \{ \dt^s, \dt^{\smx/2}\}^{\tp/2}  + N^{-s/4} \right] \exp(\tC \sm' |\xi_0|^2) (1 + |\nabla \xi_0| + |A\xi_0|^{1/2})^s,
\notag
\end{align}
for every $t \geq 0$ and $\xi_0 \in \dH^2$. Here, $0< \smx <  \frac{2\nu^2\sm'}{c_2 + \nu^2 \sm'}$ for some absolute constant $c_2$, and $\tC = c(1+ \nu \dt_0)$, $C' = C'(\nu, |\sigma|)$, $C = C(\tp, \smx, \nu, \dt_0, |\sigma|, |\nabla \sigma|, \sm)$.
\end{Prop}
\begin{proof}
Fix any $\sm \in (0, \sm']$, $s \in (0,1]$, $\varepsilon > 0$, $\tp \in (0,1/2)$, $\xi_0 \in \dH^2$, and $t \geq 0$. 
Let $n \in \ZZ^+$ such that $t \in [n\dt, (n+1)\dt)$. It follows immediately from the definitions of $\{\Pcont\}_{t \geq 0}$, $\{\Pdisct\}_{t \geq 0}$ and $\Wesa$ that
\begin{align*}
	\Wesa(\Pdisct(\xi_0, \cdot), \Pcont(\xi_0, \cdot)) 
	&\leq 
	\Wesap(\Pdisct(\xi_0, \cdot), \Pcont(\xi_0, \cdot)) 
	= \Wesap(\Pdisc(\xi_0, \cdot), \Pcont(\xi_0, \cdot)) \\
	&\leq \frac{1}{\varepsilon^{1/2}} \bE \left[ |\xi^n(\xi_0) - \xi(t;\xi_0)|^{s/2} \exp \left( \sm' |\xi^n(\xi_0)|^2 + \sm' |\xi(t; \xi_0)|^2 \right) \right].
\end{align*}
By H\"older's inequality,
\begin{multline}\label{Wesa:Pdisc:Pcont:xi0}
	\Wesap(\Pdisct(\xi_0, \cdot), \Pcont(\xi_0, \cdot)) \\
	\qquad \leq \frac{1}{\varepsilon^{1/2}} \left( \bE |\xi^n(\xi_0) - \xi(t;\xi_0)|^s \right)^{1/2}
	\left( \bE \exp \left( 4 \sm' |\xi^n(\xi_0)|^2 \right) \right)^{1/4}
	\left( \bE \exp \left( 4 \sm' |\xi(t; \xi_0)|^2 \right) \right)^{1/4}.
\end{multline}

We assume that the constants $c_1$ in \eqref{cond:sm:smx:2} and $c_2$ in the definition of $p$ are sufficiently small so that the results of \cref{prop:Lyap:exp:disc}, \cref{prop:exp:mom:cont:2}, \cref{prop:spatial:error}, and \cref{prop:time:disc:error} can be applied in the estimates to follow. In particular, invoking the exponential Lyapunov inequalities \eqref{Lyap:disc:0} from \cref{prop:Lyap:exp:disc} and the analogous version of \eqref{bE:Lyap:gal} for $\xi(t)$, $t \geq 0$, from \cref{prop:exp:mom:cont:2}, respectively, we estimate the last two terms between parentheses in \eqref{Wesa:Pdisc:Pcont:xi0} as
\begin{align}\label{E:exp:xndisc:xi}
\left( \bE \exp \left( 4 \sm' |\xi^n(\xi_0)|^2 \right) \right)^{1/4}
\left( \bE \exp \left( 4 \sm' |\xi(t; \xi_0)|^2 \right) \right)^{1/4}
&\leq
C \exp \left( \frac{ 2 \sm' |\xi_0|^2}{(1 + \nu \lambda_1 \dt)^n} \right) \exp \left( \sm' e^{-\nu t} |\xi_0|^2 \right) \notag \\
&\leq C \exp(3 \sm' |\xi_0|^2),
\end{align}
where $C = C (\nu, \dt_0, |\sigma|)$.

Regarding the first term between parentheses in \eqref{Wesa:Pdisc:Pcont:xi0}, we first estimate as
\begin{align}\label{diff:xndisc:xi:s}
	\bE |\xi^n(\xi_0) - \xi(t;\xi_0)|^s
	\leq \bE  |\xi^n(\xi_0) - \xn(n\dt;\xi_0)|^s + \bE  |\xn(n\dt;\xi_0) - \xi(n\dt;\xi_0)|^s + \bE  |\xi(n\dt;\xi_0) - \xi(t;\xi_0)|^s,
\end{align}
where we recall that $s \in (0,1]$. We proceed to estimate the terms in the right-hand side of \eqref{diff:xndisc:xi:s} by invoking \cref{thm:holder:reg}, \cref{prop:spatial:error}, and \cref{prop:time:disc:error} with $T = t$. Here we will write the $t$-dependence of the constant $C$ from \eqref{holder:reg:L2}, \eqref{ineq:lem:spatial:error} and \eqref{ineq:time:disc:error} explicitly as $e^{C' t}$, for some constant $C' = C'(\nu, |\sigma|)$, as can easily be seen from the corresponding proofs. 

In particular, invoking inequality \eqref{ineq:lem:spatial:error} from \cref{prop:spatial:error} and H\"older's inequality, we estimate the second term in the right-hand side of \eqref{diff:xndisc:xi:s} as
\begin{align}\label{xn:xi:s}
\bE |\xn(n\dt;\xi_0) - \xi(n\dt;\xi_0)|^s
&\leq
\left( \bE |\xn(n\dt;\xi_0) - \xi(n\dt;\xi_0)|^2 \right)^{s/2} \notag \\
&\leq 
\left( \frac{C e^{C' t}}{N} 
\left[  \exp(c \sm' |\xi_0|^2 ) +  |\nabla \xi_0|^2 \right] \right)^{s/2} \notag \\
&\leq
\frac{C e^{C' t}}{N^{s/2}} \exp(c \sm' |\xi_0|^2 ) (1 + |\nabla \xi_0|^2)^{s/2},
\end{align}
where $C = C(\nu,|\sigma|, |\nabla \sigma|, \sm')$. 

For the third term in the right-hand side of \eqref{diff:xndisc:xi:s}, we invoke the analogous version of inequality \eqref{holder:reg:L2} from \cref{thm:holder:reg} with $\xi(t)$, and obtain
\begin{align}\label{xi:ndt:t}
	\bE  |\xi(n\dt;\xi_0) - \xi(t;\xi_0)|^s 
	\leq \left( \bE  |\xi(n\dt;\xi_0) - \xi(t;\xi_0)|\right)^s
	\leq C e^{C' t} \dt^{s\tp} (1 + |\xi_0|^4 + |\nabla \xi_0|^2)^s,
\end{align}
where $C = C(\tp,\nu, |\sigma|, |\nabla \sigma|)$. 

Finally, to estimate the first term in the right-hand side of \eqref{diff:xndisc:xi:s}, let us first assume that $s < \smx$, with $\smx$ as in \eqref{cond:sm:smx}. In this case, it follows by H\"older's inequality and \eqref{ineq:time:disc:error} from \cref{prop:time:disc:error} that
\begin{align}\label{xndisc:xn:s:1}
\bE |\xi^n(\xi_0) - \xn(n\dt;\xi_0)|^s
&\leq
\left( \bE |\xi^n(\xi_0) - \xn(n\dt;\xi_0)|^\smx \right)^{s/\smx} \notag \\
&\leq
\left( C e^{C' t} \dt^{\tp \smx} \exp(\tC \sm' |\xi_0|^2) (1 + |\nabla \xi_0|^4 + |A\xi_0|^2 )^{\smx/2} \right)^{s/\smx} \notag \\
&\leq
C e^{C' t} \dt^{\tp s} \exp (\tC \sm' |\xi_0|^2) (1 + |\nabla \xi_0|^4 + |A\xi_0|^2 )^{s/2},
\end{align}
where $\tC = c(1 + \nu \dt_0)$ and $C = C(\tp, \smx, \nu, \dt_0, |\sigma|, |\nabla \sigma|)$. On the other hand, if $s \geq \smx$ we proceed as follows
\begin{align*}
\bE |\xi^n(\xi_0) - \xn(n\dt;\xi_0)|^s
&=
\bE \left[ |\xi^n(\xi_0) - \xn(n\dt;\xi_0)|^{\smx/2} |\xi^n(\xi_0) - \xn(n\dt;\xi_0)|^{s - \frac{\smx}{2}} \right] \\
&\leq
\bE \left[ |\xi^n(\xi_0) - \xn(n\dt;\xi_0)|^{\smx/2} \left( |\xi^n(\xi_0)|^{s - \frac{\smx}{2}} +  |\xn(n\dt;\xi_0)|^{s - \frac{\smx}{2}} \right) \right] \\
&\leq 
C \,	\bE \left[ |\xi^n(\xi_0) - \xn(n\dt;\xi_0)|^{\smx/2} \left( \exp(\sm' |\xi^n(\xi_0)|^2 ) + \exp(\sm' |\xn(n\dt; \xi_0)|^2 )  \right) \right], 
\end{align*}
where $C = C(\sm')$. Thus, by H\"older's inequality and inequalities \eqref{Lyap:disc:0}, \eqref{bE:Lyap:gal}, and \eqref{ineq:time:disc:error} from \cref{prop:Lyap:exp:disc}, \cref{prop:exp:mom:cont:2}, and \cref{prop:time:disc:error}, respectively, we obtain that
\begin{align}\label{xndisc:xn:s:2}
&\bE |\xi^n(\xi_0) - \xn(n\dt;\xi_0)|^s \notag \\
&\qquad \leq
C \left( \bE |\xi^n(\xi_0) - \xn(n\dt;\xi_0)|^\smx \right)^{1/2} \left[ \bE \exp (2 \sm' |\xi^n(\xi_0)|^2)  + \bE \exp (2 \sm' |\xn(n\dt; \xi_0)|^2 ) \right]^{1/2} \notag \\
&\qquad \leq
C \left( e^{C' t} \dt^{\tp \smx} \exp(\tC \sm' |\xi_0|^2) (1 + |\nabla \xi_0|^4 + |A\xi_0|^2 )^{\smx/2} \right)^{1/2} \notag \\
&\qquad \qquad
\cdot \left[ \exp \left( \frac{4\sm' |\xi_0|^2}{(1 + \nu \lambda_1 \dt)^n} \right) + \exp \left( 2 \sm' e^{-\nu n \dt} |\xi_0|^2 \right)  \right]^{1/2} \notag \\
&\qquad \leq
C e^{C' t} \dt^{\tp \smx /2} \exp(\tC \sm' |\xi_0|^2) (1 + |\nabla \xi_0|^4 + |A\xi_0|^2 )^{s/4},
\end{align}
where $C = C(\tp, \smx, \nu, \dt_0, |\sigma|, |\nabla \sigma|, \sm')$. Combining \eqref{xndisc:xn:s:1} and \eqref{xndisc:xn:s:2}, it thus follows that for every $s \in (0,1]$
\begin{align}\label{xndisc:xn:s:3}
\bE |\xi^n(\xi_0) - \xn(n\dt;\xi_0)|^s
\leq	
C e^{C' t} \max \{ \dt^s, \dt^{\smx/2}\}^{\tp} \exp(\tC \sm' |\xi_0|^2) (1 + |\nabla \xi_0|^4 + |A\xi_0|^2 )^{s/2}.
\end{align}

Plugging \eqref{xn:xi:s}, \eqref{xi:ndt:t} and \eqref{xndisc:xn:s:3} into \eqref{diff:xndisc:xi:s}, we thus have
\begin{align}\label{diff:xndisc:xi:s:2}
\bE |\xi^n(\xi_0) - \xi(n\dt;\xi_0)|^s
\leq
C e^{C' t} \left[  \max \{ \dt^s, \dt^{\smx/2}\}^{\tp}  + N^{-s/2} \right] \exp(\tC \sm' |\xi_0|^2) (1 + |\nabla \xi_0|^4 + |A\xi_0|^2 )^{s/2}.
\end{align}

Now, plugging \eqref{E:exp:xndisc:xi} and \eqref{diff:xndisc:xi:s:2} into \eqref{Wesa:Pdisc:Pcont:xi0}, we deduce that
\begin{align*}
&\Wesap(\Pdisct(\xi_0, \cdot), \Pcont(\xi_0, \cdot)) \\
&\qquad \qquad\qquad \qquad\leq
\frac{C e^{C' t}}{\varepsilon^{1/2}} \left[  \max \{ \dt^s, \dt^{\smx/2}\}^{\tp/2}  + N^{-s/4} \right] \exp(\tC \sm' |\xi_0|^2) (1 + |\nabla \xi_0| + |A\xi_0|^{1/2})^s,
\end{align*}
with $C = C(\tp, \smx, \nu, \dt_0, |\sigma|, |\nabla \sigma|, \sm')$, and we recall that $\tC = c(1 + \nu \dt_0)$, $C' = C'(\nu, |\sigma|)$. This shows \eqref{fin:tim:err:wass} and concludes the proof.
\end{proof}

Before proceeding with the application of \cref{thm:gen:wk:conv} and \cref{thm:gen:wk:conv:2} within this setting, we present the following lemma showing finiteness of certain moments for the invariant measure of $\Pcont$, $t \geq 0$. This will ensure that the terms in \eqref{W:inv:meas} and \eqref{gen:w:conv} concerning $\mu_{\pmr_0}$ are finite.

\begin{Lem}\label{lem:int:inv:bounds}
	Fix any $\sigma \in \bdL^2$, and let $\cinv$ be an invariant measure of the corresponding Markov semigroup $\Pcont$, $t \geq 0$, defined in \eqref{def:Pcont}. Then, the following statements hold:
	\begin{itemize}
		\item[(i)] For every $\sm > 0$ satisfying condition \eqref{condgamma0},
		\begin{align}\label{int:exp:cinv}
			\int_{\dL^2} \exp(\sm |\xi_0|^2 ) \cinv(d\xi_0) \leq C,
		\end{align}
		for some constant $C = C(\nu, |\sigma|)$.
		\item[(ii)]	Suppose additionally that $\sigma \in \bdH^k$, for some fixed $k \in \NN$. Then, for every $m \in \NN$,
		\begin{align}\label{int:grad:xi:cinv}
			\int_{\dL^2} \| \xi_0\|_{\dH^k}^{2 m} \cinv(d \xi_0) \leq C,
		\end{align}
	for some constant $C = C(\nu, m, k , \|\sigma\|_{\dH^k})$. Consequently, $\cinv$ is supported in $\dH^k$, i.e. $$\cinv(\dL^2 \,\backslash\, \dH^k)= 0.$$
	\end{itemize}
\end{Lem}
\begin{proof}
	The proof of \eqref{int:exp:cinv} follows as a consequence of the analogous version of the exponential Lyapunov inequality \eqref{bE:Lyap:gal} for $\xi(t)$, $t \geq 0$, similarly as in the proof of \cite[Theorem 2.5.3]{KuksinShirikyan12}. The subsequent bound \eqref{int:grad:xi:cinv} then follows by combining \eqref{int:exp:cinv} with inequality \eqref{ineq:sup:xn:Hk} for $\xi(t)$, $t \geq 0$, by estimating $1 + |\xi_0|^{2mp} \leq C(m,p,\sm)\exp(\sm |\xi_0|^2)$.
	We omit further details. 
\end{proof}

We are now ready to apply \cref{thm:gen:wk:conv} and derive a bias estimate between invariant measures of $\{\Pdisct\}_{t \geq 0}$ and $\{\Pcont\}_{t \geq 0}$.

\begin{Thm}[long time bias estimate]\label{thm:bias:ineq:NSE}
Fix any $N \in \NN$, $\dt, \dt_0 > 0$ with $\dt \leq \dt_0$. Suppose there exists $K \in \NN$ and $\sigma \in \bdH^2$ such that
\begin{align}\label{cond:K:sigma:1}
\HK \subset \range{\sigma}, \quad \mbox{ and } \quad
\lambda_{K+1} \geq 
\frac{c}{\nu} \max \left\{ 1, \frac{1}{\dt_0}, \frac{\dt_0^2 |\sigma|^4}{\nu^3 }, \frac{| \sigma|^4}{\nu^5 }  \right\}
\end{align}
for some absolute constant $c > 0$.	Let $\{\Pcont\}_{t\geq 0}$ and $\{\Pdisct\}_{t \geq 0}$ be the corresponding family of Markov kernels associated to systems \eqref{2DSNSEv} and \eqref{disc2DSNSEv2}, respectively, as defined in \eqref{def:Pcont} and \eqref{def:Pdisct}. Let $\cinv$ and $\dinv$ be invariant measures for $\{\Pcont\}_{t\geq 0}$ and $\{\Pdisct\}_{t \geq 0}$, respectively. 	

Then, there exists $\sm_* > 0$ such that for each fixed $\sm \in (0, \sm_*]$ there exists $\varepsilon > 0$ and $s \in (0,1]$ for which the following inequality holds for every $\tp \in (0,1/2)$
\begin{align}\label{ineq:bias:NSE:0}
\Wesa(\dinv,\cinv) \leq C \pfunc(N,\dt),
\end{align}
where 
\begin{align}\label{def:pfunc:NSE}
\pfunc(N,\dt) = \max \{ \dt^s, \dt^{\smx/2}\}^{\tp/2} + N^{-s/4},
\end{align}
with
\begin{align}\label{def:p:smp:NSE}
0 < \smx <  \frac{2 \nu^2\sm'}{c_2 + \nu^2 \sm'}, \quad \sm' = \frac{c_1}{|\sigma|^2} \min\left\{ \nu , \frac{1}{\dt_0}\right\},
\end{align}
for some absolute constants $c_1, c_2$, and where $C = C(\varepsilon,s,\sm, \nu,\dt_0, |\sigma|, |\nabla \sigma|, \|\sigma\|_{\dH^2}, \tp)$.
\end{Thm}
\begin{proof}
	We verify that all assumptions from \cref{thm:gen:wk:conv} are satisfied with the choices of $(X, \|\cdot\|)$, $\gfd$, $\gspmr$, $\{\gP_t^\pmr\}_{t \geq 0}$ and $\{\gP_t^{\pmr_0}\}_{t \geq 0}$ taken in the introduction to this section. 
	
	Assumption \ref{H:1} follows as a consequence of \cref{prop:gen:tri:ineq:rhoesa} applied to the metric $\rhoes$. Indeed, items \ref{rho:i} and \ref{rho:ii} of \cref{prop:gen:tri:ineq:rhoesa} hold with $M = K = 1$. Moreover, taking $c = 1$ in item \ref{rho:iii}, we notice that if $\rhoes(\xi_1,\xi_2) < 1$ then $|\xi_1 - \xi_2|^s < \varepsilon$. Thus, by triangle inequality, $|\xi_1|^2 < 2 |\xi_2|^2 + 2 \varepsilon^{2/s}$, so that item \ref{rho:iii} holds with $\gamma = 2$ and $C = 2 \varepsilon^{2/s}$. It thus follows from \eqref{gen:tri:ineq:rhosm} and the definition of $\tilde{K}$ inside the proof that
	\begin{align*}
	\rhoesa(\xi_1, \xi_2) \leq \exp(2 \sm \varepsilon^{2/s}) \left[ \rho_{\varepsilon, s, 2\sm} (\xi_1, \xi_3) + \rho_{\varepsilon, s, 2 \sm}(\xi_3, \xi_2) \right] \quad \mbox{ for all } \xi_1, \xi_2, \xi_3 \in \dL^2,
	\end{align*}
	as desired.
	
	Regarding assumption \ref{H:2}, the existence of an invariant measure for $\{\Pdisct\}_{t \geq 0}$ is shown in \cref{cor:exist:inv:meas:disc}. Whereas the existence of an invariant measure for $\{\Pcont\}_{t \geq 0}$, as mentioned in \cref{subsubsec:2DSNSE}, is a well-known result that is valid for even more general noise structures than we consider here, see e.g. \cite{Flandoli1994}.
	
	Assumption \ref{H:3} follows as a direct consequence of \cref{thm:sp:gap:disc:SNSE}. Finally, assumption \ref{H:4} follows from \cref{prop:fin:time:err:wass}, with $\pfunc(N,\dt)$, $\tgf(t)$ and $\gf(\xi_0)$ given for any fixed $\varepsilon > 0$ and $s \in (0,1]$ as
	\begin{gather}
	\pfunc(N,\dt) =  \max \{ \dt^s, \dt^{\smx/2}\}^{\tp/2}  + N^{-s/4}, \,\,
	\mbox{ for all } N \in \NN, \,\, 0 < \dt \leq \dt_0, \label{def:pfunc:nse}\\
	\tgf(t) = \frac{C e^{C' t}}{\varepsilon^{1/2}}, \,\, \mbox{ for all } t \geq 0, \label{def:tgf:nse}\\
	\gf(\xi_0) = \begin{cases} 
	\exp(\tC \sm' |\xi_0|^2) (1 + |\nabla \xi_0| + |A\xi_0|^{1/2})^s & \mbox{ for } \xi_0 \in \dH^2, \\ \infty &\mbox{ for } \xi_0 \in \dL^2 \backslash \dH^2,
	\end{cases} \label{def:gf:nse}
	\end{gather}
	with $p, \tp, \sm', C, C', \tC$ as in \eqref{fin:tim:err:wass}.

	Therefore, it follows from \cref{thm:gen:wk:conv} that
	\begin{align}\label{ineq:bias:NSE}
		\Wesa(\dinv,\cinv) \leq C \pfunc(N,\dt) \int_{\dL^2} \gf(\xi_0) \cinv(d\xi_0),
	\end{align}
	where $C = C(\varepsilon,s,\sm, \nu, \dt_0, |\sigma|, |\nabla\sigma|, \tp)$. Moreover, from the definitions of $\tC$ and $\sm'$ given in \cref{prop:fin:time:err:wass} it is not difficult to show that $\tC \sm' \leq c \nu |\sigma|^{-2}$. With this, we may apply H\"older's inequality and \cref{lem:int:inv:bounds} to obtain that $\int_{\dL^2} \gf(\xi_0) \cinv(d\xi_0) \leq C$, for $C = C(\nu,\|\sigma\|_{\dH^2})$. Plugging this into \eqref{ineq:bias:NSE} we conclude \eqref{ineq:bias:NSE:0}.
\end{proof}

We conclude this section by applying \cref{thm:gen:wk:conv:2} to show convergence in Wasserstein distance for $\{\Pdisct\}_{t \geq 0}$ towards $\{\Pcont\}_{t \geq 0}$, and consequently weak convergence for $\xfd^n$, $n \in \ZZ^+$, towards $\xi(t)$, $t \geq 0$, as a result of \cref{cor:gen:wk:conv:obs}.

\begin{Thm}[Uniform in time weak convergence]\label{thm:wk:conv:SNSE}
Fix any $N \in \NN$, $\dt, \dt_0 > 0$ with $\dt \leq \dt_0$. Suppose there exists $K \in \NN$ and $\sigma \in \bdH^2$ satisfying \eqref{cond:K:sigma:1}. Let $\{\Pcont\}_{t\geq 0}$ and $\{\Pdisct\}_{t \geq 0}$ be the corresponding family of Markov kernels associated to systems \eqref{2DSNSEv} and \eqref{disc2DSNSEv2}, respectively, as defined in \eqref{def:Pcont} and \eqref{def:Pdisct}. 

Then, there exists $\hat{\sm} > 0$ such that for each fixed $\sm \in (0, \hat{\sm}]$ there exists $\varepsilon > 0$ and $s \in (0,1]$ for which the following inequality holds
\begin{align}\label{sup:t:Wesa:Pdisc:Pcont:0}
\sup_{t \geq 0} \Wesa (\mu \Pdisct, \mu \Pcont)
\leq 
C \max \{ g(N,\dt)^{q C'}, g(N,\dt), g(N,\dt)^{1-q}\},
\end{align}
with $g$ as in \eqref{def:pfunc:NSE}, for every $\tp \in (0,1/2)$, $q \in (0,1)$, and $\mu \in \Pr(\dL^2)$ satisfying
\begin{align}\label{cond:mu}
\int_{\dL^2} \left[ \exp(c \sm |\xi_0|^2) +  \exp(\tC\sm' |\xi_0|^2) (1 + |\nabla \xi_0| + |A\xi_0|^{1/2})^s \right] \mu(d \xi_0) < \infty
\end{align}
for some absolute constant $c$, with $\tC$ and $\sm'$ being the same as in \eqref{fin:tim:err:wass} and \eqref{def:p:smp:NSE}, respectively. Moreover, $C = C(\varepsilon, s, \sm, \nu, \dt_0, |\sigma|,$ $|\nabla \sigma|, \|\sigma\|_{\dH^2}, \tp, \mu)$ and $C' = C'(\varepsilon, s, \sm, \nu, |\sigma|)$.

Consequently, under these same assumptions it follows that for all $\xi_0 \in \dH^2$ and $\rhoesa$-Lipschitz function $\varphi: \dL^2 \to \RR$ with Lipschitz constant $L_\varphi$,
\begin{align}\label{sup:E:obs:diff}
\sup_{n \in \NN} \left| \bE \left[ \varphi(\xi(n \dt;\xi_0)) - \varphi(\xfd^n(\xi_0)) \right] \right|
\leq
L_{\varphi} C \max \{ g(N,\dt)^{q C'}, g(N,\dt), g(N,\dt)^{1-q}\},
\end{align}
with $C$ and $C'$ as in \eqref{sup:t:Wesa:Pdisc:Pcont:0}. Here, $\xi(t;\xi_0)$, $t \geq 0$, and $\xfd^n(\xi_0)$, $n\in \NN$, denote the unique solutions of \eqref{2DSNSEv} and \eqref{disc2DSNSEv2}, respectively, such that $\xi(0;\xi_0) = \xi_0$ and $\xfd^0(\xi_0) = \xi_0$.
\end{Thm}
\begin{Rmk}
 Here we refer to \cite[Proposition 46]{GHM2022}, which gives a sufficient condition under which any suitably regular function $\varphi$ is $\tilde{\rho}$-Lipschitz for some distance function $\tilde\rho$ sharing similar structure to $\rhoesa$. 
\end{Rmk}
\begin{proof}
Let us verify that the assumptions of \cref{thm:gen:wk:conv:2} hold. The verification of assumptions \ref{H:1}-\ref{H:3} follows as in the proof of \cref{thm:bias:ineq:NSE}. Moreover, from the definitions of $\pfunc(N,\dt)$ and $\tgf(t)$ in \eqref{def:pfunc:nse} and \eqref{def:tgf:nse}, respectively, it is clear that $\tgf$ is continuous and strictly increasing in $t$, and $\pfunc$ is bounded with respect to $(N,\dt) \in \NN \times (0,\dt_0]$. Further, as argued at the end of the proof of \cref{thm:bias:ineq:NSE}, denoting by $\cinv$ an invariant measure of $\{\Pcont\}_{t \geq 0}$ it follows from the definition of $\gf$ in \eqref{def:gf:nse} and \cref{lem:int:inv:bounds} that $\int_{\dL^2} \gf(\xi_0) \cinv(d\xi_0) \leq C < \infty$, for $C = C(\nu,\|\sigma\|_{\dH^2})$.

Thus, from \cref{thm:gen:wk:conv:2} and \cref{rmk:exp:growth} we deduce that 
\begin{align}\label{sup:t:Wesa:Pdisc:Pcont}
\sup_{t \geq 0} \Wesa(\mu \Pdisct, \mu \Pcont) 
\leq
\tilde{g}(N, \dt) \left[ \Wass_{\varepsilon,s,2\sm}(\mu, \mu_*) + \int_{\dL^2} f(\xi_0) \mu(d\xi_0) + \int_{\dL^2} f(\xi_0) \mu_*(d \xi_0)  \right],
\end{align}
for  every $\mu \in \Pr(\dL^2)$ satisfying \eqref{cond:mu}, where
\begin{align*}
\tilde{g}(N, \dt) = C \max \{ g(N,\dt)^{q C'}, g(N,\dt), g(N,\dt)^{1-q}\},
\end{align*}
for any fixed $q < 1$. Here, as seen from the proof of \cref{thm:gen:wk:conv:2} and the invoked results, it follows that $C$ and $C'$ are positive constants with $C = C(\varepsilon, s, \sm, \nu, \dt_0, |\sigma|, |\nabla \sigma|, \tp)$ and $C' = C'(\varepsilon, s, \sm, \nu, |\sigma|)$.

Moreover, from the definition of $\Wesa$, together with \cref{lem:int:inv:bounds} and under condition \eqref{cond:mu}, it is not difficult to see that for $\sm > 0$ sufficiently small $\Wass_{\varepsilon,s,2\sm}(\mu, \mu_*) \leq C < \infty$, with $C = C(\nu, |\sigma|, \mu)$. This concludes the proof of \eqref{sup:t:Wesa:Pdisc:Pcont:0}. The final inequality \eqref{sup:E:obs:diff} is clearly a direct consequence of \cref{cor:gen:wk:conv:obs}.
\end{proof}

\begin{Rmk}\label{rmk:variance:estimator}
As mentioned in \cref{subsec:res:approx:SNSE}, a useful consequence of the Wasserstein contraction result \eqref{Wass:contr:Pdisc}, together with the long time bias estimate \eqref{ineq:bias:NSE:0} established in \cref{thm:bias:ineq:NSE}, are error estimates between the stationary average $\int \varphi(\xi' ) \mu_*( d\xi')$ and its estimator $\frac{1}{n} \sum_{k =1}^n\varphi( \xfd^k(\xi_0))$ for given $n \in \NN$, $\xi_0 \in \Pi_N \dL^2$, and suitable observable $\varphi: \dL^2 \to \RR$. Here, $\mu_*$ denotes the invariant measure of the Markov semigroup $\Pcont$, $t \geq 0$, associated to the 2D SNSE \eqref{2DSNSEv} and defined in \eqref{def:Pcont:ker}. Commonly, estimates are sought for the estimator bias
\begin{align}\label{bias:estr:stavg}
	\bE \left( \frac{1}{n} \sum_{k =1}^n\varphi( \xfd^k(\xi_0)) - \int \varphi(\xi' ) \mu_*( d\xi') \right) = \frac{1}{n} \sum_{k=1}^n P^{N,\dt}_k \varphi(\xi_0) - \int \varphi(\xi' ) \mu_*( d\xi'),
\end{align}
and the mean-squared error
\begin{align}\label{L2:error}
	\bE \left(  \frac{1}{n} \sum_{k =1}^n\varphi( \xfd^k(\xi_0)) - \int \varphi(\xi' ) \mu_*( d\xi') \right)^2.
\end{align}

Let us briefly sketch some of the steps that lead to these estimates. We assume $\varphi$ is a $\rhoesa$-Lipschitz function, with $\varepsilon$, $s$, $\alpha$ fixed so that \eqref{Wass:contr:Pdisc} holds, and denote its Lipschitz constant by $L_\varphi$.

To estimate the bias \eqref{bias:estr:stavg}, we first decompose it as
\begin{align}\label{bias:est:00}
	\left( \frac{1}{n} \sum_{k=1}^n P^{N,\dt}_k \varphi(\xi_0)  - \int \varphi(\xi' ) \mu_*^{N,\dt}( d\xi') \right)
	+ \left( \int \varphi(\xi' ) \mu_*^{N,\dt}( d\xi') - \int \varphi(\xi' ) \mu_*( d\xi') \right).
\end{align}
From the Lipschitzianity of $\varphi$ and the contraction inequality \eqref{Wass:contr:Pdisc}, the first term can be bounded as
\begin{align*}
	\left| 	\frac{1}{n} \sum_{k=1}^n P^{N,\dt}_k \varphi(\xi_0)  - \int \varphi(\xi' ) \mu_*^{N,\dt}( d\xi')  \right| 
	\leq \frac{L_\varphi}{n} \sum_{k=1}^n \Wesa(P^{N,\dt}_k(\xi_0,\cdot), \mu_*^{N,\dt}).
\end{align*}
Take $T > 0$ as in \cref{thm:sp:gap:disc:SNSE}. Then fix $0 < \dt \leq \dt_0$ and consider $k_0 \in \NN$ sufficiently large such that $k_0 \dt \geq T$. From \eqref{Wass:contr:Pdisc}, we obtain
\begin{align}\label{bias:est:0}
	&\left| 	\frac{1}{n} \sum_{k=1}^n P^{N,\dt}_k \varphi(\xi_0)  - \int \varphi(\xi' ) \mu_*^{N,\dt}( d\xi')  \right| \notag \\
	&\qquad \qquad   \leq
	\frac{L_\varphi}{n} k_0 \sup_{1 \leq k \leq k_0} \Wesa(P^{N,\dt}_k(\xi_0,\cdot), \mu_*^{N,\dt})
	+ \frac{L_\varphi}{n} \sum_{k=k_0 + 1}^n C_1 e^{-k \dt C_2} \Wesa(\delta_{\xi_0}, \mu_*^{N,\dt}).
\end{align}
From \eqref{def:rhoesa} and \eqref{Lyap:disc:1} above, one can show that
\begin{align*}
\sup_{1 \leq k \leq k_0} \Wesa(P^{N,\dt}_k(\xi_0,\cdot), \mu_*^{N,\dt}) < \infty \quad \mbox{ and also } \quad 
\Wesa(\delta_{\xi_0}, \mu_*^{N,\dt}) < \infty,
\end{align*}
both with bounds independent of $0 < \delta \leq \delta_0$
so that from \eqref{bias:est:0} we deduce
\begin{align}\label{bias:est:1}
	\left| 	\frac{1}{n} \sum_{k=1}^n P^{N,\dt}_k \varphi(\xi_0)  - \int \varphi(\xi' ) \mu_*^{N,\dt}( d\xi')  \right| 
	= O\left( \frac{1}{n\dt}\right) \quad \mbox{ as } n \to \infty.
\end{align}
Estimating the second term in \eqref{bias:est:00} as $L_\varphi \Wesa(\mu_*^{N,\dt}, \mu_*)$ and invoking 
the bias estimate \eqref{ineq:bias:NSE:0} above thus yields \eqref{ineq:bias:intro}.

Regarding the mean-squared error \eqref{L2:error}, we may proceed similarly as in e.g. \cite[Appendix]{GHM2022} (see also references therein) and write
\begin{align*}
	 \frac{1}{n} \sum_{k =1}^n\varphi( \xfd^k(\xi_0)) - \int \varphi(\xi' ) \mu_*^{N,\dt}( d\xi') 
	 = \frac{1}{n} \sum_{k=1}^\infty \left( P^{N,\dt}_k \bar{\varphi} (\xi_0) - P^{N,\dt}_k \bar{\varphi}(\xfd^n(\xi_0)) \right) 
	 + \frac{M_n^\varphi}{n} 
	 =: T_1^{(n)} + T_2^{(n)},
\end{align*}
where $\bar{\varphi}(\xi_0) := \varphi(\xi_0) - \int \varphi(\xi') \mu_*^{N,\dt}(d \xi')$, and
\begin{align*}
  M_n^\varphi
  :=& \sum_{k=1}^\infty
  \left[ \bE \left( \bar{\varphi}(\xfd^k(\xi_0)) | \mF_{n\dt} \right)
  - \bE \left( \bar{\varphi} (\xfd^k(\xi_0)) \right) \right]\\
  =& \sum_{k=1}^n \bar{\varphi} (\xfd^k(\xi_0))
     + \sum_{k =1}^\infty \left( P^{N,\dt}_{k+1} \bar{\varphi} (\xfd^n (\xi_0))
     - P^{N,\dt}_{k}\bar{\varphi}(\xi_0) \right),
\end{align*}
$n \in \NN$, is a martingale (relative to the filtration given by the
noise increments). In view of \eqref{ineq:bias:NSE:0}, it suffices to
estimate $\bE (T_1^{(n)})^2$ and $\bE (T_2^{(n)})^2$.

For the first term, we have
\begin{align*}
	\bE (T_1^{(n)})^2  
	\leq
	\frac{L_\varphi^2}{n^2} \left( \sum_{k=1}^\infty \Wesa(P^{N,\dt}_k(\xi_0,\cdot), \mu_*^{N,\dt}) + \Wesa(P^{N,\dt}_k(\xfd^n(\xi_0),\cdot), \mu_*^{N,\dt}) \right)^2,
\end{align*}
so that by proceeding analogously as in \eqref{bias:est:0}-\eqref{bias:est:1} we obtain $\bE (T_1^{(n)})^2 = O((n\dt)^{-2})$ as $n \to \infty$.

For the second term, invoking standard martingale properties it follows that 
\begin{align*}
	\bE \left( \frac{M_n^\varphi}{n} \right)^2 = \frac{1}{n^2} \sum_{k=1}^n \bE (M_k^\varphi - M_{k-1}^\varphi )^2 
	\leq 2 \left( T_{2,1}^{(n)} + T_{2,2}^{(n)} \right),
\end{align*}
where
\begin{align*}	
	T_{2,1}^{(n)} := \frac{1}{n^2} \sum_{k=1}^n \bE \left( \bar{\varphi} (\xfd^k(\xi_0)) \right)^2, \quad
	T_{2,2}^{(n)} := \frac{1}{n^2} \sum_{k=1}^n \bE \left(  \sum_{l=1}^\infty P^{N,\dt}_l \bar{\varphi} (\xfd^k (\xi_0)) - P^{N,\dt}_l \bar{\varphi} (\xfd^{k-1} (\xi_0)) \right)^2.
\end{align*}
Now for $T_{2,1}^{(n)}$ we fix $\bar{\xi} \in \dL^2$ and estimate 
\begin{align*}
	\bE \left( \bar{\varphi} (\xfd^k(\xi_0)) \right)^2 
	\leq 2 \left\{ \bE \left[ \bar{\varphi}(\xfd^k(\xi_0)) - \bar{\varphi}(\bar{\xi}) \right]^2 + \bar{\varphi}(\bar{\xi})^2 \right\} 
	\leq  2 \left\{ L_\varphi \bE \rhoesa (\xfd^k(\xi_0), \bar{\xi})^2 + \bar{\varphi}(\bar{\xi})^2 \right\}.
\end{align*}
Again from \eqref{def:rhoesa} and \eqref{Lyap:disc:1}, we obtain $\sup_k \bE \left( \bar{\varphi} (\xfd^k(\xi_0)) \right)^2 < \infty$, which yields $T_{2,1}^{(n)} \leq C/n$ for some constant $C$. 

Lastly, we bound $T_{2,2}^{(n)}$ as 
\begin{align*}
	T_{2,2}^{(n)} \leq 
	\frac{1}{n^2} \sum_{k=1}^n \bE \left( \sum_{l=1}^\infty L_\varphi \Wesa(P^{N,\dt}_l(\xfd^k(\xi_0), \cdot),P^{N,\dt}_l(\xfd^{k-1}(\xi_0), \cdot)) \right)^2.
\end{align*}
Taking $k_0$ as in \eqref{bias:est:0}, we obtain after further estimates that
\begin{align*}
	T_{2,2}^{(n)} \leq 
	\frac{C}{n} + \frac{1}{n^2} \sum_{k=1}^n \bE \left( \sum_{l=1}^\infty C_1 e^{-l\dt C_2} \rhoesa(\xfd^k(\xi_0),\xfd^{k-1}(\xi_0)) \right)^2.
\end{align*}
By estimating the difference $|\xfd^k - \xfd^{k-1}|$ according to \eqref{disc2DSNSEv} and choosing $s$ appropriately, one can show that $\rhoesa(\xfd^k(\xi_0),\xfd^{k-1}(\xi_0)) \lesssim \dt^{1/2}$, which ultimately yields $T_{2,2}^{(n)} = O((n\dt)^{-1})$ as $n \to \infty$. Such considerations together with the bias estimate \eqref{ineq:bias:NSE:0} thus imply \eqref{eq:inv:mrs:appro:intro}.
\end{Rmk}

\section{Wasserstein contraction in the case of a bounded domain}
\label{sec:cont:bnd:dom}

In this section, we apply \cref{thm:gen:sp:gap} to show Wasserstein
contraction for the Markov kernel associated to the 2D stochastic
Navier-Stokes equations on a bounded domain.  As discussed in the
introduction we include this domain example to illustrate the full
scope and significance of \cref{thm:gen:sp:gap}.  Indeed such a
suitable form of contraction does not appear follow from the approachs
taken in previous relivant works in this direction,
\cite{HairerMattingly2008, HairerMattinglyScheutzow2011,
  ButkovskyKulikScheutzow2019} as we present describe in technical
detail in \cref{rmk:s:bdd} below.

\subsection{Mathematical setting}\label{subsec:prelim:snse:bdd}

We start by briefly recalling the associated mathematical setting in
\cref{subsec:prelim:snse:bdd}. For further details, we refer to
e.g. \cite{ConstantinFoias88,FMRTbook,Temam2001,AFS2008}.  Let
$\cD \subset \RR^2$ be an open and bounded domain with smooth boundary
$\partial \cD$. Similarly as in \cref{subsubsec:2DSNSE}, we fix a
stochastic basis
$(\Omega, \mF, \{\mF_t\}_{t \geq 0}, \bP, \{W^k\}_{k=1}^d)$ equipped
with a finite family $\{W^k\}_{k=1}^d$ of standard independent
real-valued Brownian motions on $\Omega$ that are adapted to the
filtration $\{\mF_t\}_{t \geq 0}$. We then consider the following
stochastically forced 2D Navier-Stokes equations in $\cD$
\begin{align}\label{eq:SNSE:bdd}
	d\bu + \left( - \nu \Delta \bu + \bu \cdot \nabla \bu + \nabla p \right) dt = \f dt + \sum_{k=1}^d \sigma_k d W^k, \quad \nabla \cdot \bu = 0,
\end{align}
subject to the no-slip (Dirichlet) boundary condition
\begin{align}\label{no:slip:cond}
	\bu|_{\partial \cD} = 0,
\end{align}
where $\bu = \bu(\bx,t)$ and $p = p(\bx,t)$, $(\bx,t) \in \cD \times [0,\infty)$, are the unknowns, and denote the velocity vector field and the scalar pressure field, respectively; whereas $\nu > 0$ and $\f = \f(\bx)$, $\bx \in \cD$, are given and represent the kinematic viscosity parameter and a deterministic body force, respectively. Moreover, $\{\sigma_k\}_{k=1}^d$ are given vector fields in $\cD$. We assume that $\f, \sigma_1, \ldots, \sigma_d \in L^2(\cD)^2$.

Consider the following functional spaces
\begin{gather*}
	\spH = \{ \bu \in L^2(\cD)^2 \,:\, \nabla \cdot \bu = 0, \,\, \bu \cdot \mathbf{n} |_{\partial \cD} = 0\}, \\
	\spV = \{ \bu \in H^1(\cD)^2 \,:\, \nabla \cdot \bu = 0, \,\, \bu|_{\partial \cD} = 0\},
\end{gather*}
where $\mathbf{n}$ denotes the outward unit normal vector to $\partial \cD$. See e.g. \cite{ConstantinFoias88, Temam2001}. We endow $\spH$ with the standard inner product and associated norm from $L^2(\cD)^2$, which we denote as $(\cdot, \cdot)$ and $|\cdot|$, respectively. For $\spV$, we consider the inner product $\ipV{\bu}{\bv} \coloneqq (\nabla \bu, \nabla \bv)$, with associated norm $\|\bu\| \coloneqq \ipV{\bu}{\bu}^{1/2} = |\nabla \bu|$, which is well-defined due to Poincar\'e inequality \eqref{ineq:Poincare} below. We identify $\spH$ with its dual $\spH'$, so that $\spV \subseteq \spH \equiv \spH' \subseteq \spV'$, with continuous injections, where $\spV'$ denotes the dual space of $\spV$.

Denoting by $\PL$ the Leray projection of $L^2(\cD)^2$ onto $\spH$, and applying $\PL$ to \eqref{eq:SNSE:bdd} yields the following functional formulation
\begin{align}\label{SNSE:bdd:func}
	d \bu + (\nu A \bu + B(\bu, \bu)) dt = \f dt + \sigma dW,
\end{align}
where we assume without loss of generality that $\PL \f = \f$ and $\PL \sigma_k = \sigma_k$, and use the abbreviated notation $\sigma dW \coloneqq \sum_{k=1}^d \sigma_k dW^k$. Here, $A: \spV \cap H^2(\cD)^2 \to \spH$, $A \bu = - \PL \Delta \bu$, is the Stokes operator, and $B: \spV \times \spV \to \spV'$ is the bilinear mapping $B(\bu, \bv) = \PL (\bu \cdot \nabla ) \bv$. Similarly as in the periodic case, we have that $A$ is a positive and self-adjoint operator with compact inverse. Therefore, it admits a nondecreasing sequence of positive eigenvalues $\{\lambda_k\}_{k \in \NN}$ with $\lambda_k \to \infty$ as $k \to \infty$, which is associated to a sequence of eigenfunctions $\{e_k\}_{k \in \NN}$ forming an orthonormal basis of $\spH$. For each $K \in \NN$, we denote by $\Pi_K : \spH \to \spH$ the projection operator onto the subspace $\Pi_K \spH$ of $\spH$ consisting of the span of the first $K$ eigenfunctions of $A$.

We recall Poincar\'e inequality
\begin{align}\label{ineq:Poincare}
	|\bu| \leq \lambda_1^{-1/2} \|\bu\| \quad \mbox{ for all } \bu \in \spV,
\end{align}
where $\lambda_1$ denotes the smallest eigenvalue of the Stokes operator $A$. Moreover, for each $K \in \NN$ we have
\begin{align}\label{ineq:Poincare:K}
	|(I - \Pi_K) \bu| \leq \lambda_{K+1}^{-1/2} \| (I - \Pi_K) \bu\| \quad \mbox{ for all } \bu \in \spV.
\end{align}

Recall also the following property of the bilinear term 
\begin{align*}
	(B(\bu, \bv), \bw) = - (B(\bu,\bw),\bv) \quad \mbox{ for all } \bu, \bv, \bw \in \spV,
\end{align*}
which implies
\begin{align}\label{orth:prop:bdd}
	(B(\bu, \bv), \bv) = 0 \quad \mbox{ for all } \bu, \bv \in \spV.
\end{align}

We adopt similar notation from \cref{subsubsec:2DSNSE} regarding the noise term $\sigma d W$. Specifically, let $\bH$ denote the $d$-fold product of $\spH$ and define, for each $\sigma = (\sigma_1,\ldots, \sigma_d) \in \bH$, $|\sigma|^2 \coloneqq \sum_{k=1}^d |\sigma_k|^2$. We also abuse notation and see any $\sigma \in \bH$ as a mapping $\sigma: \RR^d \to \spH$, with $\sigma(w_1,\ldots,w_d) = \sum_{k=1}^d \sigma_k w_k$, and denote by $\sigma^{-1}: \range{\sigma} \to \RR^d$ its corresponding pseudo-inverse. Clearly, both $\sigma$ and $\sigma^{-1}$ are bounded operators. 

The existence and uniqueness of pathwise solutions of \eqref{eq:SNSE:bdd}-\eqref{no:slip:cond} satisfying a given initial condition follows analogously as in \cref{prop:wellposed:cont}, with appropriate modifications in the functional spaces. Namely, it holds 
by replacing $\dL^2$ and $\dH^1$ with $\spH$ and $\spV$, respectively. We thus define the associated transition function for all $\bu_0 \in \spH$ and Borel set $\cO \in \mB(\spH)$ as
\begin{align*}
	P_t(\bu_0,\cO) \coloneqq \bP (\bu(t;\bu_0) \in \cO),
\end{align*}
where $\bu(t;\bu_0)$, $t \geq 0$, is the unique pathwise solution of \eqref{eq:SNSE:bdd}-\eqref{no:slip:cond} satisfying $\bu(0;\bu_0) = \bu_0$ almost surely. The associated Feller Markov semigroup $P_t$, $t \geq 0$, is defined as
\begin{align}\label{def:Pt:snse:bdd}
	P_t \varphi(\bu_0) = \bE \varphi (\bu(t;\bu_0)), \quad \bu_0 \in \spH,
\end{align}
for every bounded and measurable function $\varphi: \spH \to \RR$.

\subsection{Wasserstein contraction estimates}
\label{sec:was:contr:bnd:dm}

We proceed to verify that assumptions \ref{A:1}-\ref{A:3} of \cref{thm:gen:sp:gap} hold in this setting. Specifically, following the notation from \cref{thm:gen:sp:gap}, we take $(\gsp,\|\cdot\|) = (\spH, |\cdot|)$, $\mI = \RR^+$, and $\{\gP_t\}_{t \in \mI}$ to be the Markov semigroup defined in \eqref{def:Pt:snse:bdd}. Moreover, we take $\gfd$ to be the class of distances 
\begin{align}\label{def:gfd:bdd}
	\gfd = \left\{ \rhoes \,:\, \varepsilon > 0, \, 0 < s \leq  c \frac{\nu^3 \lambda_1}{|\sigma|^2} \right\}
\end{align}
for some positive absolute constant $c$, with $\rhoes$ defined analogously as in \eqref{def:rhoes}, namely
\begin{align}\label{def:rhoes:bdd}
	\rhoes(\bu, \btu) = 1 \wedge \frac{|\bu - \btu|^s}{\varepsilon} \quad \mbox{ for all } \bu, \btu \in \spH.
\end{align}
We notice that in \eqref{def:gfd:bdd} we impose a different assumption on $s$ than in \cref{subsec:Wass:contr:NSE}, where $s \in (0,1]$. See \cref{rmk:s:bdd} below for more details.

In the next proposition, we show with \eqref{Lyap:ineq:bdd} that \ref{A:1} is satisfied under this setting. We also provide the energy-type inequality \eqref{energy:ineq:bdd} to be used later in the verification of \ref{A:2}-\ref{A:3}.

\begin{Prop}\label{eq:bnd:dmn:exp:mom}
	Fix any $\sigma \in \bH$ and $\bu_0 \in \spH$. Let $\bu(t)$, $t \geq 0$, be the solution of \eqref{SNSE:bdd:func} satisfying $\bu(0) = \bu_0$ almost surely. Then, for all $\sm \in \RR$ satisfying
	\begin{align}\label{cond:sm:bdd}
		0 < \sm \leq \frac{\nu \lambda_1}{4 |\sigma|^2}
	\end{align}
	the following inequalities hold:
	\begin{align}\label{energy:ineq:bdd}
		\bE \sup_{t \geq 0} \exp \left( \sm |\bu(t)|^2 + \sm \nu \int_0^t \|\bu(s)\|^2 ds - \sm t \left( |\sigma|^2 + \frac{2}{\nu \lambda_1} |\f|^2 \right) \right)
		\leq 2 \exp (\sm |\bu_0|^2),
	\end{align}
	and 
	\begin{align}\label{Lyap:ineq:bdd}
		P_t \exp(\sm |\bu_0|^2) = \bE \exp(\sm |\bu(t)|^2) \leq C \exp \left( e^{-\nu \lambda_1 t} |\bu_0|^2 \right) \quad \mbox{ for all } \, t \geq 0,
	\end{align}
	where $C = C(\nu, \lambda_1, |\f|, |\sigma|)$.
\end{Prop}
\begin{proof}
The proof of \eqref{energy:ineq:bdd} follows by applying It\^o's formula to the mapping $\bu \mapsto |\bu|^2$ and invoking standard exponential martingale arguments. We refer to \cite{HairerMattingly2008,KuksinShirikyan12,GHRichMa2017,ButkovskyKulikScheutzow2019} for further details.

The proof of \eqref{Lyap:ineq:bdd} is essentially an analogous continuous version of \cref{prop:Lyap:exp:disc}. Indeed, fixing $T > 0$, applying It\^o's formula to the mapping $(\tau,\bu) \mapsto e^{-\nu \lambda_1 (T - \tau)} |\bu(\tau)|^2$ and invoking \eqref{orth:prop:bdd}, it follows that for all $t \in [0,T]$
\begin{align}\label{Lyap:ineq:1}
	&e^{-\nu \lambda_1 (T - t)} |\bu(t)|^2 + 2 \nu \int_0^t e^{-\nu \lambda_1 (T - \tau)} \|\bu(\tau)\|^2 d\tau \notag \\
	&\qquad\qquad = e^{-\nu \lambda_1 T} |\bu_0|^2 + \nu \lambda_1 \int_0^t e^{-\nu \lambda_1 (T - \tau )} |\bu(\tau)|^2 d\tau + \int_0^t e^{-\nu \lambda_1 (T - \tau)} \left[ 2 (\bu(\tau), \f) + |\sigma|^2 \right] d\tau + M_t,
\end{align}
where $M_t \coloneqq 2 \int_0^t e^{-\nu \lambda_1 (T - \tau )} (\bu, \sigma) d W(\tau)$, $t \in [0,T]$, is a martingale. We estimate its quadratic variation $\langle M \rangle_t$ as 
\begin{align}\label{ineq:Mt:bdd}
	\langle M \rangle_t = 4 \int_0^t e^{-2 \nu \lambda_1 (T - \tau )} (\bu, \sigma)^2 d\tau 
	\leq 4 \frac{|\sigma|^2}{\lambda_1} \int_0^t e^{-\nu \lambda_1 (T - \tau )} \|\bu(\tau)\|^2 d\tau,
\end{align}
where we applied Cauchy-Schwarz and Poincar\'e inequality \eqref{ineq:Poincare}. Moreover, again from \eqref{ineq:Poincare} and Young's inequality, we have
\begin{align}\label{ineq:u:f}
	(\bu(\tau), \f) \leq \frac{1}{\lambda_1^{1/2}} \|\bu(\tau) \| \, |\f| \leq \frac{\nu}{4} \|\bu(\tau)\|^2 + \frac{1}{\nu \lambda_1} |\f|^2.
\end{align}
Now we add and subtract $\sm \langle M \rangle_t/2$ in \eqref{Lyap:ineq:1}, for $\sm$ satisfying \eqref{cond:sm:bdd}, and invoke \eqref{ineq:Poincare} once again to estimate the second term in the right-hand side of \eqref{Lyap:ineq:1}. Plugging the estimates \eqref{ineq:Mt:bdd}-\eqref{ineq:u:f}, it follows after rearranging terms that
\begin{align*}
	e^{-\nu \lambda_1 (T - t)} |\bu(t)|^2
	\leq e^{- \nu \lambda_1 T} |\bu_0|^2 + \left( \frac{2}{\nu \lambda_1} |\f|^2 + |\sigma|^2 \right) \frac{e^{-\nu \lambda_1 T} (e^{\nu \lambda_1 t} - 1)}{\nu \lambda_1} + M_t - \frac{\sm}{2} \langle M \rangle_t.
\end{align*}
Multiplying by $\sm$ and taking expected values on both sides,
\begin{align}\label{Lyap:ineq:2}
	\bE \exp& \left( \sm e^{-\nu \lambda_1 (T - t)} |\bu(t)|^2 \right)\notag\\
	&\quad \leq \exp \left( \sm e^{- \nu \lambda_1 T} |\bu_0|^2 \right) \exp \left( \sm  \left( \frac{2}{\nu \lambda_1} |\f|^2 + |\sigma|^2 \right) \frac{e^{-\nu \lambda_1 T} (e^{\nu \lambda_1 t} - 1)}{\nu \lambda_1} \right),
\end{align}
where we used that $\{N_t\}_{t \geq 0} = \{ \exp ( \sm M_t - (\sm^2/2) \langle M \rangle_t ) \}_{t \geq 0}$ is a supermartingale (see e.g. \cite[Appendix A.11]{KuksinShirikyan12}), and hence $\bE N_t \leq \bE N_0 = 1$ for all $t \geq 0$. Taking in particular $T = t$ in \eqref{Lyap:ineq:2}, we deduce \eqref{Lyap:ineq:bdd} with $C = \exp \left( \frac{\sm}{\nu \lambda_1} \left( \frac{2}{\nu \lambda_1} |\f|^2 + |\sigma|^2 \right)\right)$.
\end{proof}

To verify the remaining assumptions \ref{A:2}-\ref{A:3} of \cref{thm:gen:sp:gap}, we proceed similarly as in \cref{subsec:Wass:contr:NSE} and consider the following modified system
\begin{align}\label{eq:nudging:bdd}
	d \btu + \left[ \nu A \btu + B(\btu, \btu) + \frac{\nu \lambda_{K+1}}{2} \Pi_K \left( \btu - \bu(\bu_0) \right) \right] dt = \f dt + \sigma d W
\end{align}
for each fixed $\bu_0 \in \spH$ and corresponding pathwise solution $\bu(t; \bu_0)$, $t \geq 0$, of \eqref{SNSE:bdd:func} satisfying $\bu(0;\bu_0) = \bu_0$ almost surely. Here, $K \in \NN$ is a parameter to be appropriately chosen in \eqref{cond:K:sigma:2} below. 

With similar arguments as in \cref{prop:wellposed:cont}, we can show \eqref{eq:nudging:bdd} to be well-posed in the pathwise sense. This allows us to define, for any fixed $\bu_0 \in \spH$, the mapping
\begin{align}\label{def:tP:bdd}
	\tP_{t,\bu_0}(\btu_0, \cO) = \bP(\btu(t; \btu_0, \bu_0) \in \cO) \quad \mbox{ for all } t \geq 0, \, \btu_0 \in \spH \, \mbox{ and } \, \cO \in \mB(\spH),
\end{align}
where $\btu(t;\btu_0, \bu_0)$, $t \geq 0$, is the unique pathwise solution of \eqref{eq:nudging:bdd} satisfying $\btu(0;\btu_0, \bu_0) = \btu_0$ almost surely.

The next proposition presents a pathwise contraction estimate for the difference between a solution $\btu(\btu_0,\bu_0)$ of \eqref{eq:nudging:bdd} and the corresponding solution $\bu(\bu_0)$ of \eqref{SNSE:bdd:func}. The proof is given in \cite{GHRichMa2017,ButkovskyKulikScheutzow2019}, but we present the main ideas here for completeness.

\begin{Prop}\label{prop:FP:bdd}
	Fix any $\sigma \in \bH$, $\bu_0, \btu_0 \in \spH$, and $K \in \NN$. Let $\btu(t) = \btu(t;\btu_0,\bu_0)$, $t \geq 0$, be the solution of \eqref{eq:nudging:bdd} corresponding to this data and satisfying $\btu(0) = \btu_0$ almost surely. Then the following inequality holds almost surely for all $t \geq 0$
	\begin{align}\label{ineq:FP:bdd}
		|\btu(t;\btu_0;\bu_0) - \bu(t;\bu_0)|^2 
		\leq |\btu_0 - \bu_0|^2 \exp \left( - \nu \lambda_{K+1} t + \frac{c}{\nu} \int_0^t \|\bu(\tau ;\bu_0)\|^2 d \tau \right).
	\end{align}
\end{Prop}
\begin{proof}
Denote $\bv(t) = \btu(t;\btu_0,\bu_0) - \bu(t;\bu_0)$. Subtracting \eqref{SNSE:bdd:func} from \eqref{eq:nudging:bdd}, it follows that
\begin{align*}
	\frac{d\bv}{dt} + \nu A \bv + B(\bv, \bv) + B(\bv, \bu) + B(\bu, \bv) + \frac{\nu \lambda_{K+1}}{2} \Pi_K \bv = 0.
\end{align*}
Taking the inner product in $\spH$ with $\bv$ and invoking \eqref{orth:prop:bdd},
\begin{align*}
	\frac{1}{2} \frac{d}{dt} |\bv|^2 + \nu \|\bv\|^2 +  \frac{\nu \lambda_{K+1}}{2}  |\Pi_K \bv|^2 = - (B(\bv, \bu), \bv).
\end{align*}
By classical H\"older, interpolation and Young's inequalities, we estimate the nonlinear term as
\begin{align}\label{ineq:nonlin:term:bdd}
	|(B(\bv, \bu), \bv)| 
	\leq |\bv| \, \|\bv\| \, \|\bu\|
	\leq \frac{\nu}{2} \|\bv\|^2 + \frac{c}{\nu} \|\bu\|^2 |\bv|^2,
\end{align}
so that 
\begin{align}\label{energy:ineq}
	\frac{d}{dt} |\bv|^2 + \nu \|\bv\|^2 +  \nu \lambda_{K+1} |\Pi_K \bv|^2 \leq \frac{c}{\nu} \|\bu\|^2 |\bv|^2,
\end{align}
Moreover, from \eqref{ineq:Poincare:K},
\begin{align*}
	\nu \|\bv\|^2 = \nu \| \Pi_K \bv\|^2 + \nu \| (I - \Pi_K) \bv\|^2 \geq \nu \| \Pi_K \bv\|^2 + \nu \lambda_{K+1} |(I - \Pi_K) \bv|^2.
\end{align*}
Plugging back into \eqref{energy:ineq}, we obtain
\begin{align*}
	\frac{d}{dt}|\bv|^2 + \left( \nu \lambda_{K+1} - \frac{c}{\nu} \|\bu\|^2 \right) |\bv|^2 \leq 0,
\end{align*}
from which \eqref{ineq:FP:bdd} follows after integrating on $[0,t]$.
\end{proof}

In the next two propositions, we establish the validity of assumptions \ref{A:2} and \ref{A:3} from \cref{thm:gen:sp:gap}, with the help of the pathwise contraction \eqref{ineq:FP:bdd}. In particular, for the smallness property from \ref{A:2}, we also make use of the following estimate for the total variation distance between the laws of a Wiener process $W$ in $\RR^d$ and the corresponding shifted process $\hW$ as in \eqref{hW:W:phi}:
\begin{align}\label{ineq:tv:phi:1}
	\tv{\mL(\hW) - \mL(W)} \leq 1 - \frac{1}{6} \min \left\{ \frac{1}{8}, \exp\left[ - \left(2^{2-a} \, \bE \left( \int_0^\infty |\varphi(t)|^2 dt \right)^a \right)^{\frac{1}{a}} \right] \right\}
\end{align}
for any $a \in (0,1]$\footnote{Similarly as for \eqref{ineq:tv:phi:0},
  here we notice that in \cite[Theorem A.5,
  (A.14)]{ButkovskyKulikScheutzow2019} it is assumed instead
  $a \in (0,1)$. In fact, \eqref{ineq:tv:kl:2} and
  \eqref{ineq:kl:phi:0} above imply that \eqref{ineq:tv:phi:1} also
  holds with $a = 1$, although with an even sharper bound.}, see
\cite[Theorem A.5, (A.14)]{ButkovskyKulikScheutzow2019}.

\begin{Prop}\label{prop:small:Wass:bdd}
	Suppose there exists $K \in \NN$ and $\sigma \in \bH$ such that 
	\begin{align}\label{cond:K:sigma:2}
	\Pi_K \spH \subset \range{\sigma}\quad \mbox{ and } \quad
	\lambda_{K+1} \geq \frac{c}{\nu^3} \left( |\sigma|^2 + \frac{2}{\nu \lambda_1} |\f|^2 \right)
	\end{align}
	for some absolute constant $c > 0$.	Then, for every $M > 0$, and $\varepsilon, s$ as in \eqref{def:gfd:bdd}, there exist a time $T_1 = T_1(M, \varepsilon, s) > 0$ and a coefficient $\kappa_1 = \kappa_1(M) \in (0,1)$, which is independent of $\varepsilon, s$, for which the following inequality holds
	\begin{align*}
	\Wes(\Pcont(\bu_0, \cdot), \Pcont(\btu_0, \cdot)) \leq 1 - \kappa_1
	\end{align*}
	for all $t \geq T_1$ and for every $\bu_0, \btu_0 \in \spH$ with $|\bu_0| \leq M $ and $|\btu_0| \leq M$.
\end{Prop}
\begin{proof}
Fix $M > 0$, and $\varepsilon, s$ as in \eqref{def:gfd:bdd}. Let $\bu_0, \btu_0 \in \spH$ such that $|\bu_0| \leq M$ and $|\btu_0| \leq M$. Recalling the definition of $\Wes$ in \eqref{def:Wass:rhoe}, \eqref{def:rhoes}, and of $\tP_{t,\bu_0}(\btu_0, \cdot)$ in \eqref{def:tP:bdd}, we obtain by invoking \cref{prop:tri:ineq:Wass} that
\begin{align}\label{ineq:Wes:bdd}
	\Wes(P_t(\bu_0, \cdot), P_t(\btu_0,\cdot))
	\leq \Wes(P_t(\bu_0, \cdot), \tP_{t,\bu_0}(\btu_0, \cdot)) + 
	\Wes(\tP_{t,\bu_0}(\btu_0, \cdot),P_t(\btu_0,\cdot) ) .
\end{align}
For the first term in the right-hand side of \eqref{ineq:Wes:bdd}, we have
\begin{align}\label{ineq:Wes:bdd:0}
	\Wes(P_t(\bu_0, \cdot), \tP_{t,\bu_0}(\btu_0, \cdot)) 
	\leq \frac{1}{\varepsilon} \bE |\bu(t;\bu_0) - \btu(t;\btu_0, \bu_0)|^s,
\end{align}
so that by invoking the pathwise estimate \eqref{ineq:FP:bdd} it follows that
\begin{align}\label{ineq:Wes:bdd:1}
	\Wes(P_t(\bu_0, \cdot), \tP_{t,\bu_0}(\btu_0, \cdot)) 
	\leq  \frac{1}{\varepsilon} |\btu_0 - \bu_0|^s \exp\left( -\frac{s}{2} \nu \lambda_{K+1} t \right) \bE \exp \left( \frac{c s}{\nu} \int_0^t \|\bu(\tau)\|^2 d\tau \right).
\end{align}
Let $\sm = c s/\nu^2$. Since $0 < s \leq c \nu^3 \lambda_1/|\sigma|^2$ then $\sm$ satisfies \eqref{cond:sm:bdd}. We may thus invoke \eqref{energy:ineq:bdd} to further estimate  \eqref{ineq:Wes:bdd:1} as
\begin{align}
	&\Wes(P_t(\bu_0, \cdot), \tP_{t,\bu_0}(\btu_0, \cdot)) \notag\\
	&\qquad\qquad\qquad \leq
	\frac{c}{\varepsilon} |\btu_0 - \bu_0|^s \exp\left( -\frac{s}{2} \nu \lambda_{K+1} t \right)  \exp \left( \frac{cs}{\nu^2} |\bu_0|^2 \right) \exp \left(\frac{cs t}{\nu^2} \left( |\sigma|^2 + \frac{2}{\nu \lambda_1} |\f|^2 \right) \right) \notag \\
	&\qquad\qquad\qquad\leq \frac{c}{\varepsilon} |\btu_0 - \bu_0|^s \exp\left( -\frac{s}{4} \nu \lambda_{K+1} t \right)  \exp \left( \frac{cs}{\nu^2} |\bu_0|^2 \right) \label{est:first:term:bdd:0} \\
	&\qquad\qquad\qquad\leq \frac{c}{\varepsilon} (2 M)^s  \exp \left( \frac{cs}{\nu^2} M^2 \right) \exp\left( -\frac{s}{4} \nu \lambda_{K+1} t \right),\label{est:first:term:bdd}
\end{align}
where the second inequality follows from assumption \eqref{cond:K:sigma:2} on $K$. 

Regarding the second term in the right-hand side of \eqref{ineq:Wes:bdd}, we proceed analogously as in \eqref{est:sec:term}-\eqref{ineq:tv:hW:W} and obtain
\begin{align}\label{ineq:Wes:tv}
	\Wes(\tP_{t,\bu_0}(\btu_0, \cdot),P_t(\btu_0,\cdot) ) 
	\leq \tv{\mL(\hW) - \mL(W)},
\end{align}
where 
\begin{align*}
	\hW(t) \coloneqq W(t) - \int_0^t   \frac{\nu \lambda_{K+1}}{2}  \sigma^{-1} \Pi_K (\btu - \bu)(\tau) d \tau, \quad t \geq 0.
\end{align*}
From \eqref{ineq:tv:phi:1}, it follows that for any $a \in (0,1]$
\begin{align}\label{est:sec:term:bdd}
	&\Wes(\tP_{t,\bu_0}(\btu_0, \cdot),P_t(\btu_0,\cdot) ) \\
	&\qquad\qquad\qquad \leq  1 - \frac{1}{6} \min \left\{ \frac{1}{8}, \exp\left[ - \left(2^{2-a} \, \bE \left( \int_0^\infty \left|  \frac{\nu \lambda_{K+1}}{2}  \sigma^{-1} \Pi_K (\btu - \bu)(t) \right|^2 dt \right)^a \right)^{\frac{1}{a}} \right] \right\}.
	\notag
\end{align}
Invoking \eqref{ineq:FP:bdd} once again, we deduce
\begin{align}\label{ineq:E:phi:a:0}
	 &\bE \left( \int_0^\infty \left|  \frac{\nu \lambda_{K+1}}{2}  \sigma^{-1} \Pi_K (\btu - \bu)(t) \right|^2 dt \right)^a
	 \leq \left( \frac{\nu \lambda_{K+1}}{2} \|\sigma^{-1}\| \right)^{2a} \, \bE \left( \int_0^\infty | \btu(t) - \bu(t) |^2 dt \right)^a \notag \\
	 &\quad \leq  \left( \frac{\nu \lambda_{K+1}}{2} \|\sigma^{-1}\| \right)^{2a} \, \bE \left( \int_0^\infty |\btu_0 - \bu_0|^2 \exp \left( - \nu \lambda_{K+1} t + \frac{c}{\nu} \int_0^t \|\bu(\tau )\|^2 d \tau \right) dt \right)^a \notag \\
	 &\quad\leq  \left( \frac{\nu \lambda_{K+1}}{2} \|\sigma^{-1}\| \right)^{2a}  |\btu_0 - \bu_0|^{2a} \notag \\
	 &\qquad\qquad \cdot \bE \left[ \left( \int_0^\infty \exp \left(  - \nu \lambda_{K+1} t + \frac{c}{\nu^2} |\sigma|^2 t \right) dt \right)^a \sup_{t \geq 0} \exp \left( \frac{c a}{\nu} \int_0^t \|\bu(\tau)\|^2 d \tau - \frac{c a}{\nu^2} |\sigma|^2 t \right) \right] \notag \\
	 &\quad \leq  \left( \frac{\nu \lambda_{K+1}}{2} \|\sigma^{-1}\| \right)^{2a}  |\btu_0 - \bu_0|^{2a} \left( \int_0^\infty e^{-\frac{\nu \lambda_{K+1} t}{2}} dt \right)^a 
	 \bE \left[ \sup_{t \geq 0} \exp \left( \frac{c a}{\nu} \int_0^t \|\bu(\tau)\|^2 d \tau - \frac{c a}{\nu^2} |\sigma|^2 t \right)  \right],
\end{align}
where the last inequality follows from assumption \eqref{cond:K:sigma:2} on $K$. Now assuming $a \leq c \nu^3 \lambda_1/ |\sigma|^2$ so that we can resort to \eqref{energy:ineq:bdd}, we further obtain
\begin{align}
	\bE \left( \int_0^\infty \left|  \frac{\nu \lambda_{K+1}}{2} \sigma^{-1} \Pi_K (\btu - \bu)(t) \right|^2 dt \right)^a
	&\leq  c (\nu \lambda_{K+1})^a \|\sigma^{-1}\|^{2a}  |\btu_0 - \bu_0|^{2a} \exp \left( \frac{c a}{\nu^2} |\bu_0|^2 \right) \label{ineq:E:phi:a} \\
	&\leq c (\nu \lambda_{K+1})^a \|\sigma^{-1}\|^{2a} M^{2a} \exp \left( \frac{c a}{\nu^2} M^2 \right). \notag
\end{align}
Plugging back into \eqref{est:sec:term:bdd},
\begin{align}\label{ineq:sec:term:2}
	\Wes(\tP_{t,\bu_0}(\btu_0, \cdot),P_t(\btu_0,\cdot) ) 
	\leq 1 - \frac{1}{6} \min \left\{ \frac{1}{8}, \exp\left[ - c \nu \lambda_{K+1} \|\sigma^{-1}\|^2 M^2 \exp \left( c\frac{M^2}{\nu^2} \right)  \right] \right\}.
\end{align}
Thus, from \eqref{ineq:Wes:bdd}, \eqref{est:first:term:bdd} and \eqref{ineq:sec:term:2},
\begin{align*}
	\Wes(P_t(\bu_0, \cdot), P_t(\btu_0,\cdot))
	&\leq \frac{c}{\varepsilon} (2 M)^s  \exp \left( \frac{cs}{\nu^2} M^2 \right) \exp\left( -\frac{s}{4} \nu \lambda_{K+1} t \right) \\
	&\quad + 1 - \frac{1}{6} \min \left\{ \frac{1}{8}, \exp\left[ - c \nu \lambda_{K+1} \|\sigma^{-1}\|^2 M^2 \exp \left( c\frac{M^2}{\nu^2} \right)  \right] \right\}.
\end{align*}
Therefore, we may choose $T_1 = T_1(\nu, K, \|\sigma^{-1}\|, M, \varepsilon, s) > 0$ sufficiently large such that for all $t \geq T_1$
\begin{align*}
	\Wes(P_t(\bu_0, \cdot), P_t(\btu_0,\cdot))
	 \leq 1 - \frac{1}{12} \min \left\{ \frac{1}{8}, \exp\left[ - c \nu \lambda_{K+1} \|\sigma^{-1}\|^2 M^2 \exp \left( c\frac{M^2}{\nu^2} \right)  \right] \right\}.
\end{align*}
This concludes the proof. 
\end{proof}

\begin{Prop}\label{prop:contr:Wass:bdd}
	Suppose there exists $K \in \NN$ and $\sigma \in \bH$ such that \eqref{cond:K:sigma:2} holds. Then, for every $\kappa_2 \in (0,1)$ and for every $r > 0$ there exists $s > 0$ for which the following holds:
	\begin{itemize}
		\item[(i)] For every $\varepsilon > 0$, there exists a constant $C =  C(\varepsilon, s) > 0$ such that
		\begin{align}\label{ineq:contr:1:bdd}
		\sup_{t \geq 0} \Wes(\Pcont(\bu_0, \cdot), \Pcont(\btu_0, \cdot)) 
		\leq 
		C(\varepsilon,s) \exp\left( r |\bu_0|^2\right) \rhoes(\bu_0, \btu_0)
		\end{align}
		for every $\bu_0, \btu_0 \in \spH$ with $\rhoes(\bu_0, \btu_0) < 1$.
		\item[(ii)] There exist a parameter $\varepsilon > 0$ and a time $T_2 > 0$ such that 
		\begin{align}\label{ineq:contr:2:bdd}
		\Wes(\Pcont(\bu_0, \cdot), \Wes(\Pcont(\btu_0, \cdot))) \leq \kappa_2 \exp\left( r |\bu_0|^2 \right) \rhoes(\bu_0, \btu_0)
		\end{align}
		for all $t \geq T_2$ and for every $\bu_0, \btu_0 \in \spH$ with $\rhoes(\bu_0, \btu_0) < 1$.
	\end{itemize}
\end{Prop}
\begin{proof}
Fix any $\kappa_2 \in (0,1)$ and $r > 0$. Let $\bu_0, \btu_0 \in \spH$ such that $\rhoes(\bu_0, \btu_0) < 1$. Now choose $s$ such that
\begin{align}\label{choice:s:bdd}
	0 < s \leq c \min \left\{ r \nu^2 , \frac{\nu^3 \lambda_1}{|\sigma|^2 + \nu^3 \lambda_1} \right\},
\end{align}
for some absolute constant $c$. We estimate the Wasserstein distance $\Wes(P_t(\bu_0, \cdot), P_t(\btu_0, \cdot))$ as in \eqref{ineq:Wes:bdd}, and then the first term on the right-hand side as in \eqref{est:first:term:bdd:0}. Since $\rhoes(\bu_0, \btu_0) < 1$, then $|\btu_0 - \bu_0|^s \varepsilon^{-1} = \rhoes(\bu_0, \btu_0)$. Thus, from  \eqref{est:first:term:bdd:0},
\begin{align}\label{est:first:term}
	\Wes(P_t(\bu_0, \cdot), \tP_{t,\bu_0}(\btu_0, \cdot)) 
	\leq c \rhoes(\bu_0, \btu_0) \exp\left( -\frac{s}{4} \nu \lambda_{K+1} t \right)  \exp \left( \frac{cs}{\nu^2} |\bu_0|^2 \right).
\end{align}

For the second term, we first estimate as in \eqref{ineq:Wes:tv}, and then invoke \eqref{ineq:tv:phi:0} to obtain for any $a \in (0,1]$
\begin{align}\label{est:2nd:term:yay}
	\Wes(\tP_{t,\bu_0}(\btu_0, \cdot),P_t(\btu_0,\cdot) )
	\leq 2^{\frac{1-a}{1+a}} \left[ \bE \left( \int_0^\infty \left|  \frac{\nu \lambda_{K+1}}{2}  \sigma^{-1} \Pi_K (\btu - \bu)(t) \right|^2 dt \right)^a \right]^{\frac{1}{1+a}}.
\end{align} 
Choose $a \in (0,1]$ such that $\frac{2a}{1+a} = s$. From the choice of $s$ in \eqref{choice:s:bdd}, it follows in particular that $a \leq c \nu^3 \lambda_1/|\sigma|^2$. We may thus proceed as in \eqref{ineq:E:phi:a:0}-\eqref{ineq:E:phi:a} and obtain that
\begin{align}\label{est:sec:term:1}
	\Wes(\tP_{t,\bu_0}(\btu_0, \cdot),P_t(\btu_0,\cdot) )
	&\leq c (\nu \lambda_{K+1})^{\frac{a}{1+a}} \|\sigma^{-1}\|^{\frac{2a}{1+a}} |\btu_0 - \bu_0|^{\frac{2a}{1+a}} \exp \left( \frac{c}{\nu^2} \frac{a}{1+a} |\bu_0|^2 \right) \notag \\
	&\leq c (\nu \lambda_{K+1})^{s/2} \|\sigma^{-1}\|^s |\btu_0 - \bu_0|^s \exp(r |\bu_0|^2) \notag \\
	&\quad = c \varepsilon (\nu \lambda_{K+1})^{s/2} \|\sigma^{-1}\|^s \rhoes(\bu_0, \btu_0) \exp(r |\bu_0|^2),
\end{align}
where in the second inequality we used that $s \leq c r \nu^2$. From \eqref{ineq:Wes:bdd}, \eqref{est:first:term} and \eqref{est:sec:term:1}, we thus have
\begin{align}\label{ineq:Wes:contr:bdd}
	\Wes&(P_t(\bu_0, \cdot), P_t(\btu_0,\cdot))
 \notag\\
	&\leq c \left[ \exp\left( -\frac{s}{4} \nu \lambda_{K+1} t \right) + \varepsilon (\nu \lambda_{K+1})^{s/2} \|\sigma^{-1}\|^s \right] \exp(r |\bu_0|^2) \rhoes(\bu_0, \btu_0) 
\end{align}
for all $t \geq 0$. This shows \eqref{ineq:contr:1:bdd}.

For \eqref{ineq:contr:2:bdd}, we choose $T_2 = T_2(s, K, \kappa_2) > 0$ and $\varepsilon = \varepsilon( s, K, \|\sigma^{-1}\|, \kappa_2) > 0$ such that the expression between brackets in \eqref{ineq:Wes:contr:bdd} is less than $\kappa_2$ for all $t \geq T_2$. This concludes the proof. 
\end{proof}

From \eqref{Lyap:ineq:bdd}, \cref{prop:small:Wass:bdd} and \cref{prop:contr:Wass:bdd}, we now obtain the following Wasserstein contraction result as an immediate consequence of \cref{thm:gen:sp:gap}.

\begin{Thm}\label{thm:sp:gap:SNSE:bdd}
	Suppose there exists $K \in \NN$ and $\sigma \in \bH$ such that \eqref{cond:K:sigma:2} holds. Let $\Pcont$, $t \geq 0$, be the Markov semigroup defined in \eqref{def:Pt:snse:bdd}. Then, for every $m > 1$ there exists $\sm_m > 0$ such that for each $\sm \in (0, \sm_m]$ there exist $\varepsilon,s, T > 0$ and constants $C_1, C_2 > 0$ for which the following inequality holds
	\begin{align}\label{Wass:contr:Pcont}
	\Wesa(\mu \Pcont, \tmu \Pcont) \leq C_1 e^{-t C_2} \Wass_{\varepsilon, s, \sm/m}(\mu, \tmu)
	\end{align}
	for every $\mu, \tmu \in \Pr(\spH)$ and all $t\geq T$. Here we recall that, for every $a > 0$, $\Wass_{\varepsilon,s,a}$ denotes the Wasserstein-like extension to $\Pr(\spH)$ of the distance-like function $\rho_{\varepsilon,s,a}$ defined as in \eqref{def:rhoesa}, with $\dL^2$ replaced by $\spH$.
\end{Thm}

\begin{Rmk}\label{rmk:s:bdd}
  In the proofs of \cref{prop:small:Wass:bdd} and \cref{prop:contr:Wass:bdd}, namely in pursuit of the
  conditions \ref{A:2} and \ref{A:3} in \cref{thm:gen:sp:gap}, we will resort to
  the pathwise estimate \eqref{ineq:FP:bdd}.  It is precisely in
  making use of this challenging form of the `Foias-Prodi' estimate
  where \cref{thm:gen:sp:gap} improves upon the previous
  formulations of the weak Harris theorem in \cite{HairerMattingly2008,
    HairerMattinglyScheutzow2011, ButkovskyKulikScheutzow2019}.
   In effect, it does not seem possible to establish
   a suitable contraction bound \`a la \eqref{Wass:contr:Pcont} for the system
\eqref{eq:SNSE:bdd}, \eqref{no:slip:cond} with an obvious or direct application of these previous results.
  
  As throughout the extant SPDE literature, \eqref{ineq:FP:bdd}
  represents a crucial structural property of the dynamics which leads
  to the type of `irreducibility' and `smoothing' conditions embodied
  in \ref{A:2}, \ref{A:3}, respectively. However, the form that
  \eqref{ineq:FP:bdd} takes here illustrates a paradigmatic challenge
  in regards to establishing a suitable form of contractivity in the Markovian dynamic,
  one which has not been fully addressed previously in the literature
  as far as we can tell.  This is due to the coefficient in front of
  the integral term in \eqref{ineq:FP:bdd}, which is expected to be large
  in general. In turn, the size of this coefficient presents
  difficulties in terms of the available exponential moments on the
  law of the solution; namely the quadratic exponential moment bound \eqref{energy:ineq:bdd} in \cref{eq:bnd:dmn:exp:mom} degenerates as $O(\nu^2)$ while the demands of the 
  integral term in \eqref{ineq:FP:bdd} increases as $O(\nu^{-1})$ for
  small $\nu$.  Here it is notable that a similar issue does not occur
  in the case of periodic boundary conditions from
  \cref{sec:app:SNSE}, thanks to better properties regarding the
  nonlinear term in \eqref{2DSNSEv} in this periodic case,
  cf. \eqref{ineq:nonlin:term:per} and \eqref{ineq:nonlin:term:bdd}.
  On the other hand, analogously difficult (or even more difficult) forms of
  \eqref{ineq:FP:bdd} appear in other models considered in
  e.g. \cite{GHRichMa2017,ButkovskyKulikScheutzow2019,glatt2021long}.
	
  This issue makes itself evident in establishing smoothing estimates
  for $P_t$ as follows.  In the weak Harris approach developed in
  \cite{HairerMattinglyScheutzow2011} and in the subsequent literature,
  the appropriate smoothing condition is the $\rho$-contractivity
  condition. In contrast to \ref{A:3} in our formulation from
  \cref{thm:gen:sp:gap}, the $\rho$-contractivity is assumed to hold
  uniformly over the phase space. To be specific,
  \cite{HairerMattinglyScheutzow2011} requires that the bounded
  distance $\rho$ which they use to eventually build their contraction
  distance (in a fashion closely analogous to \eqref{def:gdista})
  maintains, for some $\alpha \in (0,1)$,
  \begin{align}\label{eq:glob:cont:cond}
    \Wass_\rho( P_{t_*}(\bu_0, \cdot), P_{t_*}(\bv_0, \cdot))
    \leq \alpha \rho(\bu_0, \bv_0),
    \quad \text{ whenever } \rho(\bu_0, \bv_0) < 1.
  \end{align}
  We observe that, even in the absence of boundaries, using a distance
  of the form $\rho(\bu, \bv) = \rhoes (\bu,\bv)$ given in
  \eqref{def:rhoes:bdd} for appropriately tuned $s, \epsilon > 0$ only
  leads to a local form of this contraction estimate
  \eqref{eq:glob:cont:cond}; cf. \eqref{ineq:contr:Wass} above.  To
  circumvent this difficulty, it is suggested in \cite[Proposition
  5.4]{HairerMattinglyScheutzow2011} that one use a certain geodesic
  distance developed in \cite[Section 4]{HairerMattingly2008} adapted
  to the `Lyapunov' structure in the available form
  $V(\bu) = \exp(\alpha |\bu|^2)$.  There they show that a time asymptotic
  smoothing estimate for the gradient estimate on the Markovian
  semigroup then provides the necessary global form of
  \eqref{eq:glob:cont:cond}. This is an infinitesimal 
  approach in the sense that we are trying to bound the distance between
  $P_{t_*}(\bu_0, \cdot)$ and $P_{t_*}(\bv_0, \cdot)$
  with  $\nabla P_{t_*}(\bu_0, \cdot)(\bu_0 - \bv_0)$.
  
  In our context, we could try to repeat this strategy from
  \cite[Section 4]{HairerMattingly2008} and \cite[Proposition
  5.4]{HairerMattinglyScheutzow2011} as follows.  Observe that, for
  any $C^1$ observable $\phi$ and any $\xi \in H$, we have from
  \eqref{def:Pt:snse:bdd}
  \begin{align}\label{eq:grad:approch}
  	\nabla P_t\phi(\bu_0)\xi =
	\bE \left(\nabla \phi(\bu(t; \bu_0) \bv\right)
    +
    \bE\left(  \phi(\bu(t; \bu_0)
    \medint\int_0^t
    \left( \sigma^{-1} \frac{\nu \lambda_{K+1}}{2} \Pi_K \bv \right)
    \cdot dW \right),
  \end{align}
  where
  \begin{align*}
	\frac{d \bv}{dt} + \nu A \bv + B(\bu, \bv) + B(\bv, \bu)
	+ \frac{\nu \lambda_{K+1}}{2} \Pi_K \bv = 0,
	\quad \bv(0) = \xi.
  \end{align*}
  The identity \eqref{eq:grad:approch} follows by adding and
  subtracting the gradient of $\bu(t;\bu_0)$ in its noise variable,
  taken in the direction
  $\bw = \sigma^{-1} \frac{\nu \lambda_{K+1}}{2} \Pi_K \bv$. One then
  performs a Malliavin integration by parts -- really just an
  application of the Girsanov theorem here since $\bw$ is adapted -- to
  obtain the second term in \eqref{eq:grad:approch}.  One may now view
  \eqref{eq:grad:approch} as means of estimating the possibility of
  coupling of two nearby solutions in two terms analogous to
  \eqref{ineq:Wes:bdd}. 
  
  This analogy is precisely what is
  operationalized in the proof of \cite[Proposition
  5.4]{HairerMattinglyScheutzow2011}.  To follow this approach, one
  would treat the first term as
  \begin{align}\label{first:term}
    (P_t |\nabla \phi|^2(\bu_0))^{1/2} (\bE |\bv|^2)^{1/2},
  \end{align}
  which we compare to our bound \eqref{est:first:term}.  The second
  term is estimated as
  \begin{align}\label{second:term}
    \sup_{\bu}|\phi(\bu)| \left(
    \bE \int_0^t \left|\sigma^{-1} \frac{\nu \lambda_{K+1}}{2} \Pi_K \bv \right|^2 dt
    \right)^{1/2},
  \end{align}
  which we may compare to \eqref{est:2nd:term:yay}.  However, in our
  setting neither of these terms can be shown to be finite. It is not clear how to introduce a suitable localization to the above arguments to avoid these issues with moments.

  Here we notice that choosing $s$ sufficiently small, specifically as
  in \eqref{def:gfd:bdd}, is what allows us to make use of the
  exponential moment bound \eqref{energy:ineq:bdd} to proceed with the estimates in \eqref{ineq:Wes:bdd:1}, \eqref{est:sec:term:1}.  
  However these estimates in themselves are insufficient 
  as they do not lead to the global form \eqref{eq:glob:cont:cond}.  On the other hand, \cite{ButkovskyKulikScheutzow2019} analogously
  employs the pseudo-metric 
  \begin{align*}
      \rho(\bu, \bv) = 1 \wedge \left(\frac{| \bu - \bv|^s}{\varepsilon} \exp(\alpha|\bu|^2)\right)
       \wedge \left(\frac{| \bu - \bv|^s}{\varepsilon} \exp(\alpha|\bv|^2)\right)
  \end{align*}
  to achieve \eqref{eq:glob:cont:cond}.  While this approach
  from \cite{ButkovskyKulikScheutzow2019} would lead to a contraction in a related  pseudo-metric 
  as a direct consequence of \cite[Theorem 4.8]{HairerMattinglyScheutzow2011},
  it is not clear that this pseudo-metric satisfies 
  any usable form of the generalized triangle inequality. Obviously, having such a generalized triangle inequality is indispensable for establishing continuous parameter dependence in the long time statistics of certain stochastic systems using the strategies we overviewed in \cref{subsec:unif:contr:overview} and used throughout \cref{sec:gen:wk:conv} and \cref{sec:app:SNSE}.
\end{Rmk}

\setcounter{equation}{0}
\appendix
\addcontentsline{toc}{section}{Appendix}
\addtocontents{toc}{\protect\setcounter{tocdepth}{0}}

\section{Distance-like functions}

Here we present some simple general results concerning distance-like functions on a Polish space $\gsp$ and their corresponding Wasserstein-like extensions to $\Pr(\gsp)$, as recalled in \cref{subsec:prelim}. 

We start by showing that if a given distance-like function $\gdist$ satisfies a generalized form of triangle inequality, namely \eqref{gen:tri:ineq:rho} below, together with a suitable set of conditions, then the corresponding distance-like function $\gdist_\sm$ for a fixed parameter $\sm > 0$, defined in \eqref{def:gdista}, satisfies an inequality of the form \eqref{gen:tri:ineq:gdista} from assumption \ref{H:1} in \cref{thm:gen:wk:conv}. The proof below follows similar ideas from \cite[Lemma 4.14]{HairerMattinglyScheutzow2011}.

\begin{Prop}\label{prop:gen:tri:ineq:rhoesa}
Let $(\gsp, \|\cdot \|)$ be a Banach space and let $\rho: \gsp \times \gsp \to \RR^+$ be a distance-like function on $\gsp$ satisfying the following conditions:
\begin{enumerate}[label={(\roman*)}]
	\item\label{rho:i} $\rho$ is bounded, i.e. there exists a constant $M > 0$ such that $\rho(u, v) \leq M$ for all $u, v \in \gsp$.
	\item\label{rho:ii} There exists a constant $K > 0$ such that
	\begin{align}\label{gen:tri:ineq:rho}
		\rho(u, v) \leq K \left[ \rho(u, w) + \rho(w, v) \right] \quad \mbox{ for all } u, v, w \in \gsp.
	\end{align}
	\item\label{rho:iii} There exists a constant $c > 0$ for which the following holds: if $\rho(u, v) < c$ for some $u, v \in \gsp$, then $\| u\|^2 \leq \gamma \|v\|^2 + C$ for some constants $\gamma > 1$ and $C > 0$, which are independent of $u$ and $v$.
\end{enumerate}
Then, for the distance-like function $\rho_\sm: \gsp \times \gsp \to \RR^+$ defined for a fixed parameter $\sm > 0$ by
\begin{align*}
	\rho_\sm(u, v) = \rho(u, v)^{1/2} \exp \left(\sm \|u\|^2 + \sm \|v\|^2 \right)  \quad \mbox{ for all } u, v \in \gsp,
\end{align*}
it follows that there exists a constant $\tilde{K} > 0$ such that
\begin{align}\label{gen:tri:ineq:rhosm}
	\rho_\sm(u, v) \leq \tilde{K} \left[ \rho_{\gamma \sm}(u, w) + \rho_{\gamma \sm}(w, v) \right] \quad \mbox{ for all } u, v, w \in \gsp,
\end{align}
where $\gamma > 0$ is the constant from assumption $(iii)$.
\end{Prop}
\begin{proof}
Let $u, v, w \in \gsp$. Since, for any $\sm > 0$, $\rho_\sm$ is symmetric, we may assume without loss of generality that $\|v\| \leq (\gamma - 1)^{1/2} \|u\|$, with $\gamma > 1$ being the constant from assumption \ref{rho:iii}. 

First, suppose that $\rho(u, w) \geq c$. Then, by invoking assumption \ref{rho:i} we obtain that
\begin{align}\label{ineq:rhosm:1}
	\rho_\sm(u, v) \leq M^{1/2} \exp \left(\sm \|u\|^2 + \sm \|v\|^2 \right)
	&\leq M^{1/2} \frac{\rho(u, w)^{1/2}}{c^{1/2}} \exp \left(\sm \|u\|^2 + \sm (\gamma - 1) \|u\|^2 \right) \notag \\
	&\leq \frac{M^{1/2}}{c^{1/2}} \rho_{\gamma \sm} (u, w) 
	\leq \frac{M^{1/2}}{c^{1/2}} \left[ \rho_{\gamma \sm} (u, w) + \rho_{\gamma \sm}(w, v) \right].
\end{align}
On the other hand, if $\rho(u, w) < c$ then by invoking assumptions \ref{rho:ii} and \ref{rho:iii} it follows that
\begin{align}\label{ineq:rhosm:2}
	\rho_\sm(u, v)
	&\leq K^{1/2}  \left[ \rho(u, w)^{1/2} + \rho(w, v)^{1/2} \right] \exp \left(\sm \|u\|^2 + \sm \|v\|^2 \right) \notag \\
	&\leq K^{1/2} \left[ \rho(u, w)^{1/2} \exp \left(\gamma \sm \|u\|^2  \right)  + \rho(w, v)^{1/2} \exp \left( \sm \gamma \|w\|^2 + \sm C + \sm \|v\|^2 \right) \right] \notag \\
	&\leq \tilde{C} \left[ \rho_{\gamma \sm} (u, w) + \rho_{\gamma \sm} (w, v) \right],
\end{align}
where $\tilde{C} = K^{1/2} \exp(\sm C)$. 

From \eqref{ineq:rhosm:1} and \eqref{ineq:rhosm:2}, we conclude that \eqref{gen:tri:ineq:rhosm} holds with $\tilde{K} =\max\{ (M/c)^{1/2}, \tilde{C} \}$.
\end{proof}

In the following result, we show that a generalized triangle inequality satisfied by given distance-like functions, namely \eqref{gen:trg:ineq:rho} below, induces an analogous inequality for the corresponding Wasserstein-like extensions, \eqref{tri:ineq:Wass}. The proof relies essentially on the Disintegration theorem (see e.g. \cite[Lemma 5.3.2]{AmbrosioGigliSavare2005}): fixed measures $\mu, \nu, \tilde{\nu} \in \Pr(\gsp)$, and given any couplings $\Gamma \in \Co(\mu, \tilde{\nu})$, $\Gamma' \in \Co(\tilde{\nu},\nu)$, it provides a way of constructing a coupling $\Gamma'' \in \Co(\mu,\nu)$, so that one can pass from \eqref{gen:trg:ineq:rho} to \eqref{tri:ineq:Wass}. Before we state the result, let us recall a few definitions. 

Let $(\mX, \Sigma_{\mX} )$ and $(\mY, \Sigma_{\mY})$ be measurable spaces. Given a measurable function $\phi: \mX \to \mY$ and a measure $\mu \in \Pr(\mX)$, the \emph{pushforward of $\mu$ by $\phi$}, denoted by $\phi^* \mu$, is defined as the measure on $\mY$ given by
\begin{align*}
	\phi^*\mu(A) \coloneqq \mu (\phi^{-1}(A)) \quad \mbox{ for any } A \in \Sigma_{\mX},
\end{align*}
where $\phi^{-1}(A)$ denotes the preimage of the set $A$ by $\phi$. Moreover, given a $(\phi^*\mu)$-integrable function $\psi: \mY \to \RR$, it follows that the composition $\psi \circ \phi: \mX \to \RR$ is $\mu$-integrable and the following change of variables formula holds
\begin{align}\label{change:var}
	\int_{\mY} \psi(u) (\phi^*\mu)(d u) = \int_{\mX} \psi(\phi(u)) \mu(du).
\end{align}

\begin{Prop}\label{prop:tri:ineq:Wass}
 Let $\gsp$ be a Polish space. Suppose there exist distance-like functions $\gdist_1, \gdist_2, \gdist_3: \gsp \times \gsp \to \RR^+$ for which there exists a constant $C > 0$ such that
 \begin{align}\label{gen:trg:ineq:rho}
 	\gdist_1(u, v) \leq C \left[ \gdist_2(u, w) + \gdist_3(w, v) \right] \quad \mbox{ for all } u, v, w \in \gsp.
 \end{align}
 Let $\Wass_{\gdist_1}$, $\Wass_{\gdist_2}$ and $\Wass_{\gdist_3}$ be the Wasserstein-like extensions of $\gdist_1$, $\gdist_2$ and $\gdist_3$, respectively, to $\Pr(\gsp)$, according to the definition given in \eqref{def:Wass:rhoe}. Then,
 \begin{align}\label{tri:ineq:Wass}
 	\Wass_{\gdist_1}(\mu, \mu') \leq C \left[ \Wass_{\gdist_2}(\mu, \tmu)  +  \Wass_{\gdist_3}(\tmu, \mu') \right] \quad \mbox{ for all } \mu, \mu', \tmu \in \Pr(\gsp).
 \end{align}
\end{Prop}
\begin{proof}
Let $\pi_1, \pi_2: \gsp\times \gsp \to \gsp$ denote the projection functions onto the first and second components, respectively. Namely, $\pi_1(u, v) = u$ and $\pi_2(u, v) = v$ for all $u, v \in \gsp$. Then, recalling the definition of the family of couplings $\Co(\mu, \mu')$ of any two measures $\mu, \mu' \in \Pr(\gsp)$, given in \cref{subsec:prelim}, it follows that, for any $\Gamma \in \Co(\mu, \mu') \subset \Pr(\gsp \times \gsp)$, $\pi_1^* \Gamma = \mu$ and $\pi_2^* \Gamma = \mu'$. 

Fix $\mu, \mu', \tmu \in \Pr(\gsp)$. From \eqref{def:Wass:rhoe}, it follows that for any given $\tilde{\varepsilon} > 0$ there exist $\Gamma \in \Co(\mu, \tmu)$ and $\Gamma' \in \Co(\tmu, \mu')$ such that
\begin{align*}
	\int_{\gsp \times \gsp} \gdist_2(u, v) \Gamma (du, d v) < \Wass_{\gdist_2}(\mu, \tmu) + \tilde{\varepsilon},
\end{align*}
and
\begin{align*}
	\int_{\gsp \times \gsp} \gdist_3(u, v) \Gamma' (du, d v) < \Wass_{\gdist_3}(\tmu, \mu') + \tilde{\varepsilon}.
\end{align*}

Further, let us denote by $\pi_{i,j}: \gsp \times \gsp \times \gsp \to \gsp\times \gsp$, $i,j= 1,2,3$, the projection functions 
\begin{align*}
	\pi_{i,j}(u_1, u_2, u_3) = (u_i, u_j), \quad \mbox{ for all } u_1, u_2, u_3 \in \gsp.
\end{align*}
Since $\pi_2^* \Gamma = \tmu = \pi_1^* \Gamma'$, it follows from the Disintegration theorem (see e.g. \cite[Lemma 5.3.2]{AmbrosioGigliSavare2005}) that there exists $\tilde{\Gamma} \in \Pr(\gsp \times \gsp \times \gsp)$ such that $\pi_{1,2}^\ast \tilde{\Gamma} = \Gamma$ and $\pi_{2,3}^\ast \tilde{\Gamma} = \Gamma'$. Consequently, $\pi_{1,3}^\ast \tilde{\Gamma} \in \Co(\mu, \mu')$ and 
\begin{align}\label{Wass:rho1}
	\Wass_{\gdist_1}(\mu, \mu') \leq \int_{\gsp \times \gsp} \gdist_1 (u, v) \pi_{1,3}^\ast \tilde{\Gamma} (d u, dv)
	= \int_{\gsp \times \gsp \times \gsp} \gdist_1(\pi_{1,3}(u, w, v)) \tilde{\Gamma}(du, d w, d v).
\end{align}
From assumption \eqref{gen:trg:ineq:rho}, we have that for any $u, v, w \in \gsp$
\begin{align}\label{g:trg:ineq:rho:pi}
	\gdist_1(\pi_{1,3}(u, w, v)) = \gdist_1(u, v) 
	&\leq C \left[ \gdist_2(u, w) + \gdist_3 (w, v) \right] \notag\\
	&= C \left[ \gdist_2(\pi_{1,2}(u, w, v)) + \gdist_3 (\pi_{2,3}(u, w, v)) \right].
\end{align}
Plugging \eqref{g:trg:ineq:rho:pi} in \eqref{Wass:rho1} and changing variables as in \eqref{change:var}, we deduce that
\begin{align}\label{Wass:rho1:b}
	\Wass_{\gdist_1}(\mu, \mu') 
	&\leq C  \left[ \int_{\gsp \times \gsp} \gdist_2(u, w) \pi_{1,2}^* \tilde{\Gamma}(du, d w) + \int_{\gsp \times \gsp} \gdist_3(w, v) \pi_{2,3}^* \tilde{\Gamma}(dw, d v)  \right] \notag \\
	&= C  \left[ \int_{\gsp \times \gsp} \gdist_2(u, w) \Gamma (du, d w) + \int_{\gsp \times \gsp} \gdist_3(w, v) \Gamma' (dw, d v)  \right] \notag \\
	&< C \left[  \Wass_{\gdist_2}(\mu, \tmu) +  \Wass_{\gdist_3}(\tmu, \mu') + 2 \tilde{\varepsilon} \right].
\end{align}
Since $\tilde{\varepsilon} > 0$ is arbitrary, taking the limit as $\tilde{\varepsilon}$ goes to $0$ in \eqref{Wass:rho1:b} we conclude \eqref{tri:ineq:Wass}.
\end{proof}

\section{Proof of \cref{thm:holder:reg}}\label{sec:app:holder:reg}

With the same notation from \eqref{snse:func}, we write the Galerkin system \eqref{2DSNSEvGal} in the following functional form 
\begin{align}\label{eq:SNSE:func}
	\rd \xn + \left[ \nu A \xn + \Pi_N B(\xn, \xn)  \right] \rd t = \Pi_N \sigma \rd W.
\end{align}

The following preliminary lemma provides some suitable bounds for the analytic semigroup $e^{-\nu t A}$, $t \geq 0$, generated by the operator $- \nu A$. For the proof, we refer to \cite[Theorem 6.13, Chapter 2]{Pazy2012}. The notation $\| \cdot \|_{\mL(\dL^2)}$ below refers to the standard operator norm of a linear operator on $\dL^2$.

\begin{Lem}\label{lem:prop:heat:ker}
	For every $a \geq 0$ and $b \in (0,1]$, there exist constants $c_a > 0$ and $c_b > 0$ such that 
	\begin{gather}
	\| A^a e^{-\nu t A}\|_{\mL(\dL^2)} \leq c_a (\nu t)^{-a}, \\
	\| A^{-b} (I - e^{-\nu t A})\|_{\mL(\dL^2)} \leq c_b (\nu t)^b,
	\end{gather}
	for all $t > 0$.
\end{Lem}

Having fixed the necessary terminology, we proceed to show the desired H\"older regularity for solutions of the Galerkin system \eqref{2DSNSEvGal}.

\begin{proof}[Proof of \cref{thm:holder:reg}]
We only show a proof of inequality \eqref{holder:reg:H1}, since the proof of \eqref{holder:reg:L2} is simpler and follows entirely analogously. Fix $T > 0$, $m \in \NN$ and $\tp \in (0,1/2)$. 
We consider the mild form of the solution $\xn$ that follows from the functional formulation \eqref{eq:SNSE:func}, namely
\begin{align*}
\xn(t) =  e^{- \nu t A} \Pi_N \xi_0 - \int_0^t e^{-\nu (t - \tau) A} \Pi_N B(\xn, \xn ) \rd \tau + \int_0^t e^{-\nu (t - \tau) A} \Pi_N \sigma \rd W(\tau),
\end{align*}
for every $t \geq 0$. Thus, for every $s,t \in [0,T]$,
\begin{align}\label{eq:diff:gradxi}
\nabla \xn(t) - \nabla \xn (s)
=&
\left( e^{-\nu t A}  -  e^{-\nu s A} \right) \nabla \Pi_N \xi_0 
\nonumber\\
&- \left(  \int_0^t e^{-\nu(t - \tau) A} \nabla \Pi_N B(\xn, \xn) \rd \tau  -   \int_0^s  e^{-\nu(s - \tau) A} \nabla \Pi_N B(\xn, \xn) \rd \tau  \right)
\nonumber\\
&+ \left( \int_0^t e^{-\nu (t - \tau) A} \nabla \Pi_N\sigma \rd W(\tau)  
- \int_0^s e^{-\nu (s - \tau) A} \nabla \Pi_N\sigma \rd W(\tau) \right)
\nonumber\\
=& \, (I) + (II)  + (III).
\end{align}

We proceed to estimate each term in the right-hand side of \eqref{eq:diff:gradxi}. Without loss of generality, let us assume $s < t$.
We estimate $(I)$ as
\begin{align*}
|(I)| = |\left( e^{-\nu t A}  -  e^{-\nu s A} \right) \nabla \Pi_N \xi_0 | 
&=  |e^{-\nu s A} ( e^{- \nu (t - s) A} - I ) \nabla \Pi_N \xi_0 | 
\\
&\leq  \| e^{-\nu s A}\|_{\mL(\dL^2)} \|A^{- \tp} ( e^{- \nu (t - s) A} - I )\|_{\mL(\dL^2)} |A^\tp \nabla \xi_0|
\\
&\leq c| t - s |^{\tp} |A \xi_0|,
\end{align*}
where the last inequality follows from \cref{lem:prop:heat:ker}, and the fact that $A^\tp \nabla \xi_0 = A^\tp A^{1/2} \xi_0 = A^{\tp + 1/2} \xi_0$, so that since $\tp \in (0,1/2)$ we have $|A^\tp \nabla \xi_0| = \|\xi_0\|_{\dH^{2\tp + 1}} \leq \|\xi_0\|_{\dH^2}$, see \cref{subsubsec:2DSNSE}. Hence,
\begin{align}\label{ineq:I}
\bE |(I)|^m \leq c |t-s|^{m \tp} |A \xi_0|^m.
\end{align}

Now for term $(II)$ we have
\begin{align*}
|(II)| =&
\left| \int_0^t e^{-\nu(t - \tau) A} A^{1/2} \Pi_N B(\xn, \xn) \rd \tau  -   \int_0^s  e^{-\nu(s - \tau) A} A^{1/2} \Pi_N B(\xn, \xn) \rd \tau \right| \notag \\
\leq&
\left| \int_0^s (  e^{-\nu(t - \tau) A }  - e^{-\nu(s - \tau) A} ) A^{1/2} \Pi_N B(\xn, \xn) \rd \tau  \right|
+  \left|  \int_s^t e^{-\nu (t - \tau) A } A^{1/2} \Pi_N B(\xn, \xn) \rd \tau  \right| 
\\
&= |(II_a)|  +  |(II_b)|.
\end{align*}

Notice that 
\begin{align}\label{ineq:term:IIa}
\bE |(II_a)|^m
&= \bE \left|  \int_0^s e^{-\nu (s - \tau) A} (e^{-\nu(t - s) A} - I) A^{1/2} \Pi_N B(\xn, \xn) \rd \tau  \right|^m 
\nonumber \\
&\leq \bE \left( \int_0^s \| A^{\tp + 1/2} e^{-\nu (s - \tau) A} \|_{\mL(\dL^2)} \|A^{-\tp } (e^{-\nu(t - s) A} - I)\|_{\mL(\dL^2)}  \, |\Pi_N B(\xn, \xn)| \rd \tau \right)^m
\nonumber \\
&\leq \frac{c}{\nu^{m/2}} \bE \sup_{0 \leq \tau \leq T} |\Pi_N B(\xn, \xn)|^m \left(  \int_0^s  |s - \tau|^{-\tp - 1/2} |t - s|^{\tp}  \rd \tau \right)^m,
\end{align}
where in the last inequality we invoked \cref{lem:prop:heat:ker} once again. 

With inequality \eqref{ineq:nonlin:b} for the nonlinear term and estimate \eqref{ineq:sup:nabla:xn} from \cref{prop:sup:xn:nabla}, it follows that 
\begin{align*}
\frac{c}{\nu^{m/2}} \bE \sup_{0 \leq \tau \leq T} |\Pi_N B(\xn, \xn)|^m 
\leq \frac{c}{\nu^{m/2}} \bE \sup_{0 \leq \tau \leq T} |B(\xn, \xn)|^m 
&\leq \frac{c}{\nu^{m/2}} \bE \sup_{0 \leq \tau \leq T} |\nabla \xn|^{2m} \\
&\leq C(1 + |\xi_0|^{4m} + |\nabla \xi_0|^{2m}),
\end{align*}
where $C$ is a constant depending on $m, T, \nu, |\sigma|$ and $|\nabla \sigma|$. Thus, from \eqref{ineq:term:IIa} and since $\tp \in (0,1/2)$
\begin{align}\label{ineq:IIa}
\bE |(II_a)|^m &\leq C (1 + |\xi_0|^{4m} + |\nabla \xi_0|^{2m}) |t - s|^{m\tp} \left(  \int_0^s |s - \tau|^{-\tp - 1/2} \rd \tau \right)^m 
\notag\\
&\leq C (1 + |\xi_0|^{4m} + |\nabla \xi_0|^{2m}) |t - s|^{m\tp}.
\end{align}

Similarly, we have for $(II_b)$ that
\begin{align}\label{ineq:IIb}
\bE |(II_b)|^m &\leq \bE \left(  \int_s^t \|A^{1/2} e^{-\nu(t - \tau)A}\|_{\mL(\dL^2)} |\Pi_N B(\xn, \xn)| \rd \tau \right)^m 
\notag\\
&\leq \frac{c}{\nu^{m/2}} \bE \sup_{0 \leq \tau \leq T} |\Pi_N B(\xn, \xn)|^m  \left( \int_s^t |t - \tau|^{-1/2} \rd \tau  \right)^m
\notag\\
&\leq C (1 + |\xi_0|^{4m} + |\nabla \xi_0|^{2m}) |t - s|^{m/2} 
\notag\\
&\leq C (1 +|\xi_0|^{4m} + |\nabla \xi_0|^{2m}) |t - s|^{m \tp } T^{(-\tp + 1/2)m}.
\end{align}

Lastly, we estimate $(III)$ as
\begin{align*}
|(III)| &\leq 
\left| \int_0^s (e^{-\nu (t - \tau) A}  -  e^{- \nu (s - \tau) A } ) \nabla \Pi_N \sigma \rd W(\tau) \right|
+ 
\left| \int_s^t e^{-\nu (t - \tau) A} \nabla \Pi_N \sigma \rd W (\tau)  \right| \\
&\quad = |(III_a)| + |(III_b)|.
\end{align*}

For each fixed $s,t \in [0,T]$, we define for every $r \in [0,s]$
\begin{align*}
	M_r := \int_0^r (e^{-\nu (t - \tau) A}  -  e^{- \nu (s - \tau) A } ) \nabla \Pi_N\sigma \rd W(\tau). 
\end{align*}
Then, $\{M_r\}_{0 \leq r \leq s}$ is a martingale. By Burkholder-Davis-Gundy inequality \cite[Theorem 3.28]{KaratzasShreve1991}, for every $p \in (0,\infty)$,
\begin{align*}
\bE |M_s|^p \leq \bE \sup_{0 \leq r \leq s} |M_r|^p \leq c \, \bE \left( \langle M \rangle_s^{p/2}\right),
\end{align*}
where
\begin{align*}
\langle M \rangle_s = \int_0^s | (e^{-\nu (t - \tau) A}  -  e^{- \nu (s - \tau) A }) \nabla \Pi_N\sigma |^2 \rd \tau.
\end{align*}
Hence, invoking \cref{lem:prop:heat:ker} again,
\begin{align}\label{ineq:IIIa}
\bE |(III_a)|^m = \bE |M_s|^m
&\leq c \left(   \int_0^s | (e^{-\nu (t - \tau) A}  -  e^{- \nu (s - \tau) A }) \nabla \Pi_N\sigma |^2 \rd \tau  \right)^{m/2}
\notag\\
&\leq c \left(  \int_0^s |e^{-\nu (s - \tau) A} ( e^{-\nu (t - s ) A} - I) \nabla \Pi_N \sigma|^2 \rd \tau  \right)^{m/2}
\notag\\
&\leq c \left( \int_0^s \| A^\tp e^{-\nu(s - \tau) A }\|_{\mL(\dL^2)}^2  \| A^{-\tp} (e^{-\nu (t - s ) A} - I )\|_{\mL(\dL^2)}^2 |\nabla \sigma |^2 \rd \tau  \right)^{m/2}
\notag\\
&\leq c \left( \int_0^s |s- \tau|^{-2 \tp} |t -s|^{2 \tp} |\nabla \sigma|^2 \rd \tau  \right)^{m/2} 
\notag\\
&= c |t -s |^{m \tp} |\nabla \sigma|^m \left( \int_0^s |s - \tau|^{-2 \tp} \rd \tau  \right)^{m/2}
\notag\\
&\leq c |t - s|^{m \tp} |\nabla \sigma|^m s^{(- \tp + 1/2)m}
\leq c |t - s|^{m \tp} |\nabla \sigma|^m T^{ (- \tp + 1/2)m},
\end{align}
where the last inequality holds thanks to the assumption that $\tp \in (0,1/2)$.

Analogously, we estimate $\bE |(III_b)|^m$ as 
\begin{align}\label{ineq:IIIb}
\bE |(III_b)|^m 
\leq \left( \int_s^t  |e^{-\nu(t - \tau) A} \nabla \Pi_N\sigma|^2 \rd \tau \right)^{m/2}
&\leq \left( \int_s^t \| e^{-\nu(t - \tau) A}\|_{\mL(\dL^2)}^2 |\nabla \sigma|^2 \rd \tau  \right)^{m/2}
\notag\\
&\leq |\nabla \sigma|^m |t - s|^{m/2} 
\notag\\
&\leq |t - s|^{m \tp} \, |\nabla \sigma|^m  \, T^{(- \tp + 1/2)m}
\end{align}

Therefore, it follows from \eqref{eq:diff:gradxi} and the estimates \eqref{ineq:I}, \eqref{ineq:IIa}-\eqref{ineq:IIIb} above that for all $s,t \in [0,T]$ with $s \leq t$
\begin{align}\label{holder:reg:a}
\bE |\nabla \xn(t) - \nabla \xn(s)|^m 
\leq C |t - s|^{m\tp} \left[ 1  + |\xi_0|^{4m} + |\nabla \xi_0|^{2m} +  |A \xi_0|^m \right],
\end{align}
where $C = C(m , \tp, T, \nu, |\sigma|, |\nabla \sigma|)$. This concludes the proof of \eqref{holder:reg:L2}. Clearly, by following similar steps as above one can show that \eqref{holder:reg:L2} and \eqref{holder:reg:H1} also hold with $\xn(t)$ replaced by the solution $\xi(t)$, $t \geq 0$, of \eqref{2DSNSEv} satisfying $\xi(0) = \xi_0$ almost surely.
\end{proof}

\section*{Acknowledgements}

Our efforts are supported under the grants DMS-1816551, DMS-2108790 (NEGH) and
DMS-2009859 (CFM). 

\addtocontents{toc}{\protect\setcounter{tocdepth}{1}}
\addcontentsline{toc}{section}{References}
\bibliographystyle{alpha}
\bibliography{bib}

\vfill

\begin{multicols}{2}
\noindent
Nathan E. Glatt-Holtz\\ {\footnotesize
Department of Mathematics\\
Tulane University\\
Web: \url{http://www.math.tulane.edu/~negh/}\\
Email: \url{negh@tulane.edu}} \\[.2cm]
\columnbreak

\noindent 
Cecilia F. Mondaini \\
{\footnotesize
  Department of Mathematics\\
  Drexel University\\
  Web: \url{https://www.math.drexel.edu/~cf823/}\\
  Email: \url{cf823@drexel.edu}}\\[.2cm]
 \end{multicols}

\end{document}